\documentclass{amsart}

\usepackage[margin=1in]{geometry}
\usepackage{amsthm,amsmath,amsfonts,amssymb,mathabx}
\usepackage{graphicx}
\usepackage{dsfont}
\usepackage[usenames,dvipsnames,svgnames,table]{xcolor}

\usepackage[colorlinks=true, pdfstartview=FitV,linkcolor=ForestGreen,citecolor=ForestGreen, urlcolor=blue]{hyperref}

\usepackage{hyperref}
\usepackage{esint}

\newtheorem{thm}{Theorem}[section]
\newtheorem{prop}[thm]{Proposition}
\newtheorem{cor}[thm]{Corollary}
\newtheorem{ass}[thm]{Assumption}
\newtheorem{lemma}[thm]{Lemma}
\newtheorem{defn}[thm]{Definition}
\theoremstyle{remark}
\newtheorem{remark}[thm]{Remark}
\numberwithin{equation}{section}

\newcommand{\R}{\mathbb R}
\newcommand{\N}{\mathbb N}

\newcommand{\eps}{\varepsilon}
\newcommand{\dd} {\, \mathrm{d}}

\newcommand{\dv}{\, \mathrm{d} v}
\newcommand{\dw}{\, \mathrm{d} w}
\newcommand{\du}{\, \mathrm{d} u}
\newcommand{\dz}{\, \mathrm{d} z}
\newcommand{\da}{d_{\mathrm{a}}}
\newcommand{\dgs}{d_{\mathrm{GS}}}
\newcommand{\one}{\mathds{1}}

\newcommand{\To} {\mathcal T_0}

\newcommand{\Q}{\mathcal Q}

\newcommand{\un}{\one}
\newcommand{\Cpol}{C_{\ell,\rm{fast}}}
\newcommand{\fin}{f_{\rm{in}}}
\newcommand{\LL}{\mathcal{L}}
\newcommand{\Qint}{Q_{\rm{int}}}

\newcommand{\cone}{\Xi}

\DeclareMathOperator*{\osc}{osc}

\DeclareMathOperator{\PV}{PV}

\DeclareMathOperator{\degk}{\deg_{\mathrm k}}

\title{Global regularity estimates for the Boltzmann equation without cut-off}

\author{Cyril Imbert} 
\address[C.~Imbert]{CNRS \& DMA, École normale supérieure, Université PSL, CNRS, 75005 Paris, France \\
45 rue d'Ulm, 75005 Paris, France}
\email{Cyril.Imbert@ens.psl.eu}
\author{Luis Silvestre}
\thanks{The authors would like to thank C.~Mouhot for interesting discussions regarding this work, and also for collaboration in some of the earlier work that was useful for this paper. LS is supported by NSF grant DMS-1764285.}
\address[L.~Silvestre]{Mathematics Department, University of Chicago,
  Chicago, Illinois 60637, USA} \email{luis@math.uchicago.edu}

\date{\today}

\begin{document}

\begin{abstract}
  We derive $C^\infty$ a priori estimates for solutions of the inhomogeneous Boltzmann equation without cut-off, conditional to pointwise bounds on their mass, energy and entropy densities. We also establish decay estimates for large velocities, for all derivatives of the solution. 
\end{abstract}

\maketitle

\setcounter{tocdepth}{1}
\tableofcontents

\section{Introduction}

The Boltzmann equation is a fundamental nonlinear evolution model from statistical mechanics. It describes the evolution of a
system made of a very large number of particles at an intermediate scale between the microscopic one (which consists of the trajectory of every single particle and their interactions) and the macroscopic one (the hydrodynamic models like Euler or Navier-Stockes equations).

We consider the space in-homogeneous Boltzmann equation without cut-off,
\begin{equation}\label{e:boltzmann}
\partial_t f + v \cdot \nabla_x f = \Q(f,f)  \quad \text{ for } \; (t,x,v) \in (0,T) \times \R^d \times \R^d.
\end{equation}
Boltzmann's collision operator $\Q(f,f)$ is typically written in the following way
\begin{equation}
  \label{e:Q}
  \Q(f,f) = \int_{\R^d} \int_{S^{d-1}} (f(v'_*)f(v') - f(v_*)f(v)) B(|v-v_*|,\cos \theta) \dd v_* \dd \sigma
\end{equation}
where $v_*'$ and $v'$ are computed in terms of $v_\ast$ and $\sigma$ by the formula
\[ v' = \frac{v+v_*}2 + \frac{|v-v_*|}2 \sigma \quad \text{ and } \quad 
v'_* = \frac{v+v_*}2 - \frac{|v-v_*|}2 \sigma .\]
The angle $\theta$ measures the deviation between $v$ and $v'$. In this case, it is the angle so that
\[ \cos \theta := \frac{v-v_*}{|v-v_*|} \cdot \sigma 
\qquad \left( \text{ and } \quad \sin (\theta/2) := \frac{v'-v}{|v'-v|} \cdot \sigma \right) .\]
We consider the standard non-cutoff collision kernels $B$. They have the form
\begin{equation}\label{assum:B}
B(r,\cos \theta) = r^\gamma b(\cos \theta) \quad \text{ with} \quad b (\cos \theta) 
\approx|\sin (\theta/2)|^{-(d-1)-2s}
\end{equation}
with $\gamma > -d$ and $s \in (0,1)$. 

In a microscopic model where the particles repel each other by an inverse-power law potential with exponent $q>2$, the collision kernel has the form \eqref{assum:B} with $\gamma = (q-2d+1)/(q-1)$ and $s = 1/(q-1)$ (See for example \cite{villani-book}, chapter 1, Section 1.4). In three dimensions, for inverse-power law potentials, the value of $\gamma+2s$ would be in the range $[-1,1]$. Our results in this paper apply to the range $\gamma+2s \in [0,2]$. In Subsection~\ref{s:open_problems}, we briefly discuss the problem with the \emph{very soft potential} case: $\gamma+2s<0$.

We define the hydrodynamic quantities
\[
\begin{aligned}
M (t,x) &:= \int_{\R^d}f(t,x,v) \dd v && \text{(mass density)}, \\
E (t,x) & := \int_{\R^d}f(t,x,v) |v|^2 \dd v && \text{(energy density)}, \\
H (t,x) & := \int_{\R^d} f \ln f (t,x,v) \dd v && \text{(entropy density)}. 
\end{aligned}
\]
These hydrodynamic quantities, together with moment density, are the 
quantities associated with the solution of the Boltzmann equation
that are macroscopically observable.

In this article, we are concerned with regularity estimates for the solution of \eqref{e:boltzmann}. This is intimately related with the well posed-ness problem for smooth classical solutions. The question of existence of global smooth solutions for the Boltzmann equation \eqref{e:boltzmann} is a well known and remarkable open problem. There is a warm discussion about it in the first chapter of C\'edric Villani's book \cite{cedric2012theoreme}. The Boltzmann equation is a more detailed model for the evolution of a fluid than the hydrodynamic models like Euler or Navier-Stokes equations. Indeed, in certain asymptotic regime (see \cite{bardos1991}), the hydrodynamic quantities associated to the Boltzmann equation converge to the solution of the compressible Euler equation, that is known to develop singularities in finite time \cite{sideris}. A next order expansion shows that the hydrodynamic quantities approximately solve a compressible Navier-Stokes equation, for which the classical well-posedness problem is not well understood. It makes sense to expect the Boltzmann equation to retain the difficulties of the hydrodynamic models, and add some more. Should we expect singularity formation in finite time then? The answer to this question is not straight forward. There are different types of singularities that emerge from the flow of the compressible Euler equation. Some of them may be compatible with the Boltzmann equation, and others are not. In a \emph{shock} singularity for the Euler equation, a discontinuity emerges from the flow similarly as in Burgers equation. All the quantities involved stay bounded up to the time of the discontinuity. One would not see this as a singularity for the Boltzmann equation, since the kinetic model allows for different velocities to co-exist at one point in space. The Navier-Stokes equation will not allow for shock singularities either, since the viscosity would smooth out any discontinuity for as long as solutions stay bounded and away from vacuum. A fundamentally different, and much more delicate, kind of singularity is that of an \emph{implosion}. In that case, the mass and energy concentrate and become unbounded at one point. It was only very recently (in fact, after this paper was initially posted) that smooth implosion profiles for the compressible Euler equation were found and proved to be stable in \cite{merle2019smooth} and \cite{merle2019implosion}. These implosion singularities exist for the compressible Navier-Stokes equation as well. As of now, we cannot think of any reason to rule out the existence of implossion singularities for the Boltzmann equation. If they did, which seems like a likely scenario, the question that remains is whether this is the only type of singularity that may emerge from the flow of the Boltzmann equation. Our main result in this paper aims at answering that question.

As we explained in the previous paragraph, the unconditional regularity of solutions to the in-homogeneous Boltzmann equation seems to be completely out of reach. The problem that we study is conditional to pointwise bounds on the hydrodynamic quantities. More precisely, we make the following assumption.
\begin{ass}[Hydrodynamic quantities under control] \label{a:hydro-assumption}
The following inequalities hold uniformly in $t$ and $x$,
\begin{itemize}
\item $0 < m_0 \leq M(t,x) \leq M_0$.
\item $E(t,x) \leq E_0$.
\item $H(t,x) \leq H_0$.
\end{itemize}
\end{ass}

We do not prove Assumption~\ref{a:hydro-assumption}. We take it for granted (hence the name \emph{assumption}).  Conditional to it, we obtain $C^\infty$ estimates that we state in our main theorem. Assumption~\ref{a:hydro-assumption} is a way to disallow the implossion singularities that we described above. Our result in this paper essentially says that no other types of singularity may exist for the Boltzmann equation other than (potentially) the ones that are hydrodynamically visible.

\begin{thm}[Global regularity estimates] \label{t:main2} 
Let $f$ be a solution to the Boltzmann equation in $(0,T) \times \R^d \times \R^d$ (as in Definition~\ref{d:solution}) with a collision kernel of the form \eqref{assum:B} and $\gamma+2s \in [0,2]$. If Assumption~\ref{a:hydro-assumption} holds, then for any multi-index $k \in \mathbb N^{1+2d}$, $\tau>0$ and $q>0$,
\[ \| (1+|v|)^q D^k f \|_{L^\infty([\tau,T) \times \R^d \times \R^d)} \leq C_{k,q,\tau}.\]
Here $D^k$ is any arbitrary derivative of $f$ of any order, in $t$, $x$ and/or $v$.

When $\gamma>0$, the constants $C_{k,q,\tau}$ depend only on $k$, $q$ and $\tau$, and the constants $m_0$, $M_0$, $E_0$ and $H_0$ from Assumption~\ref{a:hydro-assumption}, and the parameters $s$, $\gamma$ and dimension $d$.

When $\gamma\leq 0$, the constants $C_{k,q,\tau}$ depend in addition on the pointwise decay  of the initial data. That is, on the constants $N_r$ with $r \ge 0$, given by
\begin{equation} \label{e:initial_decay_assumption}
 N_r := \sup_{x,v} (1+|v|)^r f_0(x,v) \qquad \text{for each } r\geq 0.
\end{equation}
\end{thm}

Note that the upper bounds on energy density and mass density in \eqref{a:hydro-assumption} correspond to upper bounds for mass, moment and temperature density. Moreover, the upper bound in entropy is slightly stronger than a lower bound for temperature (and equivalent in the hydrodynamic limit).

We work with a strong notion of solution that we describe in Definition~\ref{d:solution}. See section~\ref{s:renormalized} for a discussion about weaker notions of solutions. Moreover, we work with functions $f$ that are periodic in $x$. It is a convenient, but non-essential, technical assumption. It is only applied for the upper bounds in \cite{imbert2018decay}, that require the use of a maximum principle argument (as in ``let $(t_0,x_0,v_0)$ be the point where this maximum is achieved''). Theorem~\ref{t:main2} would hold in any other regime where these upper bounds hold. The estimates in Theorem~\ref{t:main2} do not depend on the length of the period. The problem of regularity estimates as in Theorem~\ref{t:main2} in bounded domain with physical boundary conditions would require some further analysis.

\begin{remark}
The difference between $\gamma>0$ and $\gamma \leq 0$ in Theorem~\ref{t:main2} has its origin in the decay estimates from \cite{imbert2018decay}. The decay of the solution $f$ is self generated when $\gamma>0$. However, when $\gamma \leq 0$, the function $f$ will decay rapidly only if it initially does.
\end{remark}

\begin{remark}
In the case $\gamma \leq 0$, each constant $C_{k, q, \tau}$ depends on one constant $N_r$ in \eqref{e:initial_decay_assumption} for a specific value of $r$ depending on $k$ and $q$. However, its explicit dependence is hard to track. Obviously, the larger $q$, the larger the value of $r$ will be required to be. It turns out that for higher order derivatives $D^k$, we also need to use larger values of $r$.
\end{remark}

\subsection{Consequences of our main theorem}

\subsubsection{Convergence to equilibrium}

In a celebrated result \cite{desvillettes-villani}, Desvillettes and Villani proved that solutions to the non-cutoff in-homogeneous Boltzmann equation, periodic in $x$ (or with other physical boundary conditions), converge to equilibrium faster than any algebraic rate, conditional to the following two main assumptions
\begin{enumerate}
\item The solution $f$ stays in $C^\infty$ for all time with uniform bounds as $t \to \infty$.
\item The solution $f$ is bounded below by some fixed Maxwellian.
\end{enumerate}

A priori, these two assumptions appeared to be very strong. After Theorem~\ref{t:main2}, they can be reduced to only Assumption~\ref{a:hydro-assumption}. Indeed, the lower bound by a fixed Maxwellian is obtained in our earlier work with Cl\'ement Mouhot \cite{imbert2019gaussian}. 

Since the estimates in Theorem~\ref{t:main2} do not depend on $T$, we can take $T \to \infty$ and deduce a uniform regularity estimate in $(\tau,\infty] \times \R^d \times \R^d$. As a consequence, we state the following improvement for the Theorem in \cite{desvillettes-villani}.

\begin{cor}
Let $f$ be a solution of \eqref{e:boltzmann} in $(0,\infty) \times \R^d \times \R^d$ (as in Definition~\ref{d:solution}, in particular periodic in $x$). Assume that Assumption~\ref{a:hydro-assumption} holds globally. Then $f$ converges to a Maxwellian as $t \to \infty$ as described in Theorem 2 in \cite{desvillettes-villani}.
\end{cor}

\subsubsection{Continuation criteria}
\label{s:continuation}

Theorem~\ref{t:main2} also suggests the following continuation criteria. Let $f$ be a solution to the Boltzmann equation \eqref{e:boltzmann} in $(0,T) \times \R^d \times \R^d$ as in Definition~\ref{d:solution}. Suppose that it cannot be extended further in time, that is, it cannot be extended as a solution in $(0,T+\eps) \times \R^d \times \R^d$ for any $\eps>0$. Then, one of the following must happen
\begin{enumerate}
\item $\lim_{t \to T} \max_{x \in \R^d} M(t,x) = +\infty$.
\item $\lim_{t \to T} \max_{x \in \R^d} E(t,x) = +\infty$.
\item $\lim_{t \to T} \max_{x \in \R^d} H(t,x) = +\infty$.
\item $\lim_{t \to T} \min_{x \in \R^d} M(t,x) = 0$.
\end{enumerate}

This continuation criteria can be immediately justified by combining Theorem~\ref{t:main2} with an appropriate short time existence result. When $s \in (0,1/2)$ and $\gamma \in (-3/2,0]$, we can use the short time existence from \cite{morimoto2015local}. For any $s \in (0,1)$ and $\gamma \leq 0$, there is a recent appropriate short time existence result in \cite{HST1}.

Note that the short time existence result in \cite{amuxy2010} requires the initial data to have Gaussian decay, which is not propagated to positive times by our estimates in Theorem~\ref{t:main2}.

This continuation criteria says that the only way a singularity can arise in finite time for the Boltzmann equation without cutoff is by one of the hydrodynamic quantities $M$, $E$ or $H$ to blow up, or by creation of vacuum. There is a recent result in \cite{HST2} saying that this continuation criteria can be reduced to the first two items. That is, either the mass or the energy density should blow up if the solution develops a singularity.  It rules out the case in which there is creation of vacuum or zero temperature while the mass and energy density stay bounded. It is conceivable that this blow up criteria may be relaxed in some other way in the future. As we explained above, a completely unconditional continuation criteria seems to be out of reach with current techniques.

It is natural to expect a similar continuation criteria to hold in the cut-off case as well. However, the reason for it would be fundamentally different. The cut-off Boltzmann equation does not have a regularization effect. One would expect a propagation of regularity provided that Assumption~\ref{a:hydro-assumption} holds. From the mathematical point of view, it is a very different problem from the one we address here. We will not analyze the cut-off case any further.

\subsection{Future directions and open problems}

\subsubsection{Regularity estimates for weak solutions}
\label{s:renormalized}

In this paper we obtain a priori estimates for classical solutions. Working with a weaker notion of solution would entail several technical difficulties. We thought it was not the right time to take on that burden yet. In fact, we consider a very strong notion of solution (see Definition~\ref{d:solution}).  It would be interesting to extend Theorem~\ref{t:main2} as a regularity estimate for renormalized solutions with a defect measure as defined in \cite{av2002}. Below, we analyze the difficulties of this problem.

The biggest challenge of such an extension would be to recover the pointwise estimates from \cite{luis} and \cite{imbert2018decay}. The proofs in these papers use a maximum principle type argument that seems to be difficult to adapt to the setting of \cite{av2002}.

Once a weak solution is proved to be bounded, we can apply the result in \cite{imbert2016weak} (Theorem~\ref{t:local-holder} below) and deduce the H\"older continuity of the solution.

There is a (presumably minor) difficulty in the application of the Schauder estimates from \cite{schauder} in order to derive Corollary~\ref{c:schauder_boltzmann} in this paper. This is because the result of \cite{schauder} is not stated for weak solutions. The later applications of the Schauder estimates in our proof of Theorem~\ref{t:main2} in Section~\ref{s:smoothing} are not problematic. In each step of the iteration we apply the Schauder estimates to increments that are qualitatively as regular as the function itself.

\subsubsection{The grazing collision limit}

When $s \to 1$, the Boltzmann equation converges formally to the Landau equation. For that, we need the collision kernel $B$ to satisfy
\begin{equation} \label{e:normalized_B}
 B \approx (1-s) |v-v_\ast|^\gamma \sin(\theta/2)^{-d+1-2s}.
\end{equation}
The normalizing factor $(1-s)$ is transferred into the ellipticity conditions on the Boltzmann kernel $K_f$ (defined in \eqref{e:kf}). It is well known in the literature of nonlocal equations that this is the necessary factor to have uniform bounds as $s \to 1$ (see for example \cite{caffarelli2009regularity} or \cite{russell}).

It is to be expected that the estimates of Theorem~\ref{t:main2} would remain uniform as $s \to 1$ if $B$ satisfies \eqref{e:normalized_B}. However, it is still an open problem. Below, we explain the difficulties with our current approach.

Note that any technique that establishes the estimates from Theorem~\ref{t:main2} uniformly as $s \to 1$, would also imply the corresponding regularity estimates for the Landau equation as a consequence. A method that provides estimates uniform as $s \to 1$ must use techniques that apply both to integro-differential equations and second order parabolic equations.

The most challenging difficulty in proving uniform estimates as $s \to 1$ would be to establish the pointwise bounds from \cite{luis} and \cite{imbert2018decay}. The proofs in these papers use purely nonlocal techniques. The constants obtained in the estimates there certainly blow up as $s \to 1$. The corresponding pointwise upper bound for the Landau equation is established in \cite{cameron2017global} using different methods.

The H\"older estimates from \cite{imbert2016weak} are robust as $s \to 1$. We would also expect the Schauder estimates from \cite{schauder} to be robust as $s \to 1$, however it does not follow directly from the current proof in \cite{schauder} because it is non constructive. Some constants are proved to exist under a compactness argument, and by that we loose track of their dependence on $s$. It is conceivable that a refinement of the proof in \cite{schauder} may lead to robust estimates since the proof in that paper works for second order equations as well.

\subsubsection{Other open problems}
\label{s:open_problems}

The following is a (non exhaustive) list of other open questions related to the main result of this article.
\begin{enumerate}
\item Interior estimates: if Assumption~\ref{a:hydro-assumption} holds for $(t,x) \in (-1,0] \times B_1$, can we establish the regularity estimates as in Theorem~\ref{t:main2} for $(t,x) \in (-1/2,0]\times B_{1/2}$?
\item Bounded domains: when the equation is supported in a smooth bounded domain $x \in \Omega \subset \R^d$, with physical boundary conditions, do the estimates from Theorem~\ref{t:main2} hold in the full domain $\Omega$?
\item Weaker conditions: can we reduce Assumption~\ref{a:hydro-assumption} to a weaker condition?
\item Very soft potentials: can we establish regularity estimates when $\gamma+2s < 0$? This is a very difficult problem that is open even in the space homogeneous setting. The most challenging step seems to be obtaining the $L^\infty$ estimate as in \cite{luis}.
\end{enumerate}

\subsection{Previous regularity results for the Boltzmann and Landau equations}.

The well-posedness and regularity of the space homogeneous Boltzmann equation is well understood in the case of hard and moderately soft potentials (i.e. $\gamma+2s \geq 0$). See \cite{dw2004,aes2005,aes2009,huo2008,morimoto2009regularity,chen2011smoothing}. Note that in the space homogeneous case, Assumption~\ref{a:hydro-assumption} is trivially satisfied by the conservation of mass and energy and the monotonicity of entropy.

Results on the regularity for the space in-homogeneous Boltzmann equation are scarce. Other than the papers that are part of our program, the most relevant previous result is the $C^\infty$ regularity of solutions conditional to a uniform bound in $H^5_{x,v}$, plus infinite bounded moments, plus a lower bound on the mass density. These results were established in \cite{amuxyCRAS2009,amuxy2010,chen-he-2012}. We improve these results by significantly lowering the condition for regularity to the bounds of Assumption~\ref{a:hydro-assumption}, which are physically meaningful. We refer to   \cite[{\S 1.3.2}]{imbert2016weak} for further discussion on other results in the literature.

Our program of establishing conditional regularity provided that the hydrodynamic quantities are controlled as in Assumption
\ref{a:hydro-assumption} has also been studied for the Landau equation. It is currently fairly well understood in the cases of hard and moderately soft potentials. The local H\"older estimates were obtained in \cite{golse2016harnack}. The upper bounds and Gaussian decay bounds (when appropriate) for moderately soft potentials were obtained in \cite{cameron2017global} using the estimates from \cite{golse2016harnack} combined with a change of variables that inspires our construction in Section~\ref{s:cov}. The higher regularity of solutions was studied in \cite{henderson2017c} applying a kinetic version of Schauder estimates. These regularity estimates were extended to the case of hard potentials in \cite{snelson2018inhomogeneous}. In \cite{henderson2017local}, they refine the continuation criteria for the in-homogeneous Landau equation as mentioned in Subsection~\ref{s:continuation}.

\subsection{Strategy of proof}
The result in this paper is obtained as the final step in a program of conditional regularity that started in 2014. Here, we use the previous results by the authors, and also by the authors in collaboration with Cl\'ement Mouhot, that were part of this program. Theorem~\ref{t:main2} is proved by combining the following ingredients.

\begin{itemize}
\item An $L^\infty$ estimate for positive time depending only of the hydrodynamic quantities. This holds provided $\gamma+2s \geq 0$. It is proved in \cite{luis}.
\item A weak Harnack inequality for kinetic integro-differential equations. It gives us local $C^\alpha$ estimates for some $\alpha>0$ small. It was obtained in \cite{imbert2016weak}.
\item Schauder estimates for kinetic integro-differential equations. It gives us local $C^{2s+\alpha}$ estimates. It was obtained in \cite{schauder}.
\item Pointwise decay estimates. They say $f \lesssim (1+|v|)^{-q}$ for all $q>0$.
They are self generated if $\gamma > 0$ and they propagate from the initial data if $\gamma \leq 0$.
It was proved in \cite{imbert2018decay}.
\item A change of variables that turns our local H\"older and Schauder estimates into global ones. We develop it in Section~\ref{s:cov}.
\item Some new inequalities for kinetic H\"older spaces (defined in Section~\ref{s:holder}) and how they interact with the Boltzmann collision operator (see Section~\ref{s:bilinear}) and increments (see Section~\ref{s:incremental_quotients}).
\item A bootstrapping mechanism by iterating the global version of the Schauder estimates.
\end{itemize}

\medskip

In order to obtain regularity estimates like the ones in this article, it is key to think of the Boltzmann equation as a kinetic equation with integral diffusion in the $v$ variable plus a lower order term, in the way that was described in \cite{luis}. Using Carleman coordinates and the cancellation lemma (as in \cite{alexandre2000entropy}), the Boltzmann equation takes the following form
\begin{equation} \label{e:Boltzmann_IDE}
 \partial_t f + v \cdot \nabla_x f = \int_{\R^d} (f' - f) K_f(t,x,v,v') \dd v' + c (f \ast_v |\cdot|^\gamma) f.
\end{equation}

The kernel $K_f$ depends on the solution $f$ itself. We give more details in Section~\ref{s:boltzmann} and recall the formula for $K_f$ in \eqref{e:kf}. When Assumption~\ref{a:hydro-assumption} holds, the kernel $K_f$ satisfies certain ellipticity conditions that allow us to derive regularity estimates.

In \cite{imbert2016weak}, we obtained a weak Harnack inequality for kinetic integro-differential equations. It implies a regularity estimate for the local H\"older regularity, for a small exponent, of bounded solutions to \eqref{e:boltzmann} that satisfy Assumption~\ref{a:hydro-assumption}. In \cite{schauder}, we obtained a Schauder estimate for kinetic integro-differential equations. It implies a local estimate of H\"older regularity with exponent $2s+\alpha$ for some $\alpha>0$. It is enough regularity to make sense of the equation classically. These are two results for generic kinetic integro-differential equations. They apply to the solution of the Boltzmann equation thanks to the expression \eqref{e:Boltzmann_IDE}. They also apply to derivatives of $f$ with respect to $t$, $x$ and $v$ provided that we can appropriately bound each of the extra error terms that come up in the equation when differentiating the collision and transport terms. In order to turn this procedure into a bootstrap iteration leading to $C^\infty$ estimates we need to turn the local regularity estimates from \cite{imbert2016weak} and \cite{schauder} into global ones.

The weak Harnack inequality in \cite{imbert2016weak} and the Schauder estimate in \cite{schauder} depend on ellipticity conditions on the kernel $K_f$ in \eqref{e:Boltzmann_IDE}. In these papers, we showed how these ellipticity conditions are implied locally by Assumption~\ref{a:hydro-assumption}. However, they degenerate for large velocities. In order to obtain global estimates from the application of the weak Harnack inequality and Schauder estimates, we construct a special change of variables. It transforms the function $f$ into a solution to a kinetic integro-differential equation whose kernel is uniformly elliptic with parameters that do not degenerate for large velocities. This change of variables is a key ingredient in the proof of this paper. It is described in Section~\ref{s:cov}. It allows us to turn any local (in velocity) estimate for the Boltzmann equation into a global one.

The bootstrap iteration consists in applying the global version of Schauder estimates (via the change of variables) to the equation satisfied by derivatives of the solution $f$. The extra error terms are handled by careful (and new) estimates in H\"older spaces for the Boltzmann collision operator, increments and derivatives (described in Sections~\ref{s:bilinear} and \ref{s:incremental_quotients}). In each step of the iteration, we gain a regularity estimate for a higher derivative in terms of the estimates already obtained for the lower order ones. There is a loss of decay exponent in each differentiation step. Thus, we need to start with a function with rapid decay at infinity in order for the iteration to continue indefinitely. This initial decay is provided by the result in \cite{imbert2018decay}.

\subsection{Notation}

We use the notation $a \lesssim b$ to denote the fact that there exists a constant $C$ so that $a \leq Cb$. The constant $C$ can depend on a variable collection of parameters depending on context. This notation is used mostly inside proofs of lemmas, propositions and theorems. In each statement, we explain what the constants depend on. The implicit constants in each symbol $\lesssim$ inside a proof depend on the parameters specified in the corresponding statement. Even though this notation might be arguably confusing at times, it allows us to clean up the computations substantially.

The symbol $a \approx b$ means that $a \lesssim b$ and $a \gtrsim b$.

We use the standard notation $B_r$ to denote a ball of radius $r$ in $\R^d$ centered at the origin. We also write $B_r(w)$ to denote a ball centered at some point $w \in \R^d$. The kinetic cylinders $Q_r \subset \R \times \R^d \times \R^d$ are explained in Section~\ref{s:cylinders}.

\section{Preliminary estimates for the Boltzmann equation}
\label{s:preliminary}

In this section, we collect some preliminary results for the Boltzmann equation that play a role in the proofs in this paper.

As we mentioned before, we work with a very strong notion of classical solutions. In this way, all the results in the literature are applicable and we avoid technical difficulties. We give the definition below.

\begin{defn} \label{d:solution}
A function $f : (0,T) \times \R^d \times \R^d \to \R$ is a \emph{solution} to the Boltzmann equation \eqref{e:boltzmann} when 
\begin{itemize}
\item It is non-negative everywhere.
\item It is $C^\infty$ in all variables $t,x,v$.
\item It solves \eqref{e:boltzmann} for every value of $(t,x,v)$ in the classical sense.
\item For each value of $(t,x)$, the function $f$ decays rapidly as $|v|\to \infty$. More precisely, for any $q>0$, we have
\[ \lim_{|v|\to \infty} \frac{f(t,x,v)}{(1+|v|)^q} = 0,\]
locall uniformly in $(t,x)$.
For derivatives of $f$, we only assume that there is sufficient decay so that
\[ \int_{\R^d} |D^2_vf(t,x,v)| (1+|v|)^{\gamma+2s} \dd v < +\infty,\]
for every value of $t$ and $x$.
\end{itemize}
For simplicity, we will also consider $f$ to be periodic in $x$.
\end{defn}

The results in this paper certainly apply to a weaker notions of solution as well. We discuss it in Section \ref{s:renormalized}. However, by considering a strong notion of solution as in Definition \ref{d:solution}, we avoid superflous technical difficulties that would make this paper harder to read.

In the last section of \cite{imbert2018decay}, we discuss how to replace the rapid decay assumption in the last item of Definition \ref{d:solution} with a weaker algebraic decay condition. 

The last condition on the integrability of $D^2_v f$ is convenient for Lemmas \ref{lem:cancel} and \ref{l:cancel-2}. These are the only parts of this paper where it plays a role.

\subsection{Decay estimates}

We start by reviewing the decay estimates in the velocity variables obtained
in \cite{imbert2018decay} for solutions of the Boltzmann equation. When $\gamma>0$, these decay estimates are forced by the equation regardless of the initial data. When $\gamma \leq 0$, we need to impose the appropriate decay on the initial data, and it is propagated in time by the equation.

%-------------------------------------------------------------------
\begin{thm}[Decay estimates in the velocity variables]\label{t:decay}
  Let the parameters $\gamma,s$ from \eqref{assum:B} satisfy
  $\gamma+2s \in [0,2]$ and let $f$ be a solution of the Boltzmann
  equation~\eqref{e:boltzmann} in
  $(0,T) \times \R^d \times \R^d$, periodic in $x$, such that
  $f(0,x,v) = \fin (x,v)$ in $\R^d \times \R^d$ and such that
  the Assumption~\ref{a:hydro-assumption} holds. If $\gamma \leq 0$, we also
  asssume that for all $q >0$, there exists a constant $N_{0,q}$ so that $\fin (x,v) \le N_{0,q} (1+|v|)^{-q}$ for
  $(x,v) \in \R^d \times \R^d$. Then the solution $f$ satisfies 
  \[ f (t,x,v) \le N_q (1+|v|)^{-q} \qquad \text{in } (0,T) \times \R^d \times \R^d,\]
  for some constant $N_q$ only depending on dimension $d$, parameters
  $\gamma,s$ from the collision kernel, see \eqref{assum:B}, and the
  hydrodynamical bounds $m_0,M_0,E_0,H_0$ from Assumption~\ref{a:hydro-assumption}, and, in the case $\gamma \leq 0$, also on the constants $N_{0,q}$.
\end{thm}

We include the following two lemmas about the decay or growth of convolutions of $f$ with different powers of $|v|$. They will be applied repeatedly in several parts of the paper. The first one gives us an upper bound depending on the mass and energy of $f$ only. The second one is in terms of its pointwise upper bounds.

\begin{lemma} \label{l:convolution-moments}
Let $f : \R^d \to [0,\infty)$. Assume that $0 \leq \kappa \leq 2$. Then
\[ \int_{\R^d} f(v+w) |w|^\kappa \dd w \leq C \left( (1+|v|)^\kappa M_0 + E_0 \right),\]
where $C$ is a universal constant and $M_0$ and $E_0$ are numbers so that 
\[ \int_{\R^d} f(v) \dd v \leq M_0 \qquad \text{and} \qquad \int_{\R^d} f(v) |v|^2 \dd v \leq E_0 .\]
\end{lemma}

\begin{proof}
We compute directly
\begin{align*} 
 \int_{\R^d} f(v+w) |w|^\kappa \dd w&= \int_{\R^d} f(w) |w-v|^\kappa \dd w, \\
&\lesssim \int_{\R^d} f(w) \left( |w|^\kappa + |v|^\kappa \right) \dd w, \\
&\lesssim \int_{\R^d} f(w) \left( 1 + |w|^2 + |v|^\kappa \right) \dd w, \\
&\lesssim \left( \int_{\R^d}f(w) |w|^2 \dd w + (1+|v|)^\kappa \int_{\R^d}f(w) \dd w\right).\qedhere
\end{align*}
\end{proof}

\begin{lemma} \label{l:convolution-C0}
Let $f : \R^d \to [0,\infty)$ and $\kappa > -d$. Assume that $f(v) \leq N (1+|v|)^{-q}$ for some $q>d+\kappa_+$. Then
\[ \int_{\R^d} f(v+w) |w|^\kappa \dd w \leq C N (1+|v|)^\kappa,\]
for some constant $C$ depending on $d$, $\kappa$ and $q$ only (neither on $N$ nor $v$).
\end{lemma}

\begin{proof}
We do a different computation depending on whether $\kappa \geq 0$ or $\kappa < 0$.

For $\kappa \geq 0$, it is very similar to Lemma~\ref{l:convolution-moments}. We compute
\begin{align*} 
 \int_{\R^d} f(v+w) |w|^\kappa &= \int_{\R^d} f(w) |w-v|^\kappa \dd w, \\
&\lesssim \int_{\R^d} f(w) \left( |w|^\kappa + |v|^\kappa \right) \dd w, \\
&\lesssim N \left( \int_{\R^d} (1+|w|)^{-q} |w|^\kappa \dd w + |v|^\kappa \int_{\R^d} (1+|w|)^{-q} \dd w\right).
\end{align*}
Since $-q+\kappa < -d$ and $\kappa > -d$, the integrals are computable for each value of $q$ and $\kappa \geq 0$.

For $\kappa <0$, we estimate the integrals differently. In this case we will use that $q>d$.
\begin{align*} 
 \int_{\R^d} f(v+w) |w|^\kappa \dd w &\leq N \int_{\R^d} (1+|v+w|)^{-q} |w|^\kappa \dd w, \\
& = N \int_{|w|<|v|/2} (1+|v+w|)^{-q} |w|^\kappa \dd w + N \int_{|w|>|v|/2} (1+|v+w|)^{-q} |w|^\kappa \dd w , \\
&\lesssim N \int_{|w|<|v|/2} (1+|v|)^{-q} |w|^\kappa \dd w +N  \int_{|w|>|v|/2} (1+|v+w|)^{-q} |v|^\kappa \dd w , \\
&\lesssim N \int_{|w|<|v|/2} (1+|v|)^{-q} |w|^\kappa \dd w + N \int_{\R^d} (1+|v+w|)^{-q} |v|^\kappa \dd w , \\
                                     & \lesssim N (1+|v|)^{-q} |v|^{\kappa+d} + |v|^\kappa \\
  & \lesssim N (1+|v|)^\kappa. \qedhere
\end{align*}
\end{proof}

\subsection{A coercivity condition for integro-differential operators}

In \cite{chaker2019coercivity}, there is a practical condition to verify if the quadratic form associated with an integro-differential operator is coercive with respect to the $H^s$ semi-norm. The result says that following.

\begin{thm} [Coercivity condition --  {\cite[{Theorem 1.3}]{chaker2019coercivity}}] \label{t:coercivity}
Let $K: B_2 \times B_2 \to [0,\infty)$ be a kernel satisfying the following assumption. There exists $\lambda>0$ and $\mu \in (0,1)$ such that for any $v \in B_2$ and any ball $B \subset B_2$ that contains $v$,
\[ | \{ v' \in B: K(v,v') \geq \lambda |v'-v|^{-d-2s} \}| \geq \mu |B|.\]
Then, for any function $u : B_2 \to \R$,
\[ \iint_{B_2 \times B_2} (u(v') - u(v))^2 K(v,v')  \dd v' \dd v \geq c \lambda \iint_{B_1 \times B_1} \frac{(u(v') - u(v))^2}{|v'-v|^{d+2s}} \dd v' \dd v.\]
The constant $c$ depends on dimension and $\mu$ only.
\end{thm}

We recall from \eqref{e:Boltzmann_IDE} that the Boltzmann equation can be written as an integro-differential equation with some kernel $K_f$ depending on the solution $f$ itself. The explicit formula for $K_f$ is worked out in \cite{luis} and we recall it now,
\begin{equation}\label{e:kf-0}
 K_f (v,v') = \frac{2^{d-1}}{|v'-v|} \int_{w \perp v'-v} f(v+w) B(r,\cos \theta) r^{-d+2} \dd w \quad \text{ 
with } 
\begin{cases} 
r^2 =|v'-v|^2 + |w|^2, \\ 
\cos \theta = \frac{w-(v-v')}{|w-(v-v')|}\cdot \frac{w+(v'-v)}{|w+(v'-v)|} .
\end{cases} 
\end{equation} 
 This kernel $K_f$ satisfies the assumption of Theorem~\ref{t:coercivity} in the stronger form of a cone of nondegeneracy as described in \cite{luis}. We describe it in the following proposition.

\begin{prop} [Cone of nondegeneracy --  {\cite[Lemma 7.1]{luis}}] \label{p:cone_of_nondegeneracy}
Let $f:\R^d \to \R$ be a nonnegative function and $K_f$ be the corresponding Boltzmann kernel as in \eqref{e:kf-0}. For any $v \in \R^d$, there exists a symmetric subset of the unit sphere $A(v) \subset S^{d-1}$ such that
\begin{itemize}
\item $|A(v)| \geq \mu (1+|v|)^{-1}$.
\item For every $\sigma \in A(v)$, $|\sigma \cdot v| \leq C$ \\ (i.e. $A(v)$ is concentrated around the equator perpendicular to $v$ with width $\le C/(1+|v|)$).
\item For any $\sigma \in A(v)$ and $r>0$, 
\[ K_f(v, v+r\sigma) \geq \lambda (1+|v|)^{1+\gamma+2s} r^{-d-2s}.\]
\end{itemize}
Here, the constants $\mu$, $\lambda$ and $C$ depend only on dimension and on the hydrodynamic bounds of Assumption~\ref{a:hydro-assumption}.
\end{prop}

The cone of nondegeneracy described in Proposition~\ref{p:cone_of_nondegeneracy} immediately implies the assumption of Theorem~\ref{t:coercivity}. Thus, the Boltzmann kernel $K_f$ satisfies a coercivity inequality restricted to velocities in $B_2$. Naturally, we can apply a translated and dilated version of Theorem~\ref{t:coercivity} to derive a coercivity condition for the Boltzmann collision kernel in any bounded set of velocities. It naturally implies a local coercivity inequality. However, as we see in Proposition~\ref{p:cone_of_nondegeneracy}, the thickness of the cone of nondegeneracy  degenerates as $|v| \to \infty$. This is natural in view of the fact that the optimal global coercivity inequalities for the Boltzmann collision operator depend on certain modified distance and weight that degenerate as $|v| \to \infty$ (see \cite{gressmanstrainBETTERpaper}).

We are able to recover the optimal coercivity estimates for large velocities using Theorem~\ref{t:coercivity} and Proposition~\ref{p:cone_of_nondegeneracy} together with the change of variables described in Section~\ref{s:cov}. See Appendix~\ref{s:gressman} for a derivation of the global coercivity estimate in \cite{gressmanstrainBETTERpaper} along these lines.

\section{Kinetic H\"older spaces}
\label{s:holder}

Here, following \cite{schauder}, we describe the appropriate formulation of H\"older spaces for kinetic equations. These are, in the context of H\"older spaces, what the spaces described in \cite{armstrong2019variational} are in the context of Sobolev spaces. They are adapted to the group of translations and dilations that leave the equation in an invariant ellipticity class. In order to motivate and explain all the necessary machinery related to these spaces, it is best to first consider the simpler fractional Kolmogorov equation
\begin{equation} \label{e:fractional_kolmogorov}
 \partial_t f + v \cdot \nabla_x f + (-\Delta)^s f = 0.
\end{equation}
The equation \eqref{e:fractional_kolmogorov} is the simplest kinetic equation with integro-differential diffusion of order $2s$ and it serves as a model equation to start our analysis. The Boltzmann collision operator is the sum of a nonlinear integro-differential operator of order $2s$ (which is \textbf{not} the fractional Laplacian) plus a lower order term. This decomposition is precisely given by the two terms in \eqref{e:Boltzmann_IDE}.

\subsection{Scaling and translation invariances}

Assume that a function $f$ solves \eqref{e:fractional_kolmogorov} in some domain. For any $r>0$, if we scale the function by
\begin{equation} \label{e:scaling}
 f_r(t,x,v) = f(r^{2s} t, r^{1+2s} x, rv),
\end{equation}
then the scaled function $f_r$ satisfies the same equation in the appropriately scaled domain.

The space $\R \times \R^d \times \R^d$ is endowed with the following Lie group structure: for all $\xi = (h,y,w)$ and $z= (t,x,v)$, the operator $\xi \circ z$ is given by the formula
\begin{equation} \label{e:lie_group}
 \xi \circ z = (h+t,x+y+ tw, v+w). 
\end{equation}
If $f$ is a solution of \eqref{e:fractional_kolmogorov} and $z_0 = (t_0,x_0,v_0) \in \R^{1+2d}$ is arbitrary, then the function
\[ \tilde f(z) = f(z_0 \circ z) \]
solves the same equation (in a translated domain).

The scaling invariance and left-translation invariance described here are the motivation for the definitions of kinetic cylinders, distance, degree and H\"older spaces given below.

In Sections~\ref{s:degiorgi} and \ref{s:schauder}, we will describe the results from \cite{imbert2016weak} and \cite{schauder} which are kinetic integro-differential versions of the classical regularity results of De Giorgi and Schauder for elliptic equations with variable coefficients. These equations are not invariant by scaling or translations individually, but rather as a class. Scaling or left translations of functions solving an equation as in Theorems~\ref{t:local-holder} or \ref{t:local-schauder}, will solve an equation with the same structure and the same ellipticity parameters.

\subsection{Cylinders}
\label{s:cylinders}
When working with parabolic equations, one often considers parabolic cylinders
of the form $(t_0-r^2,t_0] \times B_r(x_0)$. Because of the invariant
transformations we mentioned above, it is natural and convenient to consider cylinders respecting them. For all $z_0 \in \R^{1+2d}$, we define
\[ Q_r (z_0)  = \{ (t,x,v) : t_0 -r^{2s} < t \le t_0,  |x-x_0 -(t-t_0) v_0| < r^{1+2s}, |v-v_0| < r \}.\]
Cylinders centered at the origin $(0,0,0)$ and of radius $r>0$ are
simply denoted by $Q_r$.

Note that under this definition $(t,x,v)$ belongs to $Q_1$ if and only if $(r^{2s} t, r^{1+2s} x, rv)$ belongs to $Q_r$. Thus, our cylinders honor the scaling given in \eqref{e:scaling}. Moreover, for any $z_0 \in \R^{1+2d}$, we have $Q_r(z_0) = z_0 \circ Q_r$, where $\circ$ denotes the Lie group operator given in \eqref{e:lie_group}.

\subsection{Kinetic distance}

We recall the notion of kinetic distance defined in \cite{schauder}. It is constructed so that it agrees with the scaling given in \eqref{e:scaling} and the left action of the group \eqref{e:lie_group}.

\begin{defn} \label{d:kinetic_distance}
The \emph{kinetic distance} between two points $z_1=(t_1,x_1,v_1)$ and $z_2=(t_2,x_2,v_2)$ in $\R^{1+2d}$ is given by the following formula
\[
  d_\ell(z_1,z_2) := \min_{w \in \R^d} \left\{ \max \left( |t_1-t_2|^{ \frac 1 {2s} },
     |x_1-x_2-(t_1-t_2)w|^{ \frac 1 {1+2s} } , |v_1-w| , |v_2-w| \right) \right\}.
\]
\end{defn}

We show in \cite{schauder} that $d_\ell$ is indeed a distance when $s \geq 1/2$. For $s<1/2$, the triangle inequality fails for $d_\ell$, however $d_\ell(z_1,z_2)^{2s}$ is in fact a distance. We still work with $d_\ell$ for any value of $s \in (0,1)$ in order to keep our formulas consistent.

This distance is scale invariant in the following sense: for any $z_1, z_2 \in \R^{1+2d}$ and $r>0$, if we scale $S_r z_1 := (r^{2s} t_1, r^{1+2s} x_1, rv_1)$ and $S_r z_2 := (r^{2s} t_2, r^{1+2s} x_2, rv_2)$, we have
\[ d_\ell(S_r z_1,S_r z_2) = r d_\ell(z_1,z_2).\]

This distance is also left invariant by \eqref{e:lie_group}. Indeed, for any three points $\xi$, $z_1$ and $z_2$ in $\R^{1+2d}$ we have
\[ d_\ell(\xi \circ z_1, \xi \circ z_2) = d_\ell(z_1, z_2).\]

It is also convenient to define the length of a vector $z \in \R^{1+2d}$ by $\|z\| := d_\ell(z,0)$. Technically, $\|z\|$ is not a norm. It is not homogeneous of degree one, but rather it is homogeneous with respect to the scaling in \eqref{e:scaling}. It satisfies the triangle inequality with respect to the group action \eqref{e:lie_group}:
\begin{equation} \label{e:lie_norm}
 \|z_1 \circ z_2\| \leq \|z_1\| + \|z_2\|.
\end{equation}
There are several convenient equivalent expressions for $\|z\|$ that we write below.
\begin{align*} 
 \|z\| &= \min_{w \in \R^d} \left\{ \max \left( |t|^{ \frac 1 {2s} },
     |x-tw|^{ \frac 1 {1+2s} } , |v-w| , |w| \right) \right\}, \\
&\approx \max \left( |t|^{ \frac 1 {2s} },
     |x|^{ \frac 1 {1+2s} } , |v| \right) , \\
&\approx |t|^{ \frac 1 {2s} }+ |x|^{ \frac 1 {1+2s} } + |v|
\end{align*}

The symbol $\approx$ denotes in this context that the quantities on both sides are comparable up to a factor depending on $s$ and dimension $d$ only. The last two expressions would not satisfy \eqref{e:lie_norm}, but they are easier to compute in some cases in which the constant factors do not matter.

Note that due to the left invariance of this distance, $d_\ell(z_1,z_2) = \|z_2^{-1} \circ z_1\|$. 

The distance $d_\ell$ is left invariant but not right invariant. This lack of right invariance is occasionally problematic and results in some loss in some exponents in some inequalities. It is essentially the reason why the exponents $\alpha$ and $\alpha'$ are different in the main result of \cite{schauder}. For example, for all $z_1,z_2 \in \R^{1+2d}$ and $w \in \R^d$, a direct computation shows the following inequalities that will be used repeatedly.
\begin{equation}
  \label{e:repeatedly}
\begin{aligned}
 d_\ell (z_1 \circ (0,0,w), z_2 \circ (0,0,w)) &\leq d_\ell(z_1,z_2) + |t_1-t_2|^{1/(1+2s)} |w|^{1/(1+2s)}, \\ 
&\leq d_\ell(z_1,z_2) + d_\ell(z_1,z_2)^{2s / (1+2s)} |w|^{1/(1+2s)}.
\end{aligned}
\end{equation}
Moreover, when $d_\ell(z_1,z_2) \leq 1$, then the last inequality is also less than or equal to $d_\ell(z_1,z_2)^{2s / (1+2s)} (1+|w|)^{1/(1+2s)}$. We see that the right translations $z_1\circ (0,0,w)$ and $z_2 \circ (0,0,w)$ may be an order of magniture further apart than the original points $z_1$ and $z_2$.

\subsection{Kinetic degree of polynomials}

We recall the definition of kinetic degree from \cite{schauder}, for polynomials $p$ in $\R[t,x,v]$. Given a monomial $m$ of the form
\[ m(t,x,v) = c \, t^{\alpha_0} x_1^{\alpha_1} \dots x_d^{\alpha_d} v_1^{\alpha_{d+1}} \dots v_d^{\alpha_{2d}} \text{ with } c \neq 0,\]
we define its kinetic degree as 
\[ \degk m = 2s \alpha_0 + (1+2s) \sum_{j=1}^d \alpha_j + \sum_{j=d+1}^{2d} \alpha_j.\]
That is, the degree of $m$ is computed by counting $2s$ times
the exponent for the variable $t$, $1+2s$ for the exponents in the variables $x_i$ and $1$ for variables $v_i$. This definition is justified by the fact that we want a notion of kinetic degree consistent with the scaling \eqref{e:scaling}. With this definition in mind,
\[ m(r^{2s} t, r^{1+2s} x, rv) = r^{\degk m } m(t,x,v).\]

Given any non-zero polynomial $p$ in $\R[t,x,v]$ we define the kinetic degree of $p$ (and we write it $\degk p$) as the maximum of the kinetic degree of each of its (non-zero) monomials.

The kinetic degree of the zero polynomial is not properly defined above. It is appropriate to make it equal to $-\infty$ (or perhaps $-1$). The fact that the kinetic degree of the zero polynomial is a negative value is relevant for the definition of the $C^0_\ell$ norm given below in Definition~\ref{d:holder-space}.

\subsection{Kinetic H\"older spaces}

We recall here the kinetic H\"older spaces introduced in
\cite{schauder}. 

%---------------------------------------------------------------
\begin{defn}[Kinetic H\"older spaces] \label{d:holder-space} Given an
  open set $D \subset \R^{1+2d}$ and a parameter  $\alpha \in [0,\infty)$, a continuous function $f: D \to \R$ is  \emph{$\alpha$-H\"older continuous} at a point  $z_0 \in \R^{1+2d}$ if there exists a polynomial  $p \in \R[t,x,v]$ such that $\degk p < \alpha$ and for  any $z \in D$  \[ |f(z) - p(z)| \leq C d_\ell(z,z_0)^\alpha.\] When this property  holds at every point $z_0$ in the domain $D$, with a uniform  constant $C$, we say $f \in C_\ell^\alpha(D)$. The semi-norm   $[f]_{C_\ell^\alpha(D)}$ is the smallest value of the constant $C$ so  that the inequality above holds for all $z_0, z \in D$. 

Note that with this definition $[f]_{C_\ell^0(D)} = \|f\|_{C^0(D)} = \|f\|_{L^\infty(D)}$.
We define the norm  $\|f\|_{C_\ell^\alpha(D)}$ to be $[f]_{C_\ell^\alpha(D)}+\|f\|_{C^0(D)}$.
\end{defn}

We recall the interpolation inequalities proven in
\cite[Proposition~2.10]{schauder}.
% -------------------------------------------------------------------
\begin{prop}[Interpolation inequalities --
  \cite{schauder}] \label{p:interpol} Given
  $0 \le \alpha_1 < \alpha_2 < \alpha_3$  so that
  $\alpha_2 = \theta \alpha_1 + (1-\theta)\alpha_3$, we have for all
  function $f \in C_\ell^{\alpha_3} (Q_r (z_0))$,
  \[
    [f]_{C_\ell^{\alpha_2}(Q_r (z_0))} \leq C \left( [f]_{C_\ell^{\alpha_1}(Q_r(z_0))}^\theta [f]_{C_\ell^{\alpha_3}(Q_r(z_0))}^{1-\theta}
    + r^{\alpha_1-\alpha_2} [f]_{C_\ell^{\alpha_1}(Q_r(z_0))} \right),
  \]
for some constant $C$ depending on $\alpha_1$, $\alpha_3$ and dimension only.
\end{prop}

In this article, we will iteratively gain a priori estimates for solutions to the Boltzmann equation on H\"older spaces with increasingly large exponents $\alpha$. We deal with global estimates that work for all $v \in \R^d$. We need to keep track of the asymptotic behavior of these norms for large velocities. The most convenient way to do it is by considering functions in H\"older spaces with fast decay, that we define below.
% -----------------------------------------------------------------
\begin{defn}[H\"older spaces with fast
  decay] \label{d:holder-space-fast} Given $\alpha \in [0,\infty)$, a
  function $f: [\tau,T] \times \R^d \times \R^d$ lies in $\Cpol^\alpha$
  if, for all $q >0$ and all $r \in (0,1]$, there exists $C_{q}>0$
  such that for all $z \in [\tau,T] \times \R^d \times \R^d$, with
  $Q_r (z) \subset [\tau,T] \times \R^d \times \R^d$, 
  \[ \|f\|_{C_\ell^\alpha (Q_r(z))} \le \frac{C_{q}}{(1+|v|)^q} .\]
This is a locally convex vector space with the following family of semi-norms
\[ [f]_{C_{\ell,q}^\alpha([\tau,T] \times \R^d \times \R^d)} := \sup \left\{ (1+|v|)^q [f]_{C_\ell^\alpha (Q_r(z))} : r \in (0,1] \text{ and } Q_r(z) \subset [\tau,T] \times \R^d \times \R^d \right\}.\]
We also write $\|f\|_{C^\alpha_{\ell,q}} := \|f\|_{C^0_{\ell,q}} + [f]_{C^\alpha_{\ell,q}}$. Thus, a function $f$ belongs to $\Cpol^\alpha$ when $\|f\|_{C^\alpha_{\ell,q}} < \infty$ for all $q>0$. 
\end{defn}
%------------------------------------------------------------------

Note that $[f]_{C^\alpha_{\ell,q_1}} \geq  [f]_{C^\alpha_{\ell,q_2}}$ if $q_1 \geq q_2$. Also, the norm $\|f\|_{C_\ell^\alpha (Q_r(z))}$ is monotone increasing with respect to $r$. It is pointless to consider small values of $r$ in Definition~\ref{d:holder-space-fast}. In practice, one would only take the largest $r$ that the interval $[\tau,T]$ allows, which will often be $r=1$.

\medskip

We know that the property of a function being H\"older continuous is local, but its H\"older norm is not (at least for non-integer exponents). The following lemma is useful to obtain a H\"older estimate in a large domain by covering with smaller patches where the H\"older norm is bounded.
%-------------------------------------------------------------------------
\begin{lemma} \label{l:Calpha_local}
Let $\alpha >0$, $r_0>0$, $f : Q_1 \to \R$ be a bounded continuous function. Assume that for every $z_0 \in Q_1$, there is a polynomial $p_{z_0}$ of kinetic degree strictly less than $\alpha$ such that
\[ |f(z_0 \circ \xi) - p_{z_0}(\xi)| \leq C_0 \|\xi\|^\alpha,\]
whenever $\|\xi\| \leq r_0$ and ${z_0} \circ \xi \in Q_1$. Then $f \in C^\alpha_\ell(Q_1)$ and
\[ [f]_{C^\alpha_\ell(Q_1)} \leq C_0 + C r_0^{-\alpha} \osc_{Q_1} f,\]
for a constant $C$ depending on $\alpha$ and dimension only. Here $\osc_{Q_1} f = \sup_{Q_1} f - \inf_{Q_1} f$.
\end{lemma}
%-------------------------------------------------------------------------
\begin{proof}
The inequality we assume for $|f({z_0} \circ \xi) - p_{z_0}(\xi)|$ when $\|\xi\|<r_0$ is identical to the one in Definition~\ref{d:holder-space}. We need to extend this inequality to every value of $\xi$ so that $z_0 \circ \xi \in Q_1$, regardless of whether $\|\xi\| < r_0$ or not.

Without loss of generality, let us assume $\inf_{Q_1} f=0$ (otherwise, repeat the proof below for $f - \inf_{Q_1} f$). Thus, in this case $\osc_{Q_1} f = \|f\|_{C^0(Q_1)}$. In the following we simply write $\|f\|_{C^0}$ for $\|f\|_{C^0(Q_1)}$.

For any point $z_0 \in Q_1$, let us analyze the polynomial $p_{z_0}$. We know that whenever $\|\xi\| \leq r_0$ and $z_0 \circ \xi \in Q_1$, $|f(z_0 \circ \xi) - p_{z_0}(\xi)| \leq C_0 \|\xi\|^\alpha$. In particular, $|p_{z_0}| \leq C_0 r_0^\alpha + \|f\|_{C^0}$ at those points.

We use Lemma 2.8 in \cite{schauder} (See also the proof of Lemma 2.7 and Proposition 2.10 in \cite{schauder}), and get that for any point $z_0 \in Q_1$, the polynomial $p_{z_0}$ has the form
\[ p_{z_0}(z) = \sum_{j \in \mathbb N^{1+2d}} a_j m_j(z),\]
where $a_j \neq 0$ only for multi-indexes so that $\deg_k m_j < \alpha$, and moreover 
\[ |a_j| \leq C \left( C_0 r_0^{\alpha - \deg_k m_j}  + \|f\|_{C^0} r_0^{-\deg_k m_j} \right).\]

Thus, when $z_0 \circ \xi \in Q_1$ but $\|\xi\| > r_0$, we estimate
\begin{align*} 
 |f(z_0 \circ \xi) - p_{z_0}(\xi)| &\leq \|f\|_{C^0} + |p_{z_0}(\xi)| , \\
&\leq \|f\|_{C^0} + \sum_{j \in \mathbb N^{1+2d}} |a_j| \|\xi\|^{\deg_k m_j}, \\
&\leq \|f\|_{C^0} + C \left( C_0 r_0^\alpha + \|f\|_{C^0} \right)  \sum_{\substack{j \in \mathbb N^{1+2d} \\ \deg_k m_j < \alpha}} \left( \frac{\|\xi\|}{r_0} \right)^{\deg_k m_j}, \\
&\lesssim\left( C_0 r_0^\alpha + \|f\|_{C^0} \right) \left( \frac{\|\xi\|}{r_0} \right)^\alpha, \\
&= \left( C_0 + r_0^{-\alpha} \|f\|_{C^0} \right) \|\xi\|^\alpha.
\end{align*}
And we conclude the proof.
\end{proof}

\begin{remark}
Comparing with classical H\"older spaces $C^\alpha$, we observe that the estimate in Lemma~\ref{l:Calpha_local} is not optimal for large values of $\alpha$. Consider for example the Lipschitz norm, that corresponds to $\alpha=1$ and is purely local. So, the optimal inequality for the classical Lipschitz space would not have the second term in Lemma~\ref{l:Calpha_local}. H\"older norms are non-local, so some dependence on $r_0$ ought to be retained at least when $\alpha$ is not an integer (or $\alpha \notin \mathbb N + 2s \mathbb N$ for kinetic H\"older spaces).
\end{remark}

\begin{lemma} \label{l:Calpha_product}
Let $f, g \in C^\alpha_\ell(Q_1)$. Then $fg \in C^\alpha_\ell(Q_1)$ and
\[ \|f g\|_{C^\alpha_\ell(Q_1)} \le C \|f\|_{C^\alpha_\ell(Q_1)} \|g\|_{C^\alpha_\ell(Q_1)} ,\]
for a constant $C$ depending on dimension and $\alpha$ only.
\end{lemma}

\begin{proof}
It is clear that the $C^0$ norm satisfies the inequality with constant $C=1$. We are left with verifying the inequality for the semi-norm $[\cdot]_{C^\alpha_\ell}$. To that end, let $z \in Q_1$ and consider the polynomials $p$ and $q$ of kinetic degree less than $\alpha$ so that
\[ |f(z \circ \xi) - p(\xi)| \leq [f]_{C^\alpha_\ell} \|\xi\|^\alpha \qquad \text{and} \qquad |g(z\circ \xi) - q(\xi)| \leq [g]_{C^\alpha_\ell} \|\xi\|^\alpha.\]
Thanks to Lemma~\ref{l:Calpha_local}, it is sufficient to consider the case $\|\xi\| < 1$. We have
\begin{align*} 
 |f(z \circ \xi) g (z \circ \xi) - p(\xi) q(\xi)| &\leq |f(z \circ \xi)||g(z \circ \xi) - q(\xi)| + |f(z\circ \xi) - p(\xi)| |q(\xi)|, \\
&\leq \left( [g]_{C_\ell^\alpha} \|f\|_{C^0} + [f]_{C^\alpha_\ell} \|q\|_{C^0} \right) \|\xi\|^\alpha, \\
&\lesssim \|g\|_{C^\alpha_\ell} \|f\|_{C^\alpha_\ell}  \|\xi\|^\alpha.
\end{align*}
The last inequality holds for $\|\xi\| < 1$ due to the identification of the coefficients of $q$ with derivatives of $g$ and   \cite[Lemma 2.7]{schauder} (see also  \cite[Remark 2.9]{schauder}).

The polynomial $p(\xi)q(\xi)$ may have a kinetic degree higher than $\alpha$. In that case, let $r(\xi)$ be the sum of the terms in $p(\xi)q(\xi)$ of kinetic degree larger than or equal to $\alpha$. We  also see from the application of \cite[Lemma 2.7 \& Remark 2.9]{schauder}, reasoning  term by term, that $|r(\xi)| \leq \|f\|_{C^\alpha_\ell}  \|g\|_{C^\alpha_\ell} \|\xi\|^\alpha$ whenever $\|\xi\| < 1$. Thus, the lemma follows.
\end{proof}

\section{Kinetic equations with integral diffusion}
\label{l:lin-kin}

\subsection{The kernel associated with the Boltzmann equation}
\label{s:boltzmann}

Like in \eqref{e:Boltzmann_IDE}, we use the decomposition of the Boltzmann collision operator described in \cite{silvestre2016new} and suggested earlier in \cite[Chapter 2, Section 6.2]{villani-book}. We split Boltzmann's collision operator $\Q(f,f)$ appearing in \eqref{e:boltzmann} as the sum of two terms $\Q = \Q_1 + \Q_2$. The first term $\Q_1(f,f)$ is an integro-differential operator and $\Q_2(f,f)$ a lower order term,
\begin{equation}
  \label{e:Q-bis}
\begin{aligned}
\Q_1(f,f) &:= \mathcal{L}_{K_f} f, \\
\Q_2(f,f) &:= c_b (f \ast |\cdot|^\gamma) f
\end{aligned}
\end{equation}
where $c_b$ is a positive constant only depending on the function $b$
appearing in \eqref{assum:B} and where the integro-differential diffusion operator $\mathcal L_{K_f}$ is defined
as
\[
  \mathcal{L}_{K_f} g (t,x,v) = \PV \int_{\R^d} (g(t,x,v') -g(t,x,v)) K_f (t,x,v,v') \dv'.
\]
The kernel $K_f$ characterizing the operator $\mathcal{L}_{K_f}$ is
given by the following formula
\begin{equation}\label{e:kf}
 K_f (v,v') = \frac{2^{d-1}}{|v'-v|} \int_{w \perp v'-v} f(v+w) B(r,\cos \theta) r^{-d+2} \dd w \quad \text{ 
with } 
\begin{cases} 
r^2 =|v'-v|^2 + |w|^2, \\ 
\cos \theta = \frac{w-(v-v')}{|w-(v-v')|}\cdot \frac{w+(v'-v)}{|w+(v'-v)|} .
\end{cases} 
\end{equation} 
The following expressions are easier to handle in computations.
\begin{align} 
  \label{e:A}
 K_f(v,v') &=  |v'-v|^{-d-2s} \int_{w \perp v'-v} f(t,x,v+w) A(|v'-v|,|w|)  |w|^{\gamma+2s+1} \dd w  \\
\label{e:Kernel-approx}
  & \approx |v-v'|^{-d-2s} \int_{w \perp (v'-v)} f(t,x,v+w) |w|^{\gamma+2s+1} \dd w
\end{align}
where $A \simeq 1$ is a bounded function only depending on the collision kernel $B$.

In Formula \eqref{e:kf} we omited the $(t,x)$ dependence in $K_f=K_f(t,x,v,v')$ and $f=f(t,x,v)$. This is because for every fixed value of $(t,x)$, we think of $f(t,x,\cdot)$ as a function of $v$ and compute the kernel $K_f$ accordingly by Formula \eqref{e:kf}. Thus, if $f=f(v)$ is a function of $v$ only, the kernel $K_f=K_f(v,v')$ depends on $v$ and $v'$. When $f$ depends on other parameters, so does $K_f$. In particular, $K_f=K_f(t,x,v,v')$ when $f=f(t,x,v)$. In the same spirit, we occasionally refer to Assumption~\ref{a:hydro-assumption} for a function $f=f(v)$ depending only on $v$ as a way to state that its mass, energy and entropy are bounded by constants $m_0>0$, $M_0$, $E_0$ and $H_0$. This abuse of notation is convenient when stating lemmas that relate bounds for $f$ with bounds for $K_f$.

There are two general regularity results for general kinetic integro-differential equations that we apply in this paper. The first one, given in \cite{imbert2016weak}, is a H\"older estimate with a small exponent, in the style of the well known theorem of De Giorgi and Nash. It is, in some sense, the integro-differential version of the result in \cite{golse2016harnack}. The second one, given in \cite{schauder}, is a higher order H\"older estimate in the style of the classical result by Schauder for linear elliptic equations. We will iterate these Schauder-type estimates, and combine them with the large-velocity decay estimates from \cite{imbert2018decay}, in order to obtain $C^\infty$ estimates. Each of these two regularity results depends on different conditions on the diffusion kernel $K_f$. In the next two sections, we discuss the assumptions for each of these results.

\begin{remark} \label{r:div-vs-nondiv}
The theorems for kinetic integro-differential equations in the style of De Giorgi/Nash (explained below in Section~\ref{s:degiorgi}) and the Schauder-type results (described in Section~\ref{s:schauder}) depend on assumptions on the kernel that look rather different from each other. These assumptions are best understood by comparing them with the corresponding conditions for the Landau equation, in terms of classical second order diffusion, that formally correspond to the limit as $s \to 1$.

The in-homogeneous Landau equation is also a kinetic equation of the form
\begin{equation} \label{e:landau-divergence}
 \partial_t f + v\cdot \nabla_x f = \mathcal{Q}_L(f,f),
\end{equation}
where the collision operator $\mathcal{Q}_L(f,f)$ involves second order derivatives of the function $f$. It can be written in divergence form
\begin{equation} \label{e:landau-divergence-bis}
  \mathcal{Q}_L(f,f) = \partial_{v_i} \left( a^f_{ij} \partial_{v_j} f + b^f_i f \right).
\end{equation}
The H\"older estimates (as in \cite{golse2016harnack}), obtained following De Giorgi method, depend on this expression and on the uniform ellipticity conditions on the diffusion matrix $a^f_{ij}$ (that depend on the solution $f$ itself through its hydrodynamic quantities). The term $\partial_{v_i} (b^f_i f)$ is of lower order.

The application of a Schauder type estimate (for example as in \cite{sergio2004recent}) would depend on the expression of $\mathcal{Q}_L(f,f)$ in non-divergence form. In the case of the Landau equation, it takes the form
\begin{equation} \label{e:landau-nondivergence}
 \mathcal{Q}_L(f,f) = a^f_{ij} \partial_{v_i v_j} f + c^f f.
\end{equation}
Here, the lower order term is $c^f f$. The Schauder estimate depends on the diffusion coefficients (in this case $a_{ij}^f$) being uniformly elliptic and H\"older continuous.

The difference between the divergence and nondivergence structures in \eqref{e:landau-divergence} and \eqref{e:landau-nondivergence} translates in different structure assumptions for the diffusion kernel in their integro-differential counterpart. The two terms in the decomposition of the Boltzmann collision kernel \eqref{e:Q-bis} correspond more naturally to \eqref{e:landau-nondivergence} than to \eqref{e:landau-divergence}.

One can apply divergence-form techniques to integro-differential operators when these have a variational structure. It corresponds to cancellation conditions between  $K(t,x,v,v')$ and $K(t,x,v',v)$. Ideally, the case of symmetry of the form $K(t,x,v,v') = K(t,x,v',v)$ would correspond to an integro-differential operator in divergence form without lower order terms. However, we see in \eqref{e:landau-divergence-bis} that there is a first order lower order term in the Landau equation. In Section~\ref{s:degiorgi} we will state the precise cancellation conditions for $K(t,x,v,v') - K(t,x,v',v)$ so that the asymmetry in the kernel is of lower order than the diffusion. The H\"older estimate in the style of the theorem of De Giorgi, Nash and Moser in Section~\ref{s:degiorgi} depends on this cancellation condition.

One can apply nondivergence techniques to integro-differential operators when their structure allows us to make pointwise estimates. It corresponds to the cancellation condition $K(t,x,v,v+w) = K(t,x,v,v-w)$. The Boltzmann kernel $K_f$ satisfies this symmetry by construction. Thus, the application of nondivergence techniques to the Boltzmann equation is more direct. This will be reflected in the assumptions of the Schauder-type estimate described in Section~\ref{s:schauder}.
\end{remark}

\subsection{The local H\"older estimate}

\label{s:degiorgi}

A local H\"older estimate for a general class of kinetic equations
with integral diffusion was obtained in \cite{imbert2016weak} following classical ideas from De Giorgi. It applies to equations of the form
\begin{equation}
  \label{e:class}
  \partial_t g  + v \cdot \nabla_x g= \mathcal{L}_{K} g +h 
\end{equation}
where $h$ is a given source term and $\mathcal{L}_{K}$ is an integral operator of the form
\[
 \mathcal{L}_{K} g (t,x,v) = \PV \int_{\R^d} (g(t,x,v') -g(t,x,v)) K
(t,x,v,v') \dv' 
\]
associated with a (non-negative) kernel $K(t,x,v,v')$ defined in
$(-1,0] \times B_1 \times B_2 \times \R^d$. 

The H\"older estimates for kinetic integro-differential equations developed in \cite{imbert2016weak} are a result comparable to the theorem of De Giorgi, Nash and Moser for elliptic or parabolic equations in divergence form. These regularity estimates are independent of any well-posedness questions. It does not matter where the kernel $K$ comes from, weather it depends on $f$ or not, or how smooth it is with respect to any of its parameters. It is a result that only requires some uniform ellipticity conditions on the kernel $K$ that we describe below. 

The following list of assumptions must be met uniformly in $t$ and $x$. In order to keep the formulas short, we omit their dependence on $t$ and $x$.

\medskip
\paragraph{\it Non-degeneracy conditions}

\begin{eqnarray} \label{e:nondeg1}
&&\text{For all $v \in B_2$ and $r>0$,} \;  \inf_{|e| =1} \int_{B_r(v)} ((v'-v) \cdot e)^2_+ K(v,v') \dd v' \ge \lambda r^{2-2s}.\\
 \label{e:nondeg2}
&& \text{For any $f$ supported in $B_2$,} \; \iint_{B_2 \times \R^d} f(v)(f(v)-f(v')) K(v,v') \dd v' \dd v \geq \lambda \|f\|_{\dot{H}^s(\R^d)}^2 - \Lambda \|f\|^2_{L^2(\R^d)}.
\end{eqnarray}
The first non-degeneracy condition~\eqref{e:nondeg1} is necessary only
for $s < 1/2$. It is not clear if the second condition
\eqref{e:nondeg2} may actually follow from \eqref{e:nondeg1}. In
practice, \eqref{e:nondeg1} is usually much easier to check than
\eqref{e:nondeg2}.

\medskip
\paragraph{\it Boundedness conditions}

\begin{eqnarray} \label{e:bounded1}
&& \text{For all $v \in B_2$ and $r>0$,} \qquad  \int_{\R^d \setminus B_r(v)} K(v,v') \dd v' \le \Lambda r^{-2s}.\\
 \label{e:bounded2}
&& \text{For all $v' \in B_2$ and $r>0$,} \qquad  \int_{\R^d \setminus B_r(v')} K(v,v') \dd v \le \Lambda r^{-2s}.
\end{eqnarray}

\medskip
\paragraph{\it Cancellation conditions}

\begin{eqnarray}
  \label{e:cancellation1}
  && \text{For all $v \in B_{7/4}$,} \qquad \left\vert \PV \int_{B_{1/4}(v)} \left( K(v,v') - K(v',v) \right) \dd v' \right\vert \le \Lambda.\\
  \label{e:cancellation2}
  && \text{For all $r \in [0,1/4]$ and  $v \in B_{7/4}$,} 
  \qquad \left\vert \PV \int_{B_r (v)} \left( K(v,v') - K(v',v) \right)(v'-v) \dd v' \right\vert \le \Lambda (1+r^{1-2s}).
\end{eqnarray}
The second cancellation condition~\eqref{e:cancellation2} is necessary only for $s \ge \frac12$. 

The cancellation conditions \eqref{e:cancellation1} and \eqref{e:cancellation2} correspond to the representation of the integral diffusion as a divergence form operator with a lower order asymmetry (see Remark~\ref{r:div-vs-nondiv} above).

The nondegeneracy and boundedness conditions \eqref{e:nondeg1}, \eqref{e:nondeg2}, \eqref{e:bounded1} and \eqref{e:bounded2} correspond to the uniform ellipticity of the integral diffusion kernel. In practice, the most difficult to verify is the coercivity assumption~\eqref{e:nondeg2}.

%----------------------------------------------------------------------------------
\begin{thm}[Local H\"older estimate -- \cite{imbert2016weak}]\label{t:local-holder}
  Let
  $K:(-1,0] \times B_1 \times B_2 \times \R^d \to
  [0,+\infty)$
  be a kernel satisfying the ellipticity conditions~\eqref{e:nondeg1}
  (only if $s < \frac12$), \eqref{e:nondeg2}, \eqref{e:bounded1},
  \eqref{e:bounded2}, \eqref{e:cancellation1}, \eqref{e:cancellation2}
  (only if $s \ge \frac12$).  Let $f:(-1,0] \times B_1 \times \R^d$ be
  a solution of \eqref{e:class} in $Q_1$ for some bounded function
  $h$. Assume also that $f$ is bounded in
  $(-1,0] \times B_1 \times \R^d$. Then $f$ is H\"older continuous in
  $Q_{\frac12}$ and the following estimate holds
  \[ [f]_{C^\alpha_\ell (Q_{\frac12})} \le C ( \|f\|_{L^\infty
    ((-1,0]\times B_1 \times \R^d)} + \|h \|_{L^\infty (Q_1)}) \]
  where $\alpha \in (0,1)$ and $C>0$ only depend on dimension $d$, and the
  constants $\lambda$ and $\Lambda$ appearing in the assumptions.
\end{thm}
%--------

In \cite{imbert2016weak}, we verified that the Boltzmann kernel $K_f$ (given in \eqref{e:kf}) satisfies (locally in $v$) the assumptions \eqref{e:nondeg1}, \eqref{e:nondeg2}, \eqref{e:bounded1}, \eqref{e:bounded2}, \eqref{e:cancellation1} and \eqref{e:cancellation2} with parameters depending only on the hydrodynamic constants $m_0$, $M_0$, $E_0$ and $H_0$ of Assumption~\ref{a:hydro-assumption}. 

\begin{remark} 
  In \cite{imbert2016weak}, the H\"older estimate is obtained in a
  classical H\"older space for some $\alpha>0$ sufficiently small. Such a H\"older estimate implies an estimate in the kinetic H\"older space used in the present work at the expense of reducing the exponent $\alpha$ by a factor $\min(2s,1)$. This is because for any two points $z,z_0 \in \R \times B_1 \times \R^d$ such that $d_\ell(z,z_0)<1$, we have $|z-z_0| \leq C d_\ell(z,z_0)^{\min (1,2s)}$. Indeed, when $|v_0| \leq 1$, we have
 \begin{align*} 
  |z -z_0| & = |t-t_0| + |x-x_0| + |v-v_0| \\
             & \le  (1+ |v_0|)|t-t_0| + |x-x_0 -(t-t_0)v_0 | + |v-v_0|  \\
             & \leq C( d_\ell (z,z_0)^{2s} +d_\ell (z,z_0)^{1+2s}+d_\ell (z,z_0))
  \end{align*}
  for some constant $C$ only depending on $s$.  In particular, if for all $z_0,z \in Q_{\frac12}$, we have
  $|f(z)-f(z_0)| \le C_\alpha |z-z_0|^\alpha$, then
  $|f(z)-f(z_0)| \le \tilde{C} C_\alpha d_\ell (z,z_0)^{\tilde \alpha}$ with
  $\tilde{\alpha} = \min (1,2s) \alpha$. The constant $\tilde{C}$ only
  depends on $s$ and $\alpha$.
\end{remark}

\subsection{A Schauder estimate for kinetic integro-differential equations}

\label{s:schauder}

The classical Schauder estimates for elliptic or parabolic equations of second order apply whenever we have an equation with uniformly elliptic and H\"older continuous coefficients. In \cite{schauder}, we obtained a Schauder-type estimate for kinetic equations with integro-differential diffusion like \eqref{e:class} in non-divergence form. The result depends on the kernel satisfying different
ellipticity conditions than the ones ensuring the local H\"older
estimate (Theorem~\ref{t:local-holder}). In some sense, the conditions described below reflect that the integro-differential equation is in non-divergence form and the kernel has a H\"older continuous dependence with respect to $(t,x,v)$. They should be understood from the perspective described in Remark~\ref{r:div-vs-nondiv} above.

There are two types of conditions that are necessary for a Schauder-type estimate: uniform ellipticity and H\"older continuity of coefficients. We should think of the kernel $K(t,x,v,v')$ as a map from the first three variables $(t,x,v)$ to a kernel depending on a single parameter $w \in \R^d$ given by $K_{(t,x,v)} (w) := K(t,x,v,v+w)$. The uniform ellipticity assumption will say that for every value of $(t,x,v)$, the kernel $K_{(t,x,v)}$ belongs to a certain ellipticity class. The H\"older continuity will say that for two different values $z_1=(t_1,x_1,v_1)$ and $z_2 = (t_2,x_2,v_2)$, the kernels $K_{z_1}$ and $K_{z_2}$ are in some sense at distance $\lesssim d_\ell(z_1,z_2)^\alpha$.

Let us recall the ellipticity class of order $2s$ defined in \cite{schauder}.

\begin{defn}[The ellipticity class] \label{d:classK}
Let $s \in (0,1)$.   A non-negative kernel $K: \R^d \to \R$ belongs to the ellipticity class $\mathcal{K}$ if
\begin{itemize}
\item[i.] $K(w) = K(-w)$, 
\item[ii.] For all $r>0$, $\int_{B_r} |w|^2 K(w) \dw \le \Lambda r^{2-2s}$, 
\item[iii.] For all $R>0$ and $\varphi \in C^2 (B_R)$, 
  \begin{equation} \label{e:frozen_coercivity}
 \iint_{B_R \times B_R} (\varphi (v') -\varphi (v))^2 K (v'-v) \dv' \dv \ge \lambda \iint_{B_{R/2} \times B_{R/2}} (\varphi (v') -\varphi (v))^2 |v'-v|^{-d-2s} \dv' \dv .
\end{equation}
\item[iv.] For any $r>0$ and $e \in S^{d-1}$, \[ \int_{B_r} (w \cdot e)_+^2 K(w) \dw \ge \lambda r^{2-2s}.\]
\end{itemize}
\end{defn}

Some remarks are in order. 
\begin{enumerate}
\item Definition~\ref{d:classK} is borrowed from \cite{schauder}. However, the definition in that paper is more general since $K(w) \dd w$ is supposed to be a nonnegative Radon measure that is not necessarily absolutely continuous. For the purpose of this paper, because we deal with classical solutions, our kernel $K$ will always be given by a non-negative density function and we do not need to deal with singular measures.
\item The last two items iii. and iv. might be redundant. Indeed, we do not know any example of a kernel satisfying i. and ii. and either iii. or iv., without satisfying all of them. This is related to the problem of coercivity for integro-differential operators. See the discussion in \cite{imbert2016weak}, \cite{dyda2011comparability}, \cite{dyda2015regularity} and \cite{chaker2019coercivity}. Item iv. is in practice much easier to verify than iii.
\item Earlier works on integro-differential equations concentrated on a more restricted class of kernels that were pointwise comparable to the fractional Laplacian: $K(w) \approx |w|^{-d-2s}$. This traditional assumption does not suffice to study the Boltzmann equation. The diffusion kernel that appears in the Boltzmann equation belongs to the more general class of Defintion \ref{d:classK}, with parameters depending on the constants in Assumption~\ref{a:hydro-assumption}.
\end{enumerate}

The condition on the H\"older dependence of $K_z$ with respect to the point $z$ is given in the assumption~\eqref{e:A0-schauder} below.

%------------------------------------------------------------------------------------
\begin{thm}[Local Schauder estimate -- \cite{schauder}]\label{t:local-schauder}
  Let $s \in (0,1)$ and $\alpha \in (0, \min (1,2s))$ and
  $\alpha' = \frac{2s}{1+2s} \alpha$. Let
  $K: (-(2r)^{2s},0] \times B_{(2r)^{1+2s}} \times \R^d \times \R^d \to [0,+\infty)$ such
  that for all $z =(t,x,v) \in (-(2r)^{2s},0] \times B_{(2r)^{1+2s}} \times \R^d$, the
  kernel $K_z (w) = K(t,x,v, v+w)$ belongs to the ellipticity class
  $\mathcal{K}$ from Definition~\ref{d:classK}. Assume moreover that for
  all $z_1,z_2 \in Q_{2r}$ and all $\rho >0$,
  \begin{equation}
    \label{e:A0-schauder}
  \int_{B_\rho} \big|K_{z_1} (w) - K_{z_2} (w) \big| \; |w|^2  \dd w
  \le A_0 \rho^{2-2s} d_\ell (z_1,z_2)^{\alpha'}
  \end{equation}
  with $z_i =(t_i,x_i,v_i)$ for $i=1,2$. 

If $f \in C_\ell^\alpha ((-(2r)^{2s},0] \times B_{(2r)^{1+2s}} \times \R^d)$ solves \eqref{e:class} in $Q_{2r}$, then 
\begin{align}
\label{e:loc-holder}
  [f]_{C_\ell^{2s+\alpha'} (Q_{r})} \le C \bigg( & \max \left( r^{-2s-\alpha'+\alpha}, A_0^{\frac{2s+\alpha'-\alpha}{\alpha'}} \right) [f]_{C_\ell^\alpha ((-(2r)^{2s},0] \times B_{(2r)^{1+2s}} \times \R^d)} \\
\nonumber & + [h]_{C_\ell^{\alpha'} (Q_{2r})} + \max( r^{-\alpha'}, A_0 )  \|h\|_{C^0(Q_{2r})} \bigg).
\end{align}
where the constant $C$ depends on $d,s$, and the constants
$\lambda, \Lambda$ from the definition of $\mathcal{K}$.
\end{thm}

\begin{proof}
The main result in \cite{schauder} is for $r=1$ and the constant $C$ depends on $A_0$ in an unspecified way. In order to justify \eqref{e:loc-holder} we work out explicitly its dependence on $r$ and $A_0$. It is a consequence of Theorem 1.6 in \cite{schauder} combined with a scaling argument. Indeed, let $S_r(t,x,v) = (r^{2s} t, r^{1+2s} x, rv)$. This is the natural scaling that maps $Q_1$ into $Q_r$. The function $f \circ S_r$ satisfies the scaled equation
\[ (\partial_t + v \cdot \nabla_x ) [f \circ S_r] = \mathcal L_{\tilde K} + r^{2s} h \circ S_r,\]
where
\[ \tilde K_z(w) = r^{d+2s} K_{S_r z} (rw).\]
We point out that $\tilde K$ satisfies assumption \eqref{e:A0-schauder} with $r^{\alpha'} A_0$ instead of $A_0$. Indeed,
\begin{align*} 
 \int_{B_\rho} (\tilde K_{z_1}(w) - \tilde K_{z_2}(w)) |w|^2 \dd w &= \int_{B_\rho} r^{d+2s} \left( K_{S_r z_1} (rw) - K_{S_r z_2} (rw) \right) |w|^2 \dd w, \\
&= r^{2s-2} \int_{B_{r\rho}}  \left( K_{S_r z_1} (\bar w) - K_{S_r z_2} (\bar w) \right) |\bar w|^2 \dd \bar w, \\
&\leq A_0 r^{2s-2} (r \rho)^{2-2s} d_\ell(S_r z_1, S_r z_2)^{\alpha'} = (A_0 r^{\alpha'}) \rho^{2-2s} d_\ell(z_1,z_2)^{\alpha'}.
\end{align*}

Provided that $r^{\alpha'} A_0 \leq 1$, we apply Theorem 1.6 in \cite{schauder} (same as Theorem~\ref{t:local-schauder} but for $r=1$ and constants depending implicitly on $A_0$). We get
\begin{equation} \label{e:sch-a1}  [f \circ S_r]_{C_\ell^{2s+\alpha'} (Q_1)} \le C\left([f \circ S_r]_{C_\ell^\alpha ((-2^{2s},0] \times B_{2^{1+2s}} \times \R^d)} + r^{2s} \| h \circ S_r \|_{C_\ell^{\alpha'} (Q_{2})}\right). 
\end{equation}
We can take a universal constant $C$ (depending on $d$, $\lambda$, $\Lambda$ and $s$ only) provided that $A_0 r^{\alpha'} \leq 1$. In that case, we scale back to express \eqref{e:sch-a1} in terms of the original functions $f$ and $h$ to obtain
\begin{equation} \label{e:sch-a2}
 [f]_{C_\ell^{2s+\alpha'} (Q_{r})} \le C \left(r^{\alpha-2s-\alpha'} [f]_{C_\ell^\alpha ((-(2r)^{2s},0] \times B_{(2r)^{1+2s}} \times \R^d)} + [h]_{C_\ell^{\alpha'} (Q_{2r})} + r^{-\alpha'} \|h\|_{C^0(Q_{2r})} \right),
\end{equation}
provided that $A_0 r^{\alpha'} \leq 1$.

If $A_0 r^{\alpha'} > 1$, we should further look at a smaller scale $\tilde r < r$ so that $A_0 \tilde r^{\alpha'} = 1$. In that case, the inequality \eqref{e:sch-a2} holds with $\tilde r$ instead of $r$, and for any cylinder $Q_{\tilde r}(z_0) \subset Q_r$. Taking into account Lemma~\ref{l:Calpha_local}, we get
\[
 [f]_{C_\ell^{2s+\alpha'} (Q_{r})} \le C \left({\tilde r}^{\alpha-2s-\alpha'} [f]_{C_\ell^\alpha ((-(2r)^{2s},0] \times B_{(2r)^{1+2s}} \times \R^d)} + [h]_{C_\ell^{\alpha'} (Q_{2r})} + {\tilde r}^{-\alpha'} \|h\|_{C^0(Q_{2r})} \right),
\]

Taking into account that $1/\tilde r = A_0^{1/\alpha'}$, we get
\begin{equation} \label{e:sch-a3}
[f]_{C_\ell^{2s+\alpha'} (Q_{r})} \le C \left( A_0^{\frac{2s+\alpha'-\alpha}{\alpha'}} [f]_{C_\ell^\alpha ((-(2r)^{2s},0] \times B_{(2r)^{1+2s}} \times \R^d)} + [h]_{C_\ell^{\alpha'} (Q_{2r})} + A_0 \|h\|_{C^0(Q_{2r})} \right).
\end{equation}

Combining \eqref{e:sch-a2} with \eqref{e:sch-a3} we obtain \eqref{e:loc-holder}.
\end{proof}

The Boltzmann kernel $K_{f,(t,x,v)}$ as in \eqref{e:kf} belongs (locally) to the class $\mathcal K$ provided that Assumption~\ref{a:hydro-assumption} holds. This follows from computations that are in the literature. Indeed, at least when $v$ stays in a bounded domain, we have
\begin{enumerate}
\item[i.] The symmetry of the Boltzmann kernel is immediate by construction. We see in the formula \eqref{e:kf} that $K_f(t,x,v,v+w) = K_f(t,x,v,v-w)$.
\item[ii.] The condition ii. in Definition~\ref{d:classK} tells us that the kernels $K \in \mathcal K$ are bounded in an averaged sense. It is a weaker condition than the more classical pointwise bound $K(w) \leq \Lambda |w|^{-d-2s}$. By a simple computation, we can verify that it is equivalent to any of the following two alternative formulations (see  \cite[Section 2.2]{imbert2016weak})
\begin{align*} 
 \int_{\R^d \setminus B_r} K(w) \dd w &\leq \Lambda r^{-2s}, \\
 \int_{B_{2r} \setminus B_r} K(w) \dd w &\leq \Lambda r^{-2s}.
\end{align*}
In each case, the inequality is supposed to hold for all $r>0$ and the value of $\Lambda$ may need to be adjusted by a dimensional constant when passing from one formulation to another.

It is the same inequality as in the assumption \eqref{e:bounded1} for Theorem \ref{t:local-holder}.

This boundedness assumption for a kernel $K$, together with the symmetry condition i. in Definition~\ref{d:classK} allows us to estimate the value of the integro-differential operator $\mathcal L_K f$ pointwise. See Lemma~\ref{l:pointwise_operator_bound} below.
\item[iii.] Using that $\gamma + 2s \in [0,2]$, the integral upper bound on item ii. in Definition~\ref{d:classK} holds for the Boltzmann kernel $K_f$ (at least locally) according to \cite[Lemma 4.3]{luis}. Indeed, that lemma says that for any $f: \R^d \to \R$,
\begin{align}
\label{e:upper_bound_for_Kf} \int_{\R^d \setminus B_r} K_f(v,v+w) \dd w &\lesssim \left( \int_{\R^d} f(v+w) |w|^{\gamma+2s} \dd w \right) r^{-2s}, \\
\intertext{applying Lemma~\ref{l:convolution-moments},}
\nonumber &\lesssim ((1+|v|)^{\gamma+2s} M_0 + E_0) r^{-2s}.
\end{align}
\item[iv.] The coercivity condition for $K_{f,(t,x,v)}$ is easier to verify than the usual coercivity estimates for the Boltzmann equation. This is because $K_{f,(t,x,v)}$ depends on the single variable $w \in \R^d$. We should think of the kernel $K_{f,(t,x,v)}$ as what we get from the original kernel $K_f$ by \emph{freezing coefficients}. The coercivity estimate iii. in Definition~\ref{d:classK} is a direct consequence of the existence of a cone of nondegeneracy described in \cite[Lemma 4.8]{luis} combined with the coercivity conditions from \cite{dyda2015regularity} or from \cite{chaker2019coercivity}.
\item[v.] The last nondegeneracy assumption, \textit{i.e.} item iv. in Definition~\ref{d:classK}, is a straight-forward consequence of the existence of a cone of nondegeneracy described in \cite[Lemma 4.8]{luis}.
\end{enumerate}

An important difficulty is apparent at this point: the constants $\Lambda$ in ii. and $\lambda$ in iii. and iv. deteriorate as $|v| \to \infty$. The kernel $K_{f,(t,x,v)}$ belongs to an ellipticity class only locally in $v$. In order to control the asymptotic behavior of all our regularity estimates, it will be important to establish precise asymptotics on the ellipticity of the kernel as $v \to \infty$. The same difficulty arises in regards to the assumptions for Theorem~\ref{t:local-schauder}. A change of variables will be described in Section~\ref{s:cov} that addresses this difficulty.

Now we state and prove the lemma mentioned above about pointwise bounds for $\mathcal L_K f$. For more applicability, we state it for kernels $K$ satisfying only i. and ii. in Definition~\ref{d:classK}, and that might even change sign. It is related to the   \cite[Estimate~(3.4)]{schauder}.

\begin{lemma} \label{l:pointwise_operator_bound}
Let $K: \R^d \to \R$ be a symmetric kernel (i.e. $K(w) = K(-w)$) so that 
\[ \int_{\R^d \setminus B_r} |K(w)| \dd w \leq \Lambda r^{-2s}.\]
Consider the integro-differential operator $\mathcal L_K$,
\[ \mathcal L_K f(v) = \PV \int_{\R^d} (f(v+w) - f(v)) K(w) \dd w.\]
If $f$ is bounded in $\R^d$ and $C^{2s+\alpha}$ at $v$ for some $\alpha \in (0,1)$, then
\[ \left| \mathcal L_K f(v) \right| \leq C \Lambda |f|_{C^0(\R^d)}^{\frac \alpha {2s+\alpha}} [f]_{C^{2s+\alpha}(v)}^{\frac {2s} {2s+\alpha}}.\]
The constant $C$ depends on dimension, $s$ and $\alpha$.
\end{lemma}

We use the standard notation $[\cdot]_{C^\alpha(v)}$ to denote the smallest value of $N \geq 0$ so that there exists a polynomial $q$ of degree strictly less than $\alpha$ so that $|f(v+w) - q(w)| \leq N|w|^\alpha$ for all $w \in \R^d$.
Note that $2s+\alpha$ may be larger than $2$ in Lemma~\ref{l:pointwise_operator_bound}.

\begin{proof}
The fact that $f$ is $C^{2s+\alpha}$ at the point $v$ means that
\begin{itemize}
\item $|f(v+w) - f(v)| \leq [f]_{C^{2s+\alpha}(v)} |w|^{2s+\alpha}$ if $2s+\alpha \in (0,1]$.
\item $|f(v+w) - f(v) - w \cdot \nabla f(v) | \leq [f]_{C^{2s+\alpha}(v)} |w|^{2s+\alpha}$ if $2s+\alpha \in (1,2]$.
\item $|f(v+w) - f(v) - w \cdot \nabla f(v) - \frac 12 w_i w_j \partial_{ij} f(v)| \leq [f]_{C^{2s+\alpha}(v)} |w|^{2s+\alpha}$ if $2s+\alpha \in (2,3]$.
\end{itemize}
For some $r>0$ to be determined later, we use the inequalities above to estimate the part of the integral where $w \in B_r$. Note that the term $w \cdot \nabla f(v)$ is odd in $w$. Since the kernel $K$ is symmetric, it vanishes in the principal value. We have
\begin{align*} 
 \bigg\vert \PV \int_{B_r} & (f(v+w) - f(v)) K(w) \dd w \bigg \vert \\ 
&\leq \int_{B_r} \left( [f]_{C^{2s+\alpha}(v)} |w|^{2s+\alpha} + \left\{ \frac 12 w_i w_j |\partial_{ij} f(v)| \right\}_{\text{if $2s+\alpha>2$}} \right) |K(w)| \dd w \\
&\lesssim \Lambda \left( [f]_{C^{2s+\alpha}(v)} r^\alpha + \left\{ |D^2 f(v)|  r^{2-2s} \right\}_{\text{if $2s+\alpha>2$}} \right).
\intertext{If $2s+\alpha > 2$, we use the (classical) interpolation inequality in the full space $|D^2 f(v)| \leq [f]_{C^{2s+\alpha}(v)}^{\frac 2 {2s+\alpha}}  |f|_{C^0(\R^d)}^{\frac{2s+\alpha-2}{2s+\alpha}}$,}
&\lesssim \Lambda \left( [f]_{C^{2s+\alpha}(v)} r^\alpha + \left\{ [f]_{C^{2s+\alpha}(v)}^{\frac 2 {2s+\alpha}}  |f|_{C^0(\R^d)}^{\frac{2s+\alpha-2}{2s+\alpha}} r^{2-2s} \right\}_{\text{if $2s+\alpha>2$}} \right).
\end{align*}
For the part of the integral $w \notin B_r$, we bound $|f(v+w) - f(v)|$ by $2|f|_{C^0(\R^d)}$. We get
\begin{align*} 
 \bigg\vert \int_{\R^d \setminus B_r} (f(v+w) - f(v)) K(w) \dd w \bigg \vert &\leq 2 |f|_{C^0(\R^d)} \int_{\R^d \setminus B_r} |K(w)| \dd w, \\
&\lesssim \Lambda |f|_{C^0(\R^d)} r^{-2s}.
\end{align*}
Adding up the two (or three if $2s+\alpha>2$) terms,
\[ |\mathcal L_K f(v)| \leq \Lambda \left( [f]_{C^{2s+\alpha}(v)} r^\alpha + \left\{ [f]_{C^{2s+\alpha}(v)}^{\frac 2 {2s+\alpha}}  |f|_{C^0(\R^d)}^{\frac{2s+\alpha-2}{2s+\alpha}} r^{2-2s} \right\}_{\text{if $2s+\alpha>2$}} + |f|_{C^0(\R^d)} r^{-2s} \right).\]
We finish the proof by choosing (the optimal) $r>0$ as
\[ r = \left( \frac {|f|_{C^0(\R^d)}}{[f]_{C^{2s+\alpha}(v)}} \right)^{\frac 1 {2s+\alpha}}. \qedhere\] 
\end{proof}

\section{The change of variables}
\label{s:cov}

\subsection{The change of variables}

The \emph{ellipticity} of Boltzmann's collision operator degenerates for large velocities.  This shows up, for example, in the weights of the well known coercivity estimates from \cite{gressmanstrainBETTERpaper}. Correspondingly, the constants $\lambda$ and $\Lambda$ in the assumptions of Theorem~\ref{t:local-holder} are bounded for $K_f$ only locally in $v$. Likewise, if we want to apply Theorem~\ref{t:local-schauder} to the Boltzmann kernel $K_f$ given in \eqref{e:kf}, the constants $\lambda$ and $\Lambda$ in Definition~\ref{d:classK} would only exist for a bounded set of velocities.

This is a major obstruction in order to obtain global regularity estimates using Theorems~\ref{t:local-holder} and \ref{t:local-schauder}. Moreover, global regularity estimates are crucial in order to carry out an iterative gain of regularity. The constants in Theorems~\ref{t:local-holder} and \ref{t:local-schauder} do not have an explicit dependence on the parameters $\lambda$ and $\Lambda$ in the assumptions. There is no hope to obtain a global regularity estimate unless we are able to apply these theorems with fixed values of the ellipticity parameters $\lambda$ and $\Lambda$ for all velocities in $\R^d$.

In this section, we describe a change of variables that resolves this difficulty. For any point $z_0 = (t_0,x_0,v_0) \in \R^{1+2d}$, we construct a function $\mathcal T_0$ that maps the kinetic cylinder $Q_1$ into a product of ellipsoids centered at $z_0$. Moreover, the function $f \circ \mathcal T_0$ satisfies a kinetic integro-differential equation whose kernel is elliptic with constants $\lambda$ and $\Lambda$ (either in the sense of Theorem~\ref{t:local-holder} or Theorem~\ref{t:local-schauder}) depending only on the constants $m_0$, $M_0$, $E_0$, $H_0$, $s$ and dimension, but \textbf{not} on $v_0$.

This change of variables allows us to turn our local regularity results (as in Theorems~\ref{t:local-holder} and \ref{t:local-schauder}) into global estimates with precise asymptotics as $|v| \to \infty$. It is a key tool for the proofs of the main results in this paper. It was motivated by a similar change of variables for the Landau equation from \cite{cameron2017global}.

In order to illustrate the significance of this change of variables, we show in Appendix~\ref{s:gressman} how it can be used to derive the global coercivity estimate with respect to the anisotropic distance obtained by Gressman and Strain in \cite{gressman2011global} and \cite{gressmanstrainBETTERpaper}.

Given $t_0 \in \R$, $x_0 \in \R^d$ and $v_0 \in \R^d$, we consider the transformed
function
\begin{equation} \label{e:change-of-variables}
 \bar f(t,x,v) := f(\bar t, \bar x, \bar v)
\end{equation}
with $(\bar t, \bar x, \bar v) =\mathcal{T}_0(t,x,v)$. The transformation $\mathcal T_0$ depends on the reference point $z_0 = (t_0,x_0,v_0) \in \R^{1+2d}$. If $|v_0| < 2$, we will simply take $\mathcal T_0 z := z_0 \circ z$. When $|v_0| \geq 2$, which is the important case, we define
\begin{equation}\label{e:T0}
\begin{aligned}
(\bar t, \bar x, \bar v) =\mathcal{T}_0(t,x,v) &:= \left(t_0+\frac{t}{|v_0|^{\gamma+2s} }, x_0 +
\frac{T_0x + t v_0}{|v_0|^{\gamma+2s} } , v_0 + T_0v\right) , \\
&= z_0 \circ (|v_0|^{-\gamma-2s} t, |v_0|^{-\gamma-2s} T_0 x, T_0 v)
\end{aligned}
\end{equation}
and  $T_0 : \R^d \to \R^d$ is the following transformation:
\begin{equation}\label{e:T0v}
 T_0(a v_0 + w) := \frac a {|v_0|} v_0 + w \qquad
\text{for all } w \perp v_0, a \in \R.
\end{equation}
Note that $T_0$ maps $B_1$ into an ellipsoid $E_1$ with radius
$1/|v_0|$ in the direction of $v_0$ and $1$ in the directions
perpendicular to $v_0$. For consistency, let us also define $T_0$ as the identity operator whenever $|v_0| < 2$. The following sets are
naturally associated with the change of variables. For
$ z_0 \in \R^{1+2d}$ and $r >0$, we consider
\[
  \mathcal{E}_r (z_0 ) = \mathcal{T}_0 (Q_r), \quad E_r (v_0) = v_0 +
  T_0 (B_r).
\]
The set of velocities $E_r (v_0)$ is an ellipsoid in $\R^d$. 
The linear operator $\mathcal{T}_0$ maps $Q_1$ into
$\mathcal{E}_1(z_0):= \mathcal{E}_1^{t,x} (z_0)\times E_1(v_0)$ where
$\mathcal{E}_1^{t,x}(z_0)=\{ (t_0+|v_0|^{-\gamma-2s} t, x_0 +
|v_0|^{-\gamma-2s} (T_0x + t v_0)):t \in [-1,0], x \in B_1\}$ is a
slanted cylinder. 

By a direct computation, we verify that if $f$ satisfies the Boltzmann
equation in $\mathcal E_1(z_0)$ then $\bar f$ solves the equation
\[ 
\partial_t \bar f + v \cdot \nabla_x \bar f = \mathcal{L}_{\bar K_f} \bar f + \bar h, \quad (t,x,v) \in Q_1
\] 
where
\begin{equation} \label{e:barKf}
 \bar K_f(t,x,v,v') = |v_0|^{-1-\gamma-2s} K_f(\bar t, \bar x, \bar v, v_0+T_0v')
\end{equation}
and 
\[ \bar h (t,x,v)= c_b |v_0|^{-\gamma-2s} f(\bar t,\bar x,\bar  v) (f \ast  |\cdot|^\gamma ) (\bar t,\bar x,\bar v).\] 
%------------------------------------------------------------------------

The point of this change of variables is to \emph{straighten up} the ellipticity of the Boltzmann kernel $K_f$. The following theorem says that we are able to apply the H\"older estimates of Theorem~\ref{t:local-holder} with uniform ellipticity constants to $\bar f$.

\begin{thm}[Change of variables - I] \label{t:change-of-vars}
  Let $z_0 = (t_0,x_0,v_0)$ and  $ \mathcal{E}_1(z_0) = \mathcal{E}_1^{t,x} (z_0) \times E_1 (v_0)$ be defined as above. Assume that Assumption~\ref{a:hydro-assumption} holds for all $(t,x) \in \mathcal{E}_1^{t,x} (z_0)$, and
  \begin{equation}
\label{e:extra-int}
 \text{if $\gamma<0$}, \qquad \sup_{v \in \R^d} \int_{B_1} f(v+u) |u|^{\gamma} \dd u \le C_\gamma . 
\end{equation}
Then the kernel $\bar K_f$ satisfies \eqref{e:nondeg1} (only if $s<1/2$), \eqref{e:nondeg2}, \eqref{e:bounded1}, \eqref{e:bounded2}, \eqref{e:cancellation1} and \eqref{e:cancellation2} (only if $s \ge 1/2$), with constants depending on $d,s,\gamma$ and $m_0$, $M_0$, $E_0$ and $H_0$ (and $C_\gamma$ if $\gamma<0$) only, uniformly with respect to $v_0$. 
\end{thm}
%------------------------------------------------------------------------

\begin{remark}
  Condition~\eqref{e:extra-int} is weaker than imposing  an $L^\infty$
  bound for $f$. Such a bound is proved in \cite{silvestre2016new} for
  solutions of the Boltzmann equation for $t>0$ and it only depends on
  the hydrodynamic quantities appearing in Assumption \eqref{a:hydro-assumption}
  when $\gamma +2s \in [0,2]$.
\end{remark}
\begin{remark}
Note that our computation works for $\gamma+2s \in [0,2]$. For values
of $s$ and $\gamma$ away of that range, we would need further
assumptions on either integrability of $f$ (for $\gamma +2s < 0$) or
higher moments (if $\gamma+2s>2$).
\end{remark}

We also have a corresponding result for the ellipticity assumptions of the Schauder-type estimates in Theorem~\ref{t:local-schauder}.

\begin{thm}[Change of variables - II] \label{t:change-of-vars-schauder}
  Let $z_0 = (t_0,x_0,v_0)$ with  $v_0 \in \R^d$ and $ \mathcal{E}_1(z_0) = \mathcal{E}_1^{t,x} (z_0) \times E_1 (v_0)$  be defined as above. Assume that Assumption~\ref{a:hydro-assumption} holds for all $(t,x) \in \mathcal{E}_1^{t,x} (z_0)$. Then, for every $z=(t,x,v) \in Q_1$, the kernel $\bar K_{f,z}(w) = \bar K_f(t,x,v,v+w)$ belongs to the class $\mathcal K$ of Definition~\ref{d:classK}. The constants $\lambda$ and $\Lambda$ in Definition~\ref{d:classK} depend on $d,s,\gamma$ and $m_0$, $M_0$, $E_0$ and $H_0$ (and $C_\gamma$ if $\gamma<0$) but \textbf{not} on $v_0$. 
\end{thm}

The proofs in this section largely consist in direct computations to verify the claims. However, rather involved manipulations of multiple integrals are needed, especially for the proof of the second cancellation condition in Lemma~\ref{l:cc2}.

\begin{remark} \label{r:large-v0}
When $|v_0|  \leq 2$, there is no need for a change of variables. All our ellipticity conditions hold for any arbitrary (but fixed) bounded set of velocities. The results of Theorems~\ref{t:change-of-vars} and \ref{t:change-of-vars-schauder} are already established in \cite{imbert2016weak} and \cite{schauder} for $|v_0| < 2$. Here, we need to prove them for $|v_0|\geq 2$. The purpose of the change of variables $\mathcal T_0$ is to analyze the asymptotic behavior of the estimates for large values of $|v_0|$. Thus, the case $|v_0| \geq 2$ is the important one. Yet, we define $\mathcal T_0$ for any value of $v_0$ for consistency. The change of variables $\mathcal T_0$ does not modify the equation at all when $|v_0|<2$.
\end{remark}

\subsection{Non-degeneracy conditions}

\label{s:nondegeneracy}

The nondegeneracy condition \eqref{e:nondeg1} and the coercivity condition \eqref{e:nondeg2} are a consequence of the existence of a cone of nondegeneracy described in Proposition~\ref{p:cone_of_nondegeneracy}.

Proposition~\ref{p:cone_of_nondegeneracy} describes a set of directions $A(v)$ depending on each point $v \in \R^d$, along which the kernel $K_f$ has a lower bound. Using the notation introduced in \cite{luis}, we call the cone of nondegeneracy $\cone(v)$. Here
\[ \cone(v) := \left\{ w : \frac w{|w|} \in A(v) \right\}. \]

Proposition~\ref{p:cone_of_nondegeneracy} says that for each value of $v \in \R^d$, the set of directions $A(v) \subset S^{d-1}$ is contained in a strip of width $\approx 1/(1+|v|)$ around the equator perpendicular to $v$, with measure $\gtrsim 1/(1+|v|)$, so that $K(v,v+w)  \gtrsim (1+|v|)^{1+\gamma+2s} |w|^{-d-2s}$ whenever $w$ belongs to $\cone(v)$.

Naturally, there is a cone of nondegeneracy for $\bar K_f$ corresponding to the cone of nondegenerate directions for $K_f$. Indeed, we write
\begin{align*} 
 \bar \cone(v) &= \{ w \in \R^d : T_0 w \in \cone(v_0 + T_0 v)\}, \\
 \bar A(v) &= \{\sigma \in S^{d-1} : T_0 \sigma / |T_0 \sigma| \in A(v_0 + T_0 v)\}.
\end{align*}
By construction, we have that $w \in \bar \cone(v)$ if and only if $w/|w| \in \bar A(v)$. Moreover, $\cone(v_0 + T_0 v) = T_0(\bar \cone(v))$.

The following lemma tells us that $\bar K_f$ has its nondegenerate directions $\bar A(v)$, and both its lower bound $\bar K_f$ and the volume of $\bar A(v)$ are independent of the center point $v_0$ of the change of variables.

\begin{lemma}[Transformed cone of nondegeneracy] \label{l:cone_transformed}
Let $f$ be a function such that Assumption~\ref{a:hydro-assumption} holds. Let $v_0 \in \R^d$ and $v \in B_2$, with $\bar A(v)$ and $\bar \cone(v)$ defined as above. Then
\begin{itemize}
\item $\bar K_f(v,v+w) \geq \lambda |w|^{-d-2s}$ whenever $w \in \bar \cone(v)$.
\item $|\bar A(v)| \geq \bar \mu$ for some $\bar \mu>0$ depending on the parameters of Assumption~\ref{a:hydro-assumption} and dimension, but \textbf{not} on $v_0$.
\end{itemize}
\end{lemma}

\begin{proof}
Proposition~\ref{p:cone_of_nondegeneracy} immediately implies the result of this lemma when $|v_0| \leq 2$. In order to prove it for $|v_0| \geq 2$, we need to analyze the interaction of the change of variables with the bounds in Proposition~\ref{p:cone_of_nondegeneracy}.

We first check the first item in the lemma. Pick $w$ such that $w \in \bar{\cone} (v)$, i.e. $T_0w \in \cone(v_0 + T_0 v)$. Then 
\begin{align*}
\bar K_f(v,v+w) = \frac{1}{|v_0|^{1+\gamma+2s}} K_f(v_0+T_0 v, v_0+T_0 v + T_0 w) \geq \lambda |T_0 w|^{-d-2s}
& \ge \lambda |w|^{-d-2s}.
\end{align*}
For the last inequality, we used the fact that $|T_0(w)| \le |w|$.

We are left with checking the second item. Note that $A(v_0 + T_0 v)$ and $\bar A(v)$ are subsets of $S^{d-1}$ related by the nonlinear map $\sigma \mapsto T_0 \sigma / |T_0 \sigma|$. In order to relate their volumes, we would have to make a computation involving the Jacobian of the map, which in this case is the determinant of the derivatives that act on the tangent space of $S^{d-1}$. This kind of computations are confusing to the best of us. So, instead, we opt to estimate the volume of $\bar A(v)$ through $\bar \cone(v)$. Indeed, the following elementary formula allows us to relate the volume of a set of directions with the volume of the corresponding cone. For any $R>0$, 
\[ | A(v_0 + T_0 v) |_{\mathcal H^{d-1}} = \frac d {R^d} | \cone(v_0 + T_0 v) \cap B_R| \qquad \text{and} \qquad  |\bar A(v) |_{\mathcal H^{d-1}} = \frac d {R^d} | \bar \cone(v) \cap B_R|.\]

Combining this formula with the estimate in Proposition~\ref{p:cone_of_nondegeneracy}, we have, for any $R>0$,
\begin{equation} \label{e:cv_nd1}
 | \cone(v_0 + T_0 v) \cap B_R| \geq \frac{R^d}d \mu (1+|v_0 + T_0v|)^{-1}.
\end{equation}

Let us recall the definition of $T_0$. Given any $w \in \R^d$, we write it as $w = a v_0 + w^\perp$ with $v_0 \cdot w^\perp = 0$. Then $T_0 w = a v_0 / |v_0| + w^\perp$. We want to estimate an upper bound for $a$ under the condition that $T_0 w \in \cone(v_0 + T_0 v) \cap B_R$.

According to the width condition in Proposition~\ref{p:cone_of_nondegeneracy}, if $T_0 w \in \cone(v_0 + T_0 v)$, we must have
\[ \left\vert \frac{T_0 w}{|T_0 w|} \cdot (v_0 + T_0 v) \right\vert \leq C_0.\]
Replacing the formula $T_0 w = a v_0 / |v_0| + w^\perp$, and recalling $w^\perp \cdot v_0 = 0$, we get
\[ \frac{|a|}{|T_0 w|} |v_0| \leq C_0 + |T_0 v| \leq C_0 + 2.\]
If in addition we know that $|T_0 w| \leq R$, we conclude that
\[ |a| \leq \frac{(C_0 + 2) R }{|v_0|}.\]
Now, for every $w$ such that $T_0 w \in \cone(v_0 + T_0 v) \cap B_R$, we have
\begin{align*} 
 |w| &= \sqrt{a^2 |v_0|^2 + |w^\perp|^2} , \\ 
 &\leq a |v_0| + |w^\perp| , \\ 
 &\leq R(C_0+3).
\end{align*}
Let us pick $R = (C_0+3)^{-1}$, which is a constant depending on the parameters of Assumption~\ref{a:hydro-assumption} and dimension only. We deduce that
\[ T_0^{-1} ( \cone(v_0 + T_0 v) \cap B_R) \subset B_1.\]
Therefore, 
\begin{align*} 
 |\bar \cone(v) \cap B_1| &\geq |T_0^{-1} ( \cone(v_0 + T_0 v) \cap B_R) | , \\ 
 &= ( \det T_0^{-1} ) \left\vert \cone(v_0 + T_0 v) \cap B_R \right \vert , \\ 
 &\geq |v_0| (1+|v_0 + T_0 v|)^{-1} \mu R^d / d \qquad \text{using \eqref{e:cv_nd1} and Proposition~\ref{l:convolution-C0}}, \\ 
 &\geq \bar \mu
\end{align*}
for some constant $\bar \mu>0$ depending only on the constants $\mu$ and $C_0$ of Proposition~\ref{p:cone_of_nondegeneracy} and dimension (and \textbf{not} on $v_0$).
\end{proof}

\begin{cor}[Non-degeneracy conditions for the H\"older estimates]\label{c:ndc}
When $f$ satisfies Assumption~\ref{a:hydro-assumption}, the kernel $\bar K_f$ satisfies
  \eqref{e:nondeg1} and also the inequality
\begin{equation} \label{e:local_coercivity}
 \iint_{B_1 \times B_1} (g(v') - g(v))^2 \bar K_f(v,v') \dd v' \dd v \geq \lambda \iint_{B_{1/2} \times B_{1/2}} \frac{(g(v') - g(v))^2}{|v'-v|^{d+2s}}  \dd v' \dd v
\end{equation}
for a parameter $\lambda>0$ depending on the constants in Assumption~\ref{a:hydro-assumption} (uniform with respect to $(t_0,x_0,v_0)$).
\end{cor}

\begin{proof}
The cone of nondegeneracy described in Lemma~\ref{l:cone_transformed} trivially implies \eqref{e:nondeg1} for $\bar K_f$. It also fulfills the assumption of Theorem~\ref{t:coercivity}, from which the inequality \eqref{e:local_coercivity} follows.
\end{proof}

Recall that for $z=(t,x,v)$,  $\bar K_{f,z}(w)$ denotes the kernel $\bar K_f(t,x,v,v+w)$.
\begin{cor}[Non-degeneracy conditions for the Schauder estimates]\label{c:ndc-schauder}
When $f$ satisfies Assumption~\ref{a:hydro-assumption},  the kernel $\bar K_{f,z}$ satisfies for any $z \in Q_1$  the coercivity conditions iii. and iv. in Definition~\ref{d:classK} for a parameter $\lambda>0$ uniform with respect to $(t_0,x_0,v_0)$.
\end{cor}

Note that the statement of Corollary~\ref{c:ndc-schauder} is not the same as \eqref{e:local_coercivity}. One is for the kernels $\bar K_{f,z} : \R^d \to \R^d$, depending on a single parameter $w \in \R^d$. The other is for the kernels $\bar K_f$ as a function of both $v$ and $v'$. The dependence of $t,x$ is irrelevant for either statement. The coercivity condition given in Corollary~\ref{c:ndc-schauder} is in some sense simpler than \eqref{e:local_coercivity} given in Corollary~\ref{c:ndc} since it applies to the kernel with \emph{frozen coefficients} (referring to standard terminology from elliptic PDEs). Theorem~\ref{t:coercivity} is sufficiently strong that it implies both Corollaries~\ref{c:ndc} and \ref{c:ndc-schauder}. One could alternatively justify Corollary~\ref{c:ndc-schauder} using coercivity results that are suitable for translation invariant integro-differential operators like the ones described in \cite{dyda2015regularity}.

\begin{proof} [Proof of Corollary~\ref{c:ndc-schauder}]
Let $\tilde K(w) := \bar K_{f,z}(w)$. The cone of nondegeneracy described in Lemma~\ref{l:cone_transformed} applies to $\bar K_f$, and therefore also to $\tilde K$. In this case, we have
\begin{equation} \label{e:ndc-o2}
 \tilde K(w) \geq \tilde \lambda |w|^{-d-2s} \qquad \text{ whenever } w \in \cone(z).
\end{equation}
Here, $\cone(z)$ is the cone of nondegeneracy for $\bar K_f$ at the point $z=(t,x,v)$. The \emph{frozen} kernel $\tilde K$ has this cone of nondegeneracy at every point $v \in \R^d$.

The coercivity condition iv. from Definition~\ref{d:classK} easily follows from the existence of the cone of non-degeneracy for $\tilde K$. 

The properties of condition of the cone of nondegeneracy described in Lemma~\ref{l:cone_transformed} imply the assumptions of Theorem~\ref{t:coercivity}. Therefore, we have
\[ \iint_{B_2 \times B_2} (\varphi(v') - \varphi(v))^2 \tilde K(v,v') \dd v' \dd v \gtrsim  \tilde \lambda \iint_{B_{1} \times B_{1}} (\varphi(v') - \varphi(v))^2 |v'-v|^{-d-2s} \dd v' \dd v.\]
Since the inequality \eqref{e:ndc-o2} is scale invariant, a standard scaling argument allows us to conclude \eqref{e:frozen_coercivity} for any $R>0$.
\end{proof}

\subsection{First boundedness condition}

%---------------------------------------------------------
\begin{lemma}[First boundedness condition]\label{l:bc1}
  Let us assume that $\gamma+2s \in [0,2]$ and an upper bound in mass and energy
  \[
    \int_{\R^d} f(v) \dd v \leq M_0, \qquad \int_{\R^d} f(v) |v|^2 \dd v \leq E_0.
  \]
  The kernel $\bar K_f$ from \eqref{e:barKf} satisfies \eqref{e:bounded1} with parameters depending on $M_0$, $E_0$,  $\gamma$, $s$ and $d$ only.
\end{lemma}
%---------------------------------------------------------

We start with the following computation. 
%-------------------------------------------------------------------
\begin{lemma} \label{lem:reduction}
  Let $v_0 \in \R^d \setminus B_2$ and $v \in B_2$. For any $r>0$, we have 
  \begin{align*}
&    \int_{\R^d \setminus B_r(v)} \bar K_f(v,v') \dd v' \lesssim \bar  \Lambda r^{-2s} \\
&\text{with} \qquad   \bar \Lambda =  |v_0|^{-\gamma-2s}  \int_{w\in \R^d} f(\bar v+w) \left( |v_0|^2 - \left( v_0 \cdot \frac{w}{|w|} \right)^2 + 1\right)^s |w|^{\gamma+2s} \dd w .
\end{align*}
\end{lemma}
%--------------------------------------------------------------------
\begin{proof}
In view of the defintion of $\bar K_f$, we can write
\begin{align*} 
 \int_{\R^d \setminus B_r(v)} \bar K_f(v,v') \dd v' &= \int_{\R^d \setminus E_r} |v_0|^{-1-\gamma-2s} K_f(\bar v,\bar v + u) \frac{\dd u}{\det T_0}, \\
&= |v_0|^{-\gamma-2s} \int_{\R^d \setminus E_r} K_f(\bar v,\bar v + u) \dd u,
\end{align*}
where $\bar v = v_0 + T_0v$ and $E_r=  T_0(B_r)$. The
set $E_r$ is an ellipsoid centered at the origin with radius
$r/|v_0|$ in the direction of $v_0$ and $r$ in the directions
perpendicular to $v_0$. 
%\textcolor{red}{Note that since we consider $v \in B_2$ and $\|T_0\| =1$, we also have $\bar v \in B_2(v_0)$.}\cyril{why do we mention this?}

Using \eqref{e:Kernel-approx}, we rewrite the expression above as
\begin{align*}
 \int_{\R^d \setminus B_r(v)} \bar K_f(v,v') \dd v' 
&\approx |v_0|^{-\gamma-2s} \int_{u \in \R^d \setminus E_r} |u|^{-d-2s} \left( \int_{w \perp u} f(\bar v+w) |w|^{1+\gamma+2s} \dd w \right) \dd u, \\
&= |v_0|^{-\gamma-2s} \int_{w\in \R^d} \left( \int_{\substack{u \perp w, \\ u \in \R^d \setminus E_r}} |u|^{-d+1-2s}  \dd u \right)  f(\bar v+w) |w|^{\gamma+2s} \dd w.
\end{align*}
We used the fact that 
\begin{equation}\label{e:cdv}
\int_u \left\{ \int_{w\perp u} (\dots) \dd w \right\} \dd u = \int_w \left\{ \int_{u\perp w} (\dots) \frac{|u|}{|w|} \dd u \right\} \dd w.
\end{equation}

In order to estimate the integral in the inner factor, we analyze the intersection of the ellipsoid $E_r$ with the hyperplane $\{u:u \perp w\}$. This is of course a $(d-1)$-dimensional ellipsoid whose dimensions are computable. Its smallest radius $\rho$ equals
\begin{equation}\label{eq:rho}
 \rho := \frac{r}{\sqrt{|v_0|^2 \left( 1- \left( \frac{v_0 \cdot w}{|v_0| |w|} \right)^2 \right) + \left( \frac{v_0 \cdot w}{|v_0| |w|} \right)^2 }}.
\end{equation}
Indeed, it is an elementary planar computation: for $u \in \partial E_r$ with $|u|=\rho$ and $u \perp w$, we write $u = (u \cdot \hat v_0) \hat v_0 + u_1$ with $u_1 \perp v_0$ and $\hat v_0 = v_0 / |v_0|$. In particular, $\rho^2 = (u \cdot \hat v_0)^2 + |u_1|^2$. From the definition of $E_r$ and our choice of $u$, we know that the vector $v = (u\cdot \hat v_0) |v_0| \hat v_0 + u_1$ is of norm $r$. We now write
\begin{align}
\nonumber
  r^2 &= (u \cdot \hat v_0)^2 |v_0|^2 + |u_1|^2 \\
      &= (u \cdot \hat v_0)^2 (|v_0|^2-1) + \rho^2.
        \label{e:rho-sigma}
\end{align}
Use next the fact that $u \perp w$ and get $\rho^{-2} (u \cdot \hat v_0)^2 + (\hat w \cdot \hat v_0)^2 = 1$ with $\hat w = w / |w|$. In particular,
\[ (u \cdot \hat v_0)^2 = \rho^2 (1 - (\hat w \cdot \hat v_0)^2 ).\]
Combining the formulas for $r^2$ and $(u \cdot \hat v_0)^2$ yields \eqref{eq:rho}.

Therefore
\begin{align*}
  \int_{\substack{u \perp w, \\ u \in \R^d \setminus E_r}} |u|^{-d+1-2s}  \dd u
  & \leq \int_{\substack{u \perp w, \\ u \in \R^d \setminus B_{\rho}}} |u|^{-d+1-2s}  \dd u \lesssim \rho^{-2s} \\
  & = r^{-2s} \left( |v_0|^2 \left( 1- \left( \frac{v_0 \cdot w}{|v_0| |w|} \right)^2 \right) + \left( \frac{v_0 \cdot w}{|v_0| |w|} \right)^2 \right)^s \\
  & \leq r^{-2s} \left( |v_0|^2 - \left( v_0 \cdot \frac{ w}{|w|} \right)^2  + 1 \right)^s .
\end{align*}
Substituting in our previous formula, we get the desired result.
\end{proof}
We next aim at estimating the constant $\bar \Lambda$ appearing in Lemma~\ref{lem:reduction}. 
% ---------------------------------------------------------------
\begin{lemma} \label{lem:barlambda}
  Let $f : \R^d \to [0,\infty)$, $v_0 \in \R^d \setminus B_1$, $|v - v_0|\leq 2$ and $s \geq 0$, then for all $w \in \R^d$,
 \begin{align*}
 |v_0|^{-\gamma-2s} \int_{w \in \R^d} f(v+w) \left( |v_0|^2 - \left( v_0 \cdot \frac{w}{|w|} \right)^2 + 1 \right)^s |w|^{\gamma+2s} \dw
   &\le C \int_{\R^d} f(\tilde v)( 1+|\tilde v|^{2s} + |\tilde v|^{\gamma+2s}) \dd \tilde v \\
   &\le C (M_0+E_0)
\end{align*}
for a constant $C$ depending on $s$ and dimension only (not on $v_0$).
\end{lemma}

\begin{proof}
It is enough to prove
\begin{align} \label{e:ub1}
 |v_0|^{-\gamma-2s} \left( |v_0|^2 - \left( v_0 \cdot \frac{w}{|w|} \right)^2 + 1 \right)^s |w|^{\gamma+2s}
 &\lesssim  1+|v_0+w|^{2s} + |v_0+w|^{\gamma+2s}, \\
  &\lesssim 1+|v+w|^{2s} + |v+w|^{\gamma+2s}.
    \nonumber
\end{align}

  The second of these inequalities follows simply because $|v-v_0| \le 2$. We need to prove the first one.
Note that 
\[ \left( |v_0|^2 - \left( v_0 \cdot \frac{w}{|w|} \right)^2 + 1 \right)^s \lesssim \left( |v_0|^2 - \left( v_0 \cdot \frac{w}{|w|} \right)^2 \right)^s + 1. \]
Moreover, 
\[
  |v_0|^{-\gamma-2s} |w|^{\gamma+2s} \lesssim |v_0|^{-\gamma-2s} (|v_0+w|^{\gamma+2s}+|v_0|^{\gamma+2s}) \leq 1+|v_0+w|^{\gamma+2s},
\]
using that $|v_0|>1$. We are left with studying an upper bound for 
\begin{equation} \label{e:ub-aux1} 
|v_0|^{-\gamma-2s} \left( |v_0|^2 - \left( v_0 \cdot \frac{w}{|w|} \right)^2 \right)^s |w|^{\gamma+2s}.
\end{equation}
It is convenient to write $w = \alpha v_0/|v_0| + b$, where $\alpha \in \R$ and $b \in \R^d$ is perpendicular to $v_0$. With this notation
\[ |v_0|^2 - \left( v_0 \cdot \frac{w}{|w|} \right)^2 = \frac{|v_0|^2 |b|^2}{|w|^2}.\]
Thus, the expression in \eqref{e:ub-aux1} equals $|v_0|^{-\gamma} |b|^{2s} |w|^{\gamma}$. We need to study an upper bound for it. The key observation here is that $|b|$ satisfies the two inequalities $|b| \leq |w|$ and $|b| \leq |v_0+w|$, both of which are immediate from its definition.

If $|w| \leq 2|v_0|$ and $\gamma \geq 0$, we have $|v_0|^{-\gamma} |b|^{2s} |w|^{\gamma} \lesssim |b|^{2s} \lesssim |v_0+w|^{2s}$ and \eqref{e:ub1} follows. The same conclusion holds if $|v_0| \leq 2|w|$ and $\gamma \leq 0$.

If $\gamma \geq 0$ and $2|v_0|\leq |w|$, we use that $|b|\leq |v_0+w|$ and $|w|^\gamma \lesssim (|v_0+w|^\gamma+|v_0|^\gamma )$. Thus,
\begin{align*} 
 |v_0|^{-\gamma} |b|^{2s} |w|^\gamma &\lesssim |v_0|^{-\gamma} |v_0+w|^{2s} (|v_0+w|^\gamma+|v_0|^\gamma ) ,\\
 &\lesssim |v_0+w|^{2s} + |v_0+w|^{\gamma+2s} .
\end{align*}

We are left to analyze the case $\gamma < 0$ and $|v_0| > 2|w|$. Using that $|b|\leq |w|$ and $\gamma + 2s \geq 0$, we have
\[ |v_0|^{-\gamma} |w|^\gamma |b|^{2s} \leq |v_0|^{-\gamma} |w|^{\gamma+2s} \lesssim |v_0|^{2s} \lesssim |v_0+w|^{2s}.\]

We conclude that the inequality \eqref{e:ub1} holds in all cases. The proof is now complete. 
\end{proof}

From Lemmas~\ref{lem:reduction} and \ref{lem:barlambda}, we can derive Lemma~\ref{l:bc1}.
\begin{proof}[Proof of Lemma~\ref{l:bc1}]
We get from Lemma~ \ref{lem:barlambda} that
  \[ |v_0|^{-\gamma-2s}  \int_{w\in \R^d} f(v+w) \left( |v_0|^2 - \left( v_0 \cdot \frac{w}{|w|} \right)^2 + 1 \right)^s |w|^{\gamma+2s} \dd w \leq C \int_{\R^d}f(\tilde v) (1+ |\tilde v|^{2s}+|\tilde v |^{\gamma+2s})\dd \tilde v. \]
Combining this estimate with  Lemma~\ref{lem:reduction} yields
  \begin{align*}
    \int_{\R^d \setminus B_r(v)} \bar K_f(v,v') \dd v' & \lesssim \bar  \Lambda r^{-2s} \\
                                                       & \lesssim \left( \int_{\R^d}f(\tilde v) (1+ |\tilde v|^{2s}+|\tilde v|^{\gamma+2s})\dd \tilde v \right) r^{-2s} \\
    \intertext{then use the fact that $\gamma +2s \in [0,2]$,}
    & \lesssim \left( \int_{\R^d}f(\tilde v) (1+ |\tilde v|^2)\dd \tilde v \right) r^{-2s} \\
    & \lesssim (M_0+E_0) r^{-2s}. \qedhere
  \end{align*}
\end{proof}
% -------------------------------------------------------------

Lemma~\ref{l:bc1} is phrased in terms of the condition \eqref{e:bounded1} for Theorem~\ref{t:local-holder}. As far as Theorem~\ref{t:change-of-vars-schauder} is concerned, we remark that it is the same condition as the second item in Definition~\ref{d:classK}. We rephrase it in the following corollary. Recall that $\bar K_{f,z}(w)$ denotes $\bar K_f(t,x,v,v+w)$ for $z = (t,x,v)$.
\begin{cor}[Item ii. in Definition~\ref{d:classK} for $\bar K_f$] \label{c:bounded_for_schauder}
Let $f \in [\tau,T]\times \R^d \times \R^d \to [0,\infty)$ and $K_f$ be given by \eqref{e:kf}. Then, for any $z \in Q_1$,
\[ \int_{B_r} \bar K_{f,z}(w) |w|^2 \dd w \leq C \left( \int_{\R^d} (1+|\tilde v|^2) f(\bar t, \bar x, \tilde v) \dd \tilde v \right) r^{2-2s}\]
with $z=(t,x,v)$ and $(\bar t,\bar x,\bar v) = \mathcal{T}_0 (z)$. 
 The constant $C$ in the right hand side depends on $\gamma$, $s$ and dimension only.
\end{cor}

\begin{proof}
In terms of the notation $\bar K_{f,z}$, Lemmas~\ref{lem:reduction} and \ref{lem:barlambda} say exactly that
\begin{align*} 
 \int_{\R^d \setminus B_r} \bar K_{f,z}(w) \dd w &\leq C \left( \int_{\R^d} (1+|v|^{2s}+|v|^{\gamma+2s}) f(\bar t, \bar x, v) \dd v \right) r^{-2s}, \\
&\leq C \left( \int_{\R^d} (1+|v|^2) f(\bar t, \bar x,v) \dd v \right) r^{-2s}
\end{align*}
with $(\bar t,\bar x,\bar v) = \mathcal{T}_0 (t,x,v)$. 
The bound for the integral of the kernel on the complement of any ball $\R^d \setminus B_r$, is equivalent to the bound of the integral of the same kernel times $|w|^2$ on any ball. Indeed, the inequality above implies that
\[ \int_{B_{2r} \setminus B_{r}} \bar K_{f,z}(w) |w|^2 \dd w \lesssim \left( \int_{\R^d} (1+|v|^2) f(\bar t, \bar x, v) \dd v \right) r^{2-2s}.\]
Applying this inequality to dyadic rings for $\tilde r = r/2, r/4, r/8, \dots$ and summing the resulting estimates, we conclude the proof.
\end{proof}

\subsection{Second boundedness condition}

%-----------------------------------------------------
\begin{lemma}[Second boundedness condition]\label{l:bc2}
Let us assume that $\gamma+2s \in [0,2]$ and let the upper bound in the mass and energy hold
\[ \int_{\R^d} f(v) \dd v \leq M_0, \qquad \int_{\R^d} f(v) |v|^2 \dd v \leq E_0.\]
If $\gamma+s<0$, we also assume \eqref{e:extra-int}.

Then, the kernel $\bar K_f$ from \eqref{e:barKf} satisfies \eqref{e:bounded2} with parameters depending on $M_0$, $E_0$, $\gamma$, $s$, $d$ and $C_\gamma$ (in case $\gamma+s<0$).
\end{lemma}

%-----------------------------------------------------
\begin{proof}
For $|v_0| < 4$ (or any other fixed constant) the result follows from \cite[Lemma 3.5]{imbert2016weak}. Here, we concentrate on the case $|v_0| \geq 4$.

Given $v' \in B_2$, recall that $\bar v' = v_0 + T_0 v'$ and write
\begin{align*}
\int_{\R^d \setminus B_r (v')} \bar K_f(v,v') \dv & \leq \frac1{|v_0|^{\gamma + 2s}} \int_{\R^d \setminus E_r(\bar v')} K_f(\bar v, \bar v') \dd \bar v \\
& \simeq \frac1{|v_0|^{\gamma + 2s}} \int_{\R^d \setminus E_r(\bar v')} |\bar v'-\bar v|^{-d -2s} \left\{ \int_{w \perp (\bar v'- \bar v)} f(\bar v+w) |w|^{\gamma + 2s +1} \dw \right\} \dd \bar v, \\
\intertext{we use polar coordinates for $\bar v$: $\bar v - \bar v' = \rho \sigma$; we denote $r_\sigma$ to be $\max \{\rho  : \rho \sigma \in E_r\}$,}
& = \frac 1{|v_0|^{\gamma + 2s}} \int_{S^{d-1}} \int_{r_\sigma}^\infty \rho^{-d -2s} \left\{ \int_{w \perp \sigma} f(\bar v'+ \rho \sigma +w) |w|^{\gamma + 2s +1} \dw \right\} \rho^{d-1} \dd \rho \dd \sigma, \\
\intertext{writing $u = \rho \sigma + w$, we have $\dd w \dd \rho = \dd u$ and $\rho = u \cdot \sigma$,}
& \lesssim \frac 1{|v_0|^{\gamma + 2s}} \int_{S^{d-1}} \int_{\{u:u \cdot \sigma > r_\sigma\}} f(\bar v'+ u) |u|^{\gamma + 2s +1} (u \cdot \sigma)^{-1-2s} \dd u \dd \sigma, \\
& = \frac 1{|v_0|^{\gamma + 2s}} \int_{S^{d-1}} \int_{\{u:u \cdot \sigma > r_\sigma\}} f(\bar v'+ u) |u|^{\gamma} \left(\frac{u}{|u|} \cdot \sigma \right)^{-1-2s} \dd u \dd \sigma, \\
& = \frac 1{|v_0|^{\gamma + 2s}} \int_{\R^d} f(\bar v'+ u) |u|^{\gamma} \left( \int_{\{\sigma : u \cdot \sigma > r_\sigma\}} \left(\frac{u}{|u|} \cdot \sigma \right)^{-1-2s} \dd \sigma \right) \dd u.
\end{align*}

We computed above $r_\sigma = \rho$, see \eqref{e:rho-sigma} with $u = \rho \sigma = r_\sigma \sigma$. In particular, this implies
\begin{align*}
  r_\sigma &= \frac{r}{\sqrt{(\sigma \cdot v_0)^2 (1-|v_0|^{-2}) +1}}.\\
\intertext{For $|v_0|$ ``large'' (as a matter of fact $|v_0|\ge 2$), we have}
  r_\sigma  & \approx \frac{r}{\sqrt{1 + (\sigma \cdot v_0)^2}}.
\end{align*}
Thus, we are left with the inequality,
\begin{align*}
  \int_{\R^d \setminus B_r (v')} \bar K_f(v,v') \dv  &\lesssim \frac 1{|v_0|^{\gamma + 2s}} \int_{\R^{d}} f(\bar v'+ u) |u|^{\gamma} \left( \int_{\left\{\sigma : u \cdot \sigma >\frac{r}{\sqrt{1 + (\sigma \cdot v_0)^2}} \right\} } \left(\frac{u}{|u|} \cdot \sigma \right)^{-1-2s} \dd \sigma \right) \dd u.
\end{align*}
We need to estimate the inner integral now. For that, it is essential to understand the smallest value of $\sigma \cdot u/|u|$ in the set $\left\{\sigma : u \cdot \sigma >\frac{r}{\sqrt{1 + (\sigma \cdot v_0)^2}} \right\}$. Let us write $e = u/|u|$ and $v_0 = a e + b$. We see that for every $\sigma$ in the domain of integration we have
\[ (e \cdot \sigma) \sqrt{ 1 + (e \cdot \sigma)^2 a^2 + |b|^2} > r / |u|.\]
Therefore, either $(e\cdot \sigma) \sqrt{1+|b|^2} \gtrsim r / |u|$ or $(e\cdot \sigma) |e\cdot \sigma| |a| \gtrsim r / |u|$. In other words,
\[ e \cdot \sigma \gtrsim \min\left( \frac{r}{|u| (1+|b|^2)^{1/2}} , \left( \frac{ r  }{ |a| |u| } \right)^{1/2} \right) =: \rho_0.\]
Therefore
\begin{align*} 
  \int_{\left\{u \cdot \sigma >\frac{r}{\sqrt{1 + (\sigma \cdot v_0)^2}} \right\}}  \left(\frac{u}{|u|} \cdot \sigma \right)^{-1-2s} \dd \sigma
  &\lesssim \int_{\{ e \cdot \sigma  \gtrsim \rho_0\}} \left(e \cdot \sigma \right)^{-1-2s} \dd \sigma, \\
&\lesssim \rho_0^{-2s},\\
&= \min\left( \frac{r}{|u| (1+|b|^2)^{1/2}} , \left( \frac{ r  }{ |a| |u| } \right)^{1/2} \right)^{-2s}, \\
&\leq r^{-2s} |u|^{2s} (1+|b|^2)^s + r^{-s} |u|^s |a|^s,\\
&= r^{-2s} |u|^{2s} \left(1+|v_0|^2 - \frac{(v_0 \cdot u)^2}{|u|^2} \right)^s + r^{-s} |u|^s \left\vert \frac{u}{|u|} \cdot v_0 \right\vert^s.
\end{align*}
Thus, we are left with
\[   \int_{\R^d \setminus B_r (v')} \bar K_f(v,v') \dv \lesssim I_1 + I_2,\]
where
\begin{align*}
I_1 &= r^{-2s} |v_0|^{-\gamma - 2s} \int_{\R^d} f(\bar v'+ u) |u|^{\gamma+2s} \left(1 + |v_0|^2 - \frac{(v_0 \cdot u)^2}{|u|^2} \right)^s \dd u, \\
I_2 &= r^{-s} |v_0|^{-\gamma - s} \int_{\R^d} f(\bar v'+ u) |u|^{\gamma+s} \dd u.
\end{align*}

The term $I_1$ is bounded thanks to Lemma~\ref{lem:barlambda}.

The term $I_2$ is lower order in the sense that it has a smaller power of $r$. Note that only the values of $r \in (0,2)$ are relevant for \eqref{e:bounded2}, since for larger values of $r$ the domain of integration is empty. We still need to make sure that the factor multiplying $r^{-s}$ in the definition of $I_2$ is bounded independently of $v_0$.

When $\gamma+s \geq 0$, we apply Lemma \ref{l:convolution-moments},
\[ \int_{\R^d} f(\bar v'+ u) |u|^{\gamma+s} \dd u \lesssim (1+|\bar v'|)^{\gamma+s} \approx (1+|v_0|)^{\gamma+s}\]
yielding
\[ I_2 \lesssim r^{-s}.\]

The case $\gamma+s < 0$ is more involved. We split the integral in $I_2$ into three subdomains
\[ 
D_1 = \R^d \setminus B_{|v_0|/4}, \qquad D_2 = B_{|v_0|^{-1}}, \qquad D_3 = \R^d \setminus (D_1 \cup D_2).
\]
We estimate each subintegral separately. In $D_1$, we have $|u|\gtrsim |v_0|$. Thus,
\[
\int_{D_1} f(\bar v'+ u) |u|^{\gamma+s} \dd u \leq  M_0 |v_0|^{\gamma+s}.
\]
In $D_2$, we have $|u|^{\gamma+s} \le |u|^\gamma |v_0|^{-s}$. Therefore
\begin{align*}
\int_{D_2} f(\bar v'+ u) |u|^{\gamma+s} \dd u &\lesssim  |v_0|^{-s} \int_{D_2} f(\bar v'+ u) |u|^{\gamma} \dd u, \\
&\leq |v_0|^{-s} (M_0 + C_\gamma), \\
&\leq |v_0|^{\gamma+s} (M_0 + C_\gamma) \qquad \qquad \text{ using that $\gamma+2s \ge 0$.}
\end{align*}
We estimate the integral on $D_3 $ using the upper bound on the energy: $E_0$. We use the fact that in $D_3$, we have on the one hand $|\bar v' + u | \gtrsim |v_0|$ and on the other hand $|u| \ge |v_0|^{-1}$.  In particular, 
\begin{align*}
  \int_{D_3} f(\bar v'+ u) |u|^{\gamma+s} \dd u &\le E_0 |v_0|^{-2} |v_0|^{-\gamma -s} ,\\
&\lesssim E_0 |v_0|^{\gamma+s} \qquad \text{using that $\gamma+2s \geq 0$ and $2s \leq 2$.}
\end{align*}

Adding the three terms together, we conclude that $I_2 \leq C r^{-s}$, for a constant $C$ independent of $v_0$. The proof is now complete.
\end{proof}

\subsection{First cancellation condition}

The cancellation condition \eqref{e:cancellation1} involves an integral in the principal value sense. We have to be careful when we compute a change of variables of such an integral. There is a delicate cancellation that takes place and we have to make sure that the change of variables does not cause an imbalance around the origin that would ruin this cancellation. The following lemma is precisely what we need in order to carry out the rest of our computations.

\begin{lemma}[Modified principal value]\label{lem:cancel}
 Assuming that $D^2 f \in L^1(\R^d, (1+|v|)^{\gamma+2s}\dv)$, we have 
\begin{equation}\label{e:cancel}
\lim_{R\to 0^+} \int_{B_R \setminus E_R} (K_f(\bar v,\bar v + w) -K_f(\bar v + w,\bar v)) \dw = 0.
\end{equation}
\end{lemma}

The result of Lemma~\ref{lem:cancel} is certainly to be expected by common sense. However, its rigorous verification requires some work. We prove it under the condition that $D^2 f \in L^1(\R^d, (1+|v|)^{\gamma+2s}\dv)$. It holds under much more general conditions on the function $f$. However, due to the kind of solutions that we work with, the scope of Lemma~\ref{lem:cancel} is enough for the purpose of this article. A proof of a version of Lemma~\ref{lem:cancel} under less restrictive assumptions on the function $f$ would require considerably more work. Note that Lemma~\ref{lem:cancel} is merely a qualitative result. There is no estimate resulting from this lemma in terms of the weighted $L^1$ norm of $D_v^2 f$. We invite the reader to skip its proof, that we include below for completeness.

\begin{proof}[Proof of Lemma~\ref{lem:cancel}]
For all $\bar v = v_0 + T_0 v$ with $v \in B_1$, we expand the integral in terms of the formula \eqref{e:A} and symmetrize it.
\begin{align*} 
(I) :=&  \int_{B_R \setminus E_R} (K_f(\bar v,\bar v + w) -K_f(\bar v + w,\bar v)) \dw , \\ 
 =& - \frac12 \int_{B_R \setminus E_R} \int_{u : u \perp w} (\delta^2 f) (\bar v+u, w)  A (|w|,|u|)  |w|^{-d-2s} |u|^{\gamma+2s+1} \dd u \dd w 
\end{align*}
with $(\delta^2 f) (v,w) =f(v+w) + f(v-w) -2f(v)$ and $A$ is a bounded. We express this second order differential quotient using  $D^2_v f$ using the elementary formula
\[ (\delta^2 f) (v,w) = \int_{-1}^1 (1-|\tau|)  D^2f(v+\tau w) w \cdot w  \dd \tau.\]

Thus, we bound the integral in terms of $\|D^2_v f\|_{L^1(\R^d, |v|^{\gamma+2s} \dd v)}$.
\[
  (I) \lesssim \int_{B_R \setminus E_R}  \int_{u \perp w} \int_{-1}^1  |D^2_vf (\bar v+u + \tau w)| |w|^{-d-2s+2}  |u|^{\gamma+2s+1} \dd \tau \dd u \dd w.
\]

For each value of $w \in \R^d$, we make the change of variables $(\tau,u) \mapsto z = u +\tau w$; in particular $\dd z = |w| \dd u \dd \tau$ and $\tau = (z \cdot w)/|w|^2$ and $|u| \le |z|$. Therefore, we estimate the integral above by
\begin{align*} 
 (I) &\lesssim \int_{B_R \setminus E_R} \int_{-|w|^2 < z \cdot w < |w|^2} |D^2_vf (\bar v+z)| |w|^{-d-2s+1} |z|^{\gamma+2s+1} \dd z \dd w, \\
\intertext{switching the order of integration,}
&= \int_{z \in \R^d} |z|^{\gamma+2s+1} |D^2_v f(\bar v+z)| \left\{\int_{\substack{w \in B_R, \\ -|w|^2 < z \cdot w < |w|^2}} |w|^{1-d-2s} \dd w \right\} \dd z, \\
  \intertext{we claim that,}
     &\lesssim R^{2-2s} \int_{z \in \R^d} |D^2_vf (\bar v+z)| |z|^{\gamma+2s} \dz.
\end{align*}

In order to justify the last inequality, observe that the set $\{w :-|w|^2 < z \cdot w < |w|^2\}$ is the complement of two balls of radius $|z|/2$ centered at $z/2$ and $-z/2$ respectively. The intersection of that set with the ball $B_R$ is of volume $\lesssim R^{d+1} |z|^{-1}$. Thus, the following inequality follows by an elementary computation.
\begin{equation}
  \label{e:2ballsout}
  \int_{\substack{w \in B_R, \\ -|w|^2 < z \cdot w < |w|^2}} |w|^{1-d-2s} \dd w \lesssim R^{2-2s} |z|^{-1}.
  \end{equation}

Since $s\in (0,1)$, $R^{2-2s}$ converges to zero as $R \to 0$, which concludes the proof.
\end{proof}

%---------------------------------------------------------
\begin{lemma}[First cancellation condition]\label{l:cc1}
Let $f: \R^d \to \R$ be a function so that Assumption~\ref{a:hydro-assumption} holds. If $\gamma<0$, we assume \eqref{e:extra-int} in addition. Then, the kernel $\bar K_f$ given in \eqref{e:barKf} satisfies \eqref{e:cancellation1} with parameters depending on $M_0$, $E_0$, $\gamma$, $s$, $d$, and also $C_\gamma$ if $\gamma<0$. More precisely, for $v \in B_2$,
\begin{align*} 
\left| \PV \int_{\R^d} (\bar K_f(v,v') - \bar K_f(v',v) ) \dv' \right| &\lesssim  |v_0|^{-\gamma -2s}  \int_{\R^d} f(\bar v') |\bar v'-\bar v|^\gamma \dd v', \\
&\lesssim \begin{cases}
(M_0 + E_0) |v_0|^{-2s} & \text{if } \gamma \in [0,2], \\
(M_0 + C_\gamma) |v_0|^{-\gamma-2s} & \text{if } \gamma < 0.
\end{cases}
\end{align*}
\end{lemma}

The inequality in Lemma~\ref{l:cc1} implies \eqref{e:cancellation1} because the tail of the integral (\textit{i.e.} $|v'-v|>1/4$) is bounded by Lemmas~\ref{l:bc1} and \ref{l:bc2}.

%---------------------------------------------------------
\begin{proof}[Proof of Lemma~\ref{l:cc1}]
When $|v_0|<2$, the result is proved in \cite[Lemma 3.6]{imbert2016weak}, and it corresponds to the classical cancellation lemma. Here, we focus on the case $|v_0| \geq 2$.

As before, for $v \in B_2$, we write $\bar v = v_0 + T_0 v \in B_2 (v_0)$. Using Lemma~\ref{lem:cancel}, we compute
\begin{align*}
\PV \int_{\R^d} (\bar K_f(v,v') &- \bar K_f(v',v) ) \dv'  \\ &= |v_0|^{-\gamma -2s-1} \PV \int_{\R^d} (K_f(v_0+T_0v,v_0+T_0 v') -K_f(v_0+T_0v',v_0+T_0 v)) \dv' \\
 & =  |v_0|^{-\gamma -2s-1} \lim_{R \to 0^+} \int_{|z|\ge R} (K_f(\bar v,\bar v + T_0 w) -K_f(\bar v + T_0 w,\bar v)) \dw \\
                                & =  |v_0|^{-\gamma -2s} \lim_{R \to 0^+} \int_{\R^d \setminus E_R} (K_f(\bar v,\bar v + \bar w) -K_f(\bar v + \bar w,\bar v)) \dd \bar w \\
                                & =  |v_0|^{-\gamma -2s} \lim_{R \to 0^+} \int_{|\bar w|\ge R} (K_f(\bar v,\bar v + \bar w) -K_f(\bar v + \bar w,\bar v)) \dd \bar w \\
                                & =  |v_0|^{-\gamma -2s} \PV \int_{\R^d} (K_f(\bar v,\bar v + \bar w) -K_f(\bar v + \bar w,\bar v)) \dd \bar w. 
\end{align*}
We now  use \cite[Lemma~3.6]{imbert2016weak} and get
\[ 
 \left| \PV \int_{\R^d} (\bar K_f(v,v') - \bar K_f(v',v) ) \dv' \right| \le C |v_0|^{-\gamma -2s}  \int_{\R^d} f(\bar v') |\bar v'-\bar v|^\gamma \dz.
\]
If $\gamma <0$, we use \eqref{e:extra-int} and get 
\[ 
\left| \PV \int_{\R^d} (\bar K_f(v,v') - \bar K_f(v',v) ) \dv' \right| \lesssim |v_0|^{-\gamma -2s} (M_0 + C_\gamma).
\]
while for $\gamma >0$, we estimate it using Lemma \ref{l:convolution-moments},
\[ 
\left| \PV \int_{\R^d} (\bar K_f(v,v') - \bar K_f(v',v) ) \dv' \right| \lesssim |v_0|^{-\gamma -2s} (E_0 + M_0 (1+|\bar v|)^\gamma ) \lesssim (M_0 + E_0) |v_0|^{-2s}.
\] 
The proof is now complete. 
\end{proof}

\subsection{Second cancellation condition}

Like in the first cancellation condition, the second cancellation \eqref{e:cancellation2} also involves the principal value of an integral. In this case, the following technical lemma is the one that ensures that we can perform the change of variables.
%-------------------------------------------------------------
\begin{lemma}[Modified principal value]\label{l:cancel-2}
 Assuming that $\nabla_v f \in L^1(\R^d, (1+|v|)^{\gamma+2s}\dv )$, we have 
\begin{equation} \label{e:claim}
 \lim_{R\to 0} \int_{B_R \setminus E_R} w K_f(\bar v+w,\bar v) \dd w = 0.
\end{equation}
\end{lemma}
%-------------------------------------------------------------
Like in Lemma~\ref{lem:cancel}, the identity \eqref{e:claim} is clearly to be expected by common sense but it takes some work to prove it rigorously. As before, the condition $D f \in L^1(\R^d, (1+|v|)^{\gamma+2s}\dv)$ is a qualitative requirement that does not affect any of our estimates. 

\begin{proof}[Proof of Lemma~\ref{l:cancel-2}]
As in the proof of Lemma~\ref{lem:cancel}, we expand the integral using \eqref{e:A} and symmetrize it in $w$.
\begin{align*} 
 \int_{B_R \setminus E_R} w &K_f(\bar v + w, \bar v) \dw  = 
\int_{B_R \setminus E_R} w \left\{ \int_{u \perp w} f(\bar v+w  + u) 
A (|u|,|w|) \frac{|u|^{\gamma+2s+1}}{|w|^{d+2s}} \du \right\} \dd w, \\
&= \frac 12 \int_{B_R \setminus E_R} w \left\{ \int_{u \perp w} ( f(\bar v+w  + u) - f(\bar v-w  + u) )
A (|u|,|w|) \frac{|u|^{\gamma+2s+1}}{|w|^{d+2s}} \du \right\} \dd w.
\end{align*}
We write the increment $f(\bar v+w  + u) - f(\bar v-w  + u)$ in terms of the integral of the derivative along the segment, and proceed like in the proof of Lemma~\ref{lem:cancel}.
\begin{align*} 
\bigg\vert \int_{B_R \setminus E_R} w &K_f(\bar v + w, \bar v) \dw  \bigg\vert  \\
&\lesssim \int_{B_R \setminus E_R} \int_{u \perp w}
|u|^{\gamma+2s+1} |w|^{2-d-2s} \int_{-1}^1 |\nabla f(\bar v + \tau w + u)| \dd \tau \du  \dd w, \\
\intertext{for fixed $w$, we write $z = \tau w + u$ and observe that $\dd z = |w| \dd u \dd \tau$ and $|u| \le |z|$,}
& \leq \int_{B_R \setminus E_R}  \int_{-|w|^2 < z \cdot w < |w|^2} |z|^{\gamma+2s+1} |w|^{1-d-2s} |\nabla f(\bar v + z)| \dd z \dd w, \\
\intertext{switching the order of integration,}
                                      &\leq \int_{z \in \R^d} |z|^{\gamma+2s+1} |\nabla f(\bar v+z)| \left\{\int_{\substack{w \in B_R, \\ -|w|^2 < z \cdot w < |w|^2}} |w|^{1-d-2s} \dd w \right\} \dd z, \\
  \intertext{using again \eqref{e:2ballsout},}
                                      &\lesssim R^{2-2s} \int_{z \in \R^d} |z|^{\gamma+2s} |\nabla f(\bar v+z)| \dd z.
\end{align*}
This converges to zero as $R \to 0$, so the proof is concluded.
\end{proof}

The following auxiliary lemma contains an estimate that will be useful in the next lemma.
\begin{lemma} \label{l:aux2}
Let $s \in [1/2,1)$. For any $v_0 \in \R^d \setminus B_2$, $v \in B_1(v_0)$, and $r \in (0,1)$, we have
\[ \int_{\R^d} f(v+z) |z|^{\gamma+1} \min(1,r^{2-2s} |z|^{2s-2}) \dd z \leq C |v_0|^{\gamma+2s-1} r^{1-2s},\]
for some constant $C$ depending on $M_0$, $E_0$, $s$, $\gamma$, dimension $d$, and $C_\gamma$ if $\gamma<0$, but not on $v_0$.
\end{lemma}

\begin{proof}
Note that since $s \geq 1/2$ and $r<1$, we have $r^{1-2s} \geq 1$.

If $\gamma+2s \geq 1$ the inequality follows easily applying Lemma \ref{l:convolution-moments}. Indeed,
\begin{align*}
  \int_{\R^d} f(v+z) |z|^{\gamma+1} \min(1,r^{2-2s} |z|^{2s-2}) \dd z
  &\leq r^{2-2s} \int_{\R^d} f(v+z) |z|^{\gamma+2s-1} \dd z, \\
  & \lesssim r^{2-2s} (E_0 + (1+|v|^{\gamma+2s-1}) M_0),\\
  & \lesssim r^{2-2s} |v_0|^{\gamma+2s-1} \\
  & \lesssim r^{1-2s} |v_0|^{\gamma+2s-1}.
\end{align*}
We used that $r \in (0,1)$ and $|v|\approx |v_0|$.

So, let us concentrate on the more difficult case $\gamma+2s < 1$.

We split the domain of integration between $z \in B_r$ and $z \notin B_r$. Let us call each term $I_1$ and $I_2$,
\begin{align*} 
I_1 &:= \int_{|z| \le r} f (v  +z)  |z|^{\gamma+1}  \dz, \\
I_2 &:= r^{2-2s} \int_{|z| \ge r} f (v  +z)  |z|^{\gamma+2s-1}  \dd z.
\end{align*}
We now estimate $I_1$ and $I_2$ separately.

When $\gamma+1 \geq 0$, we easily get $I_1 \lesssim E_0 |v_0|^{-2} \lesssim |v_0|^{\gamma+2s-1} r^{1-2s}$, since $r \leq 1$ and $|v_0| \geq 2$. If $\gamma+1 < 0$, we use H\"older inequality and get
\begin{align*}
I_1 = \int_{|z|\le r} f(v+z) |z|^{\gamma+1} \dd z  & \le \left( \int_{|z|\le r} f(v+z) |z|^{\gamma} \dd z \right)^{\frac{\gamma}{1+\gamma}}
\left( \int_{|z|\le r} f(v+z)  \dd z \right)^{-\frac1{\gamma}} \\
& \lesssim C_\gamma^{\frac{\gamma}{1+\gamma}} E_0^{-1/\gamma} |v_0|^{\frac2{\gamma}}.
\end{align*}
Since $2/\gamma \leq -1$ and $\gamma+2s \geq 0$, then $\gamma+2s-1 \geq 2/\gamma$ and $I_1 \lesssim |v_0|^{\gamma+2s-1} \leq |v_0|^{\gamma+2s-1} r^{1-2s}$. This concludes the upper bound for $I_1$.

As far as $I_2$ is concerned, we further split the integral between two subdomains. If $v+z \in B_{|v_0|/2}$ and $|z| \ge r$, we use that $|z| \approx |v_0|$. If $v+z \notin B_{|v_0|/2}$ and $|z| \ge r$, we have,
\begin{align*}
  \int_{v+z \notin B_{|v_0|/2}} f(v+z) |z|^{\gamma+2s} \dd z &\lesssim \int_{v+z \notin B_{|v_0|/2}} f(v+z) (|v+z|^{\gamma+2s} +|v|^{\gamma +2s} ) \dd z\\
  & \lesssim \int_{v+z \notin B_{|v_0|/2}} f(v+z) (|v+z|^{2} |v_0|^{\gamma +2s-2} +|v_0|^{\gamma +2s}) \dd z\\
                                                             & \lesssim \int_{v+z \notin B_{|v_0|/2}} f(v+z) (|v+z|^{2} +1) |v_0|^{\gamma +2s-2} \dd z \\
  & \lesssim (E_0+M_0)|v_0|^{\gamma+2s-2}.
\end{align*}
With such an estimate at hand, we can now write,
\begin{align*}
  I_2 &\lesssim r^{2-2s} |v_0|^{\gamma+2s-1} \int_{v+z \in B_{|v_0|/2}} f(v+z) \dd z + r^{1-2s} \int_{v+z \notin B_{|v_0|/2}} f(v+z) |z|^{\gamma+2s} \dd z, \\
&\lesssim r^{2-2s} |v_0|^{\gamma+2s-1} M_0 + r^{1-2s} |v_0|^{\gamma+2s-2} (E_0+M_0) \lesssim r^{1-2s} |v_0|^{\gamma+2s-1}.
\end{align*}
Adding up the upper bounds for $I_1$ and $I_2$, we finish the proof of the lemma.
\end{proof}

The main result of this subsection is the following.

\begin{lemma}[Second cancellation condition]\label{l:cc2}
Let $s \geq 1/2$ and $f(v)$ be a function so that Assumption~\ref{a:hydro-assumption} holds. If $\gamma<0$, we assume \eqref{e:extra-int} as well. Then, the kernel $\bar K_f$ from \eqref{e:barKf} satisfies
  \eqref{e:cancellation2} with a parameter $\Lambda$ depending on $M_0$, $E_0$,
  $\gamma$, $s$, $d$, and  also $C_\gamma$ if $\gamma<0$.
\end{lemma}
\begin{proof}
For $|v_0|<2$, the result was established in \cite[Lemma 3.7]{imbert2016weak}. Here, we focus on the case $|v_0| \geq 2$.

Given $r \in [0,1/4]$, note that for $s=1/2$, $r^{1-2s} = 1$, whereas for $s>1/2$, $r^{1-2s} > 1$.

For any  $v \in \R^d$ and $r>0$, we have
\[ \PV \int_{B_r(v)} (v'-v) \bar K_f(v,v') \dv' = \frac1{|v_0|^{\gamma +2s+1}} \PV \int_{B_r} w K_f(\bar v, \bar v + T_0w) \dd w =0.\]
Here, we write $\bar v=v_0 + T_0v$ and we use the symmetry of the Boltzmann kernel: $K_f(\bar v, \bar v+\bar w) = K_f(\bar v, \bar v - \bar w)$.

Therefore, the proof is reduced to estimating the term in \eqref{e:cancellation2} involving $K_f(v',v)$ only. That is, we need to estimate the quantity $I(v')$ for $v' \in B_{7/4}$ given by
\[ I(v'):= \left\vert \PV \int_{B_r( v')} ( v'-v) \bar K_f( v, v') \dv \right\vert.\]

Let us change variables. As usual, we write $\bar v' = v_0 + T_0 v'$ and $\bar w = T_0(v'-v)$. We get
\[ I( v') = |v_0|^{-\gamma-2s} \left| \PV \int_{E_r} (T_0^{-1} \bar w) \, K_f(\bar v'-\bar w,\bar v') \dd \bar w \right|.\]
We used Formula \eqref{e:barKf}, $\det T_0^{-1}=|v_0|$, and Lemma~\ref{l:cancel-2} in order to justify the change of variables under the principal value.

Recall that the domain of integration $E_r$ is an ellipsoid. In order to capture the cancellations correctly, it is better to work with a round ball $B_r$. We use that $I(v') \leq I_1 + I_2$, where
\[ I_1 := |v_0|^{-\gamma-2s} \left \vert \PV \int_{B_r} (T_0^{-1} \bar w) \, K_f(\bar v'-\bar w,\bar v') \dd \bar w \right\vert, \qquad I_2 := |v_0|^{-\gamma-2s} \left \vert \int_{B_r \setminus E_r} (\dots) \dd \bar w \right\vert.\]
The rest of the proof is divided into two steps corresponding to establishing the bound for each one of the terms $I_1$ and $I_2$. 

\textbf{Step 1: $I_2 \lesssim r^{1-2s}$.}

In order to estimate $I_2$, we do not need to take any cancellation into consideration. We simply take absolute values everywhere and estimate each factor separately. Recalling Formula \eqref{e:Kernel-approx}, we have
\begin{align*} 
|v_0|^{\gamma+2s} I_2 &=  \left \vert \int_{B_r \setminus E_r} (T_0^{-1} \bar w) \, K_f(\bar v'-\bar w,\bar v') \dd \bar w \right\vert, \\
&\lesssim  \int_{B_r \setminus E_r} |T_0^{-1} \bar w| |\bar w|^{-d-2s} \int_{u \perp \bar w} f(\bar v' - \bar w + u) |u|^{\gamma+2s+1} \dd u \dd \bar w.
\intertext{Like in the proof of Lemma~\ref{l:bc2} about the second boundedness condition, we use polar coordinates $-\bar w = \rho \sigma$, and we write $r_\sigma\approx r(1+(\sigma \cdot v_0)^2)^{-1/2}$ for the maximum value of $\rho$ so that $\rho \sigma \in E_R$.}
&= \int_{\sigma \in S^{d-1}} \int_{r_\sigma < \rho < r} \rho^{-2s} |T_0^{-1} \sigma| \int_{u \perp \sigma} f(\bar v' + \rho \sigma + u) |u|^{\gamma+2s+1} \dd u \dd \rho \dd \sigma. \\
\intertext{We now write $z = \rho \sigma + u$ so that $\dd z = \dd \rho \dd u$, and observe $\rho = z \cdot \sigma$ and $|u| \leq |z|$.}
&\leq \int_{\sigma \in S^{d-1}} \int_{r_\sigma < z \cdot \sigma <r} (z \cdot \sigma)^{-2s} |T_0^{-1} \sigma| f(\bar v' + z) |z|^{\gamma+2s+1} \dd z \dd \sigma, \\
&= \int_{\R^d} f(\bar v' + z) |z|^{\gamma+1}  \left( \int_{\{\sigma : r_\sigma < \sigma \cdot z < r\}} (\sigma \cdot z/|z|)^{-2s} |T_0^{-1} \sigma| \dd \sigma \right) \dd z, \\
&= \int_{\R^d} f(\bar v' + z) |z|^{\gamma+1}  \left( \int_{\{\sigma : r_\sigma < \sigma \cdot z < r\}} (\sigma \cdot z/|z|)^{-2s}  \frac{r}{r_\sigma} \dd \sigma \right) \dd z, \\
&\approx \int_{\R^d} f(\bar v' + z) |z|^{\gamma+1}  \left( \int_{\{\sigma : r_\sigma < \sigma \cdot z < r\}} (\sigma \cdot z/|z|)^{-2s}  \sqrt{1+ (\sigma \cdot v_0)^2} \dd \sigma \right) \dd z. 
\end{align*}

We analyze the spherical integral similarly as in the proof of Lemma \ref{l:bc2}. We write $e = z/|z|$, $v_0 = ae+b$ with $b \perp e$ and $a = v_0 \cdot e$. We observe $ \sqrt{1+ (\sigma \cdot v_0)^2} \lesssim \left( 1 + a^2 (\sigma \cdot e)^2 + |b|^2 \right)^{1/2}$. We divide the domain of integration depending on whether $a^2 (\sigma \cdot e)^2 > 1+|b|^2$ or not. The purpose of these two subdomains is to know which term in $r_\sigma$ to focus on. Let us call $S_1$ and $S_2$  these two integrals respectively. We have
\begin{align*}
  I_2 &\lesssim |v_0|^{-\gamma-2s} \int_{\R^d} f(\bar v'+z) |z|^{\gamma+1} (S_1+S_2) \dd z,\\
\intertext{where}
  S_1 &\lesssim \int_{\sigma\cdot e \leq r/|z|} (\sigma\cdot e)^{-2s+1} |a| \dd \sigma \\
&  \lesssim |a| \min \left(1,(r/|z|)^{2-2s} \right)  \\
  S_2 &\lesssim \int_{\frac{r}{\sqrt{2} |z| (1+|b|^2)^{1/2}} \leq e \cdot \sigma} (e \cdot \sigma)^{-2s} (1+|b|^2)^{1/2} \dd \sigma, \\
      &\lesssim r^{1-2s} |z|^{2s-1} (1+|b|^2)^{s}.
\end{align*}

For the first term in $I_2$, we have $|a| \le |v_0|$ and consequently,
\begin{align*}
  |v_0|^{-\gamma-2s} \int_{\R^d} f(\bar v'+z) |z|^{\gamma+1} S_1 \dd z &\lesssim |v_0|^{-\gamma-2s+1} \int_{\R^d} f(\bar v'+z) |z|^{\gamma+1} \min(1,r^{2-2s} |z|^{2s-2}) \dd z \\
  & \lesssim r^{1-2s} \qquad \text{thanks to Lemma \ref{l:aux2}.}
\end{align*}

Note that $|b|^2 = |v_0|^2 - (v_0 \cdot e)^2$. So, for the second term in $I_2$, we have
\begin{align*}
|v_0|^{-\gamma-2s} \int_{\R^d} f(\bar v'+z) |z|^{\gamma+1} S_2 \dd z &\lesssim r^{1-2s} |v_0|^{-\gamma-2s} \int_{\R^d} f(\bar v'+z) |z|^{\gamma+2s} (1+|v_0|^2-(v_0 \cdot z)^2/|z|^2)^s \dd z, \\
&\lesssim r^{1-2s} \qquad \text{thanks to Lemma \ref{lem:barlambda}.}
\end{align*}

Adding all the terms, we get the announced estimate on $I_2$. 

\medskip
\textbf{Step 2: $I_1 \lesssim r^{1-2s}$.}

The cancellation inside the integral, and in particular in the principal value, plays a central role in the inequality for $I_1$. We proceed similarly as in step 1 but without taking absolute values and keeping equalities. We use polar coordinates $\bar w = r \sigma$ and write $z = - r \sigma + u$ with $\dd z = \dd r \dd u$.
\[ 
 |v_0|^{\gamma+2s} I_1
= \left\vert \PV \int_{z \in \R^d} f(\bar v'+z) T_0^{-1} \left\{  \int_{\sigma \in S^{d-1} : 0 < z \cdot \sigma < r}  \sigma (\sigma \cdot z)^{-2s} ( |z|^2 - (z \cdot \sigma)^2)^{\frac{\gamma+2s+1}2} A(\dots ) \dd \sigma \right\}  \dd z \right\vert . \]
Here, $A$ is the function from \eqref{e:A}. In this case, its value depends on $\sigma \cdot z/|z|$ only: $A(\dots) = \tilde{A} (\sigma \cdot z /|z|) \approx 1$.

If we write $\sigma = \rho z / |z| + \sigma^\perp$ with $\sigma^\perp$ perpendicular to $z$, we see that the only factor in the integrand that depends on $\sigma^\perp$ is $\sigma$. Since $\sigma^\perp$ takes values on a $(d-2)$-dimensional circle, its values cancel out in the integral. Thus, we reduce the integral to
\begin{align*} 
&  |v_0|^{\gamma+2s}  I_1 \\
  &= \left| \int_{z \in \R^d} f(\bar v'+z) |z|^{-2s} T_0^{-1} \left(\frac z {|z|}\right) \int_{\sigma \in S^{d-1} : 0 < z \cdot \sigma < r}  \left(\sigma \cdot \frac z {|z|} \right)^{1-2s} ( |z|^2 - (z \cdot \sigma)^2)^{\frac{\gamma+2s+1}2} \tilde A(\sigma \cdot z/|z| ) \dd \sigma \dd z \right|. \\
  \intertext{ At this point, the cancellations have been taken into account already. All quantities that remain are positive so we continue with inequalities. We use that $\|T_0^{-1}\|=|v_0|$ and that $\tilde A$ is bounded.}
& \le \int_{z \in \R^d} f(\bar v'+z) |z|^{-2s} \left\vert T_0^{-1} \frac z {|z|} \right\vert \int_{\sigma \in S^{d-1} : 0 < z \cdot \sigma < r}  \left(\sigma \cdot \frac z {|z|} \right)^{1-2s} ( |z|^2 - (z \cdot \sigma)^2)^{\frac{\gamma+2s+1}2} \tilde A(\sigma \cdot z/|z| ) \dd \sigma \dd z, \\ 
&\lesssim |v_0| \int_{z \in \R^d} f(\bar v'+z) |z|^{\gamma+1} \left( \int_{\sigma \in S^{d-1} : 0 < z \cdot \sigma < r}  \left(\sigma \cdot \frac z {|z|} \right)^{1-2s} \dd \sigma \right) \dd z , \\
&\lesssim |v_0| \int_{z \in \R^d} f(\bar v'+z) |z|^{\gamma+1} \min(1,(r/|z|)^{2-2s}) \dd z .
\end{align*}
We conclude that $I_1 \lesssim r^{1-2s}$ by applying Lemma \ref{l:aux2}.
\end{proof}

The change of variables theorems derive from the series of lemmas established in this section.  
\begin{proof}[Proof of Theorem~\ref{t:change-of-vars}]
  Combine Corollary~\ref{c:ndc} and Lemmas~\ref{l:bc1}, \ref{l:bc2}, \ref{l:cc1},  \ref{l:cc2}. The inequality \eqref{e:nondeg2} is a combination of \eqref{e:local_coercivity} with \eqref{e:cancellation1}.
\end{proof}

\begin{proof}[Proof of Theorem~\ref{t:change-of-vars-schauder}]
Combine Corollary~\ref{c:ndc-schauder} and Corollary~\ref{c:bounded_for_schauder}.
\end{proof}

 \subsection{H\"older spaces through the change of variables}

We examine how the change of variables $\mathcal{T}_0$ defined in \eqref{e:T0}
affects the kinetic H\"older spaces introduced in Section~\ref{s:holder}.
%----------
\begin{lemma}\label{l:holder-cov}
  Given $z_0 \in \R^{1+2d}$ and  $F : \mathcal{E}_R (z_0) \to \R$, we define $\bar F :Q_R \to \R$ by
  $\bar F(z) = F (\mathcal{T}_0 (z))$. Then, for any $\beta>0$,
  \begin{equation}
    \label{e:holder-after-cdv}
    \|\bar F\|_{C_\ell^\beta (Q_R)} \lesssim \|F\|_{C_\ell^{\beta} (\mathcal{E}_R (z_0))} \lesssim |v_0|^{{\bar c} \beta} \|\bar F\|_{C_\ell^\beta (Q_R)}, 
  \end{equation}
  with ${\bar c} = \max \left( \frac{\gamma+2s}{2s},1 \right)$. 
\end{lemma}

% -----------
\begin{proof}
We point out that we only need to prove this Lemma for $|v_0|>2$ (See Remark~\ref{r:large-v0}).

We can write $\mathcal{T}_0$ as $z_0 \circ \mathcal{T}$ with 
\[ \mathcal{T} (t,x,v) = (|v_0|^{-\gamma -2s}t ,|v_0|^{-\gamma - 2s} (T_0 x),T_0v).\]

We first prove that for $v_0 \in \R^d \setminus B_2$ and for all $z,z_1 \in \R^{1+2d}$, 
\begin{align}
  \label{e:dellT}
  d_\ell (\mathcal{T}^{-1}z,\mathcal{T}^{-1} z_1) &\lesssim |v_0|^{{\bar c}} d_\ell(z,z_1) \\
\label{e:dellTbis}
  d_\ell (\mathcal{T} z,\mathcal{T} z_1) &\lesssim d_\ell(z,z_1) 
\end{align}
with ${\bar c} = \max (\frac{\gamma+2s}{2s},\frac{\gamma+2s}{1+2s}+1)$. As far as \eqref{e:dellT} is concerned, using the definition of $d_\ell$, we write
\begin{align*}
d_\ell (\mathcal{T}^{-1}z,\mathcal{T}^{-1} z_1) &= \min_{w \in \R^d} \max \bigg\{ (|v_0|^{\gamma+2s} |t-t_1|)^{1/(2s)}, \\
&\phantom{= \min_{w \in \R^d} \max \bigg\{} \left( |v_0|^{\gamma+2s} |T_0^{-1} (x-x_1-(t-t_1)w)| \right)^{1/(1+2s)}, \\
&\phantom{= \min_{w \in \R^d} \max \bigg\{} |T_0^{-1}(v-w)|, |T_0^{-1}(v_1-w)| \bigg\}, \\
\intertext{note that $\|T_0^{-1}\| = |v_0|$,}
&\leq \max\left( |v_0|^{\frac{\gamma+2s}{2s}}, |v_0|^{\frac{\gamma+2s+1}{1+2s}},|v_0| \right) \min_{w \in \R^d} \max \big\{ |t-t_1|^{\frac 1 {2s}}, \\
&\phantom{\leq \max\left( |v_0|^{\frac{\gamma+2s}{2s}}, |v_0|^{\frac{\gamma+2s+1}{1+2s}},|v_0| \right) \min_{w \in \R^d} \max \big\{ } |x-x_1-(t-t_1)w|^{\frac 1 {1+2s}},\\
&\phantom{\leq \max\left( |v_0|^{\frac{\gamma+2s}{2s}}, |v_0|^{\frac{\gamma+2s+1}{1+2s}},|v_0| \right) \min_{w \in \R^d} \max \big\{ } |v-w|, |v_1-w| \big\}, \\
&\leq |v_0|^{{\bar c}} d_\ell(z,z_1).
\end{align*}
Note that $|v_0|^{\frac{\gamma+2s+1}{1+2s}} \leq \max \left( |v_0|^{\frac{\gamma+2s}{2s}},|v_0| \right)$. This justifies \eqref{e:dellT}. The verification of \eqref{e:dellTbis} is very similar using that $\|T_0\|=1$.

Since $d_\ell$ is left invariant, \eqref{e:dellT} implies that for $v_0 \in \R^d \setminus B_2$, $z,z_1 \in \R^{1+2d}$, and $\bar z = \mathcal T_0 z$, $\bar z_1 = \mathcal T_0 z_1$,
\begin{equation}
  \label{e:dellT0}
d_\ell(\bar z, \bar z_1) \leq  d_\ell (z,z_1) \lesssim |v_0|^{{\bar c}} d_\ell(\bar z, \bar z_1).
\end{equation}

We deduce \eqref{e:holder-after-cdv} from \eqref{e:dellT0}. Given any $\bar z_1,\bar z \in Q_R$, and $\bar z = \mathcal T_0 z$, $\bar z_1 = \mathcal T_0 z_1$, let $p$ be the polynomial expansion of $F$ at the point $\bar z_1$ so that $|F(\bar z) - p(\bar z)| \leq [F]_{C^\beta_\ell} d_\ell(\bar z, \bar z_1)^\beta$ and $\deg_k p < \beta$. We observe that $p \circ \mathcal T_0$ is a polynomial of the same degree. Thus,
\[ |\bar F(z) - p \circ \mathcal T_0 (z)| = |F(\bar z) - p(\bar z)| \leq [F]_{C^\beta_\ell} d_\ell(\bar z, \bar z_1)^\beta \leq [F]_{C^\beta_\ell} d_\ell(z, z_1)^\beta.\]
We deduced the first inequality in \eqref{e:holder-after-cdv} from the first inequality in \eqref{e:dellT0}. We deduce the second one similarly.
\end{proof}

\subsection{H\"older continuity of the kernel}

In the next lemma we explore how a H\"older estimate for $f$ of the form $f \in \Cpol^\alpha$ results in a H\"older estimate for the kernel $\bar K_f$ as in the assumption \eqref{e:A0-schauder} in Theorem~\ref{t:local-schauder}. 

% ---------------------------------------
\begin{lemma}[H\"older continuity of the kernel] 
\label{l:cdv-coef-holder}
Let $f : [\tau,T] \times \R^d \times \R^d$ be such that  $f \in \Cpol^{\alpha}$.  Then  \eqref{e:A0-schauder} holds true for $\bar K_f$ with  $\alpha' = \frac{2s}{1+2s} \alpha$ and 
\[
    \bar  A_0 \leq C \left( \|f\|_{C^0_{\ell,q}} +  (1+|v_0|)^{\frac \alpha {1+2s} (1-2s-\gamma)_+} [f]_{C^\alpha_{\ell,q}} \right).
\] 
Here, $q$ can be any number larger than $d+2+\alpha/(1+2s)$. The constant $C$ depends only on dimension, $\gamma$, $s$, $\min(1,T-\tau)$, and the choice of $q$.
\end{lemma}

%--------------
\begin{proof}
  Without loss of generality, let us assume $T-\tau \geq 1$. The effect of this assumption is that we take cylinders of the form $Q_1(z_0) \subset [\tau,T] \times \R^d \times \R^d$. Otherwise we would have to work with cylinders $Q_r(z_0)$ for a smaller $r>0$ and the choice of $r$, depending on $T-\tau$, would affect the constants in the lemma.

  As usual, we also focus on $|v_0|>2$.

Recall that the definition of $\Cpol$ says that for all $q>0$, there exists a constant $C_q = [f]_{C^\alpha_{\ell,q}}$ so that
\[ \|f\|_{C^\alpha_{\ell}(Q_1(z_0))} \leq C_q (1+|v_0|)^{-q},\]
whenever $Q_1(z_0) \subset [\tau,T] \times \R^d \times \R^d$. In particular, for $z=(t,x,v)$,
\[ |f(z)| \leq C_1 (1+|v|)^{-q} \qquad \text{and} \qquad |f(z_1) - f(z_2)| \leq C_q d_\ell(z_1,z_2)^\alpha (1+|v_1|)^{-q} \text{ whenever } d_\ell(z_1,z_2) < 1.\]

As usual, we write $\bar z = \mathcal T_0 z$. According to the change
of variables formula \eqref{e:barKf}, we thus have for all
$z_1,z_2 \in Q_1$ (in particular $d_\ell(z_1,z_2)<1$),
\begin{equation} \label{e:Kh1}
 \bar K_{f,z_1}(w) - \bar K_{f,z_2}(w) = |v_0|^{-1-\gamma-2s} \left( K_{f,\bar z_1} (T_0 w) - K_{f,\bar z_2} (T_0 w) \right).
\end{equation}

From the formula \eqref{e:A}, we observe that, for any $\bar z_1 = \mathcal T_0 z_1$ and $\bar z_2 = \mathcal T_0 z_2$, 
\begin{align*} 
 |K_{f,\bar z_1} (w) &- K_{f,\bar z_2}(w)| \\
& = |w|^{-d-2s}\left\vert \int_{u \perp w} |u|^{\gamma+2s+1} A(|w|,|u|) \left( f(\bar z_1 \circ (0,0,u)) - f(\bar z_2 \circ (0,0,u)) \right) \dd u \right\vert, \\
& \leq |w|^{-d-2s} \int_{u \perp w} |u|^{\gamma+2s+1} A(|w|,|u|) \left\vert f(\bar z_1 \circ (0,0,u)) - f(\xi \circ \bar z_1\circ (0,0,u)) \right \vert \dd u, \\
&= K_{\Delta f, \bar z_1}(w),
\end{align*}
where $\Delta f(z) = |f(z) - f(\xi \circ z)|$ and  $\xi = \bar z_1 \circ \bar z_2^{-1}$. Combining with \eqref{e:Kh1}, we get
\begin{align*}
\int_{B_\rho} |\bar K_{f,z_1}(w) - \bar K_{f,z_2}(w)| |w|^2 \dd w &\leq \int_{B_\rho} |\bar K_{\Delta f,z_1}(w)| |w|^2 \dd w, \\
\intertext{we now apply Corollary~\ref{c:bounded_for_schauder} to obtain}
&\lesssim \left( \int_{\R^d} (1+|v|^2) \Delta f(\bar t_1, \bar x_1, v) \dd v \right) \rho^{2-2s}, \\
&= \left( \int_{\R^d} (1+|\bar v_1+w|^2) | f(\bar z_1 \circ (0,0,w)) - f(\bar z_2 \circ (0,0,w))  |  \dd w \right) \rho^{2-2s}.
\end{align*}

In order to estimate the difference $| f(\bar z_1 \circ (0,0,w)) - f(\bar z_2 \circ (0,0,w))  |$, we use the $C^\alpha_{\ell,q}$ semi-norms of $f$. We use \eqref{e:repeatedly} and we get
\begin{equation} \label{e:kh1}
 | f(\bar z_1 \circ (0,0,w)) - f(\bar z_2 \circ (0,0,w))  |  \leq
\begin{cases} \left( d_\ell(\bar z_1,\bar z_2) + |\bar t_1-\bar t_2|^{1/(1+2s)} |w|^{1/(1+2s)} \right)^\alpha (1+|\bar v_1+w|)^{-q} [f]_{C^\alpha_{\ell,q}}, \\
\left\{ (1+|\bar v_1+w|)^{-q} + (1+|\bar v_2+w|)^{-q} \right\}   [f]_{C^0_{\ell,q}}. \end{cases} 
\end{equation}
The first line applies whenever $d_\ell(\bar z_1 \circ (0,0,w),\bar z_2 \circ (0,0,w)) < 1$. Note that from \eqref{e:repeatedly}, this distance is less or equal than the factor inside the parenthesis on the first line.

The second line applied whenever $d_\ell(\bar z_1 \circ (0,0,w),\bar z_2 \circ (0,0,w)) \geq 1$. In that case, it is better to estimate the left hand side in \eqref{e:kh1} by the $C^0$ norm of each term. 

Note that since $\bar v_1, \bar v_2 \in B_1(v_0)$, then $(1+|\bar v_1+w|) \approx (1+|\bar v_2+w|) \approx (1+|v_0+w|)$. Recall that we write $\|f\|_{C^\alpha_{\ell,q}} = [f]_{C^\alpha_{\ell,q}} +\|f\|_{C^0_{\ell,q}}$. Thus, in any case we have,
\begin{equation} \label{e:kh11}
 | f(\bar z_1 \circ (0,0,w)) - f(\bar z_2 \circ (0,0,w))  |  \lesssim
\left( d_\ell(\bar z_1,\bar z_2) + |\bar t_1-\bar t_2|^{1/(1+2s)} |w|^{1/(1+2s)} \right)^\alpha (1+|\bar v_1+w|)^{-q} \|f\|_{C^\alpha_{\ell,q}}.
\end{equation}

By the definition of the transformation $\mathcal T_0$, we always have $d_\ell(\bar z_1, \bar z_2) \leq d_\ell(z_1,z_2)$. Moreover, $|\bar t_1-\bar t_2| = |v_0|^{-\gamma-2s} |t_1-t_2|$.

We estimate
\begin{align*} 
 \int  & (1+|\bar v_1+w|^2) | f(\bar z_1 \circ (0,0,w)) - f(\bar z_2 \circ (0,0,w))  |  \dd w  \\
&\lesssim \int \left( d_\ell(\bar z_1,\bar z_2) + |\bar t_1-\bar t_2|^{1/(1+2s)} |w|^{1/(1+2s)} \right)^\alpha (1+|v_0+w|)^{2-q} \|f\|_{C^\alpha_{\ell,q}} \dd w, \\
&\lesssim \|f\|_{C^\alpha_{\ell,q}} \left( d_\ell(z_1,z_2)^\alpha + |v_0|^{-\frac{\gamma+2s}{1+2s} \alpha} |t_1-t_2|^{\frac 1 {1+2s} \alpha} \int_{\R^d} |w|^{\frac \alpha {1+2s}}(1+|v_0+w|)^{2-q} \dd w \right), \\
&\lesssim \|f\|_{C^\alpha_{\ell,q}} \left( d_\ell(z_1,z_2)^\alpha + d_\ell(z_1,z_2)^{\alpha'} |v_0|^{\frac{(-\gamma-2s+1) \alpha}{1+2s}} \right), \\
&\lesssim \|f\|_{C^\alpha_{\ell,q}} d_\ell(z_1,z_2)^{\alpha'} |v_0|^{\frac{\alpha}{1+2s} (1-\gamma-2s)_+ }.
\end{align*}
In the last inequality we used $d_\ell(z_1,z_2) < 1$ and $|v_0| > 2$. We also used Lemma \ref{l:convolution-C0} to estimate the last integral, which holds provided that $q > d+2+\alpha/(1+2s)$.
\end{proof}

\section{Bounds for the bilinear operator $\Q(\cdot, \cdot)$}

\label{s:bilinear}

The right hand side of the Boltzmann equation $\Q(f,f)$ is a quadratic function of $f$. Its structure as a bilinear operator $\Q(f,g)$ is relevant when differentiating the equation. In this section we collect several lemmas to evaluate the H\"older regularity of $\Q(f,g)$ in terms of H\"older norms of $f$ and $g$.

Recall that we write $\Q = \Q_1 + \Q_2$. We obtain bounds for each of these two terms separately.

\subsection{Bounds for $\Q_2$}

Recall that $\Q_2 (f,g) = c_b (f \ast |\cdot |^\gamma) g$. We start with estimating how the convolution with $|\cdot|^\gamma$ affects the local H\"older norm of a function. 

\begin{lemma}[Convolution with  $|\cdot|^\gamma$] \label{l:convol-gamma} 
Let $f \in \Cpol^\alpha([\tau,T] \times \R^d \times \R^d)$ for $0 < \alpha \leq \min(1,2s)$. Let us consider the convolution of $f$ with $|\cdot|^\gamma$ in velocity. That is
\[ g(z) = \int_{\R^d} f(z \circ (0,0,w)) |w|^\gamma \dd w.\]
Then for all $z_0 \in [\tau,T] \times \R^d \times \R^d$ and $r \in (0,1)$ such that $Q_r(z_0) \subset [\tau,T] \times \R^d \times \R^d$, and any $q > d+\gamma_+ + \alpha/(1+2s)$,
  \[
   \| g \|_{C_\ell^{\alpha'}(Q_r(z_0))}  \le C  (1+|v_0|)^{\frac{\alpha}{1+2s}+\gamma} \|f\|_{C_{\ell,q}^\alpha}
  \]
 with $C=C(d,\gamma,s,\alpha)$ and $\alpha' = \frac{2s}{1+2s} \alpha$.
\end{lemma}

\begin{proof}
Let $z_0=(t_0,x_0,v_0), z_1=(t_1,x_1,v_1) \in [\tau,T] \times \R^d \times \R^d$ with $d_\ell(z_0,z_1) < 1$. We will show the two inequalities below from which the conclusion follows.
\begin{align*} 
 |g(z_0)| &\leq C  (1+|v_0|)^{\gamma} [f]_{C_{\ell,q}^0},\\
 |g(z_1) - g(z_0)| &\leq C  (1+|v_0|)^{\frac{\alpha}{1+2s}+\gamma} \|f\|_{C_{\ell,q}^\alpha} d_\ell(z_0,z_1)^{\alpha'}.
\end{align*}
Note that the assumption $\alpha \leq \min(1,2s)$ implies that any polynomial of degree less than $\alpha$ must be  constant. Thus, the H\"older semi-norm of order $\alpha$ involves merely increments.

Let us start with the first inequality. Using Lemma~\ref{l:convolution-C0} (since $q > d+\gamma^+$), we compute
\begin{align*} 
| g(z_0)| & \le  \int_{\R^d} |f|(t_0,x_0,v_0+w) |w|^\gamma \dd w, \\
&\leq [f]_{C_{\ell,q}^0} \int_{\R^d} (1+|v_0+w|)^{-q} |w|^\gamma \dd w, \\
&\leq C [f]_{C_{\ell,q}^0} (1+|v_0|)^\gamma .
\end{align*}

The second inequality requires a slightly longer computation,
\begin{align*} 
 |g(z_1) - g(z_0)| &\leq \int_{\R^d} |f(z_1 \circ (0,0,w)) - f(z_0 \circ (0,0,w))| |w|^\gamma \dd w, \\
&\leq \int_{|w| < d_\ell(z_0,z_1)^{-2s}} (\dots) \dw  + \int_{|w|>d_\ell(z_0,z_1)^{-2s}} (\dots) \dw , \\[1ex]
&=: \mathrm{I} + \mathrm{II}.
\end{align*}

For the first integral $\mathrm I$, we use the inequality \eqref{e:repeatedly} together with the semi-norm $[f]_{C^\alpha_{\ell,q}}$. In this domain, we have 
\begin{align*} 
 d_\ell(z_0 \circ (0,0,w) , z_1 \circ (0,0,w)) &\leq d_\ell(z_0,z_1) + d_\ell(z_0,z_1)^{2s/(1+2s)}|w|^{1/(1+2s)}, \\
&\leq d_\ell(z_0,z_1) + 1 \leq 2.
\end{align*}
Therefore, using again Lemma~\ref{l:convolution-C0},
\begin{align*} 
 \mathrm I &\lesssim [f]_{C^\alpha_{\ell,q}} \int_{\R^d} \left(d_\ell(z_0,z_1) +  d_\ell(z_0,z_1)^{2s/(1+2s)}|w|^{1/(1+2s)} \right)^\alpha |w|^\gamma (1+|v_0+w|)^{-q} \dd w, \\ 
 &\lesssim [f]_{C^\alpha_{\ell,q}} \left( d_\ell(z_0,z_1)^\alpha (1+|v_0|)^\gamma + d_\ell(z_0,z_1)^{\alpha'} \int_{\R^d} |w|^{\alpha/(1+2s)+\gamma} (1+|v_0+w|)^{-q} \dd w \right), \\
\intertext{(recall $\alpha' = 2s \alpha / (1+2s)$)}
 &\lesssim [f]_{C^\alpha_{\ell,q}} \left( d_\ell(z_0,z_1)^\alpha (1+|v_0|)^\gamma + d_\ell(z_0,z_1)^{\alpha'} (1+|v_0|)^{\alpha/(1+2s)+\gamma} \right) \ \ \text{since } q > d+\gamma_++\alpha/(1+2s).
\end{align*}
Naturally, the second term is larger than the first one.

For the second integral $\mathrm{II}$ we bound $|f(z_0 \circ (0,0,w)) - f(z_1 \circ (0,0,w))|$ using $[f]_{C^0_{\ell,q}}$. That is
\begin{align*} 
 \mathrm{II} &\leq [f]_{C^0_{\ell,q}} \int_{|w|>d_\ell(z_0,z_1)^{-2s}}  ((1+|v_0 + w|)^{-q} + (1+|v_1 + w|)^{-q} ) |w|^\gamma \dd w.
\end{align*}
We analyze two cases depending on whether $d_\ell(z_0,z_1)^{-2s} > (1+3|v_0|)$ or not. In the first case, we have $|v_1+w| \gtrsim |w|$ and $|v_1 + w|\gtrsim |w|$, therefore
\begin{align*} 
 \mathrm{II} &\lesssim [f]_{C^0_{\ell,q}} \int_{|w|>d_\ell(z_0,z_1)^{-2s}} |w|^{-q+\gamma} \dd w, \\
& \lesssim [f]_{C^0_{\ell,q}} d_\ell(z_0,z_1)^{2s (q-d-\gamma)}, \\
&\lesssim  [f]_{C^0_{\ell,q}} (1+|v_0|)^\gamma d_\ell(z_0,z_1)^{\alpha'}.
\end{align*}
The last inequality holds because $q > d+\gamma_++\alpha/(1+2s)$ and $d_\ell(z_0,z_1)^{-2s} > (1+3|v_0|)$. Indeed, in this case we have
\begin{align*} 
 \frac{ d_\ell(z_0,z_1)^{2s (q-d-\gamma)} } { (1+|v_0|)^\gamma d_\ell(z_0,z_1)^{\alpha'} } &= d(z_0,z_1)^{2s (q-d-\gamma - \alpha/(1+2s) )} (1+|v_0|)^{-\gamma}, \\
& \lesssim (1+|v_0|)^{-(q-d-\gamma - \alpha/(1+2s)) - \gamma}, \\
&= (1+|v_0|)^{-(q-d - \alpha/(1+2s)) } ,\\
  & \lesssim 1.
\end{align*}

In the second case $d_\ell(z_0,z_1)^{-2s} \leq (1+3|v_0|)$, which means that $d(z_0,z_1) \geq (1+3|v_0|)^{-1/(2s)}$. Therefore
\begin{align*} 
 \mathrm{II} &\leq [f]_{C^0_{\ell,q}}  \int_{\R^d} ((1+|v_0 + w|)^{-q} + (1+|v_1 + w|)^{-q} ) |w|^\gamma \dd w,   \\
&\lesssim [f]_{C^0_{\ell,q}}  (1+|v_0|)^{\gamma}, \\
&\lesssim [f]_{C^0_{\ell,q}}  (1+|v_0|)^{\gamma+\frac{\alpha}{1+2s}} d_\ell(z_0,z_1)^{\alpha'}, \qquad \text{ using } d(z_0,z_1) \gtrsim (1+|v_0|)^{-1/(2s)}.
\end{align*}
Adding the upper bounds for $\mathrm I$ and $\mathrm{II}$ we conclude the proof of Lemma~\ref{l:convol-gamma}.
\end{proof}

\begin{lemma}[Bound for $\Q_2$] \label{l:Q2_bound} 
Let $f, g \in \Cpol^\alpha([\tau,T] \times \R^d \times \R^d)$ for $0 < \alpha \leq \min(1,2s)$. Then $Q_2(f,g) \in \Cpol^\alpha([\tau,T] \times \R^d \times \R^d)$ and the following estimates hold for any $q > d+\gamma_++\alpha/(1+2s)$,
  \[
   \| \Q_2(f,g) \|_{C_{\ell,q}^{\alpha'}}  \leq C   \|f\|_{C_{\ell,q}^\alpha} \|g\|_{C_{\ell,q+\alpha/(1+2s) + \gamma}^{\alpha'}}
  \]
 with $C=C(d,\gamma,s,\alpha)$ and $\alpha' = \frac{2s}{1+2s} \alpha$.
\end{lemma}

\begin{proof}
Recall that
\( \Q_2(f,g)(v) = c_b \left(\int_{\R^d} f(v+w) |w|^\gamma \dd w \right) g(v).\)
Given $Q_r(z_0) \subset [\tau,T] \times \R^d \times \R^d$, we combine Lemma~\ref{l:Calpha_product} and Lemma~\ref{l:convol-gamma} to get
\[
 (1+|v_0|)^q \| \Q_2 (f,g) \|_{C_\ell^{\alpha'}(Q_r(z_0))} \lesssim (1+|v_0|)^{q+ \gamma + \frac{\alpha}{1+2s}}\|f\|_{C_{\ell,q}^\alpha} \| g\|_{C_\ell^{\alpha'}(Q_r(z_0))}.
\]
Taking the supremum over $Q_r(z_0)$ yields the announced estimate. 
\end{proof}

\subsection{Bounds for $\Q_1$}
This paragraph is dedicated to the derivation of appropriate bounds for $\Q_1 (f,g)$ when $f \in \Cpol^\alpha$ and $g \in \Cpol^{2s+\alpha}$. For this purpose, we need to localize around a given point $z_0$. In order to measure the effect of this localization procedure, we need some preparatory lemmas. \medskip

The proof of the following lemma uses ideas introduced in \cite{imbert2018decay}. It is used to bound a part of the integral involved in the computation of an integro-differential operator.
%--------------------------------------------------------------------
\begin{lemma} \label{l:bump_integral_C0}
  Let $f$ be such that for $q> d+ (\gamma +2s)$ and for all $v \in \R^d$, 
  \[ |f(v)| \le N_q (1+|v|)^{-q}.\]
Let $K_f$ be the Boltzmann kernel given in formula \eqref{e:kf} applied to this function $f$. Then for all $v \in B_r (v_0)$ with $r \in (0,1)$  and $g \in L^1(B_{|v_0|/8})$,
  \[
    \int_{\substack{
|v'-v| > 1+ |v_0|/8, \\
|v'| < |v_0|/8 
}} |g(v')| K_f (v,v') \dv' \le C N_q \|g\|_{L^1(B_{|v_0|/8})} (1+|v_0|)^{-q+\gamma}
  \]
for some $C$ depends on $q,d,\gamma,s$ and the constants in $B$ (the Boltzmann kernel). 
\end{lemma}
\begin{remark}\label{r:bump}
Later on, Lemma~\ref{l:bump_integral_C0} will be applied to large values of $|v_0|$. In that case, the inequality $|v'| < |v_0|/8$, together with $v \in B_1(v_0)$, implies $|v'-v| > 1+|v_0|/8$.
\end{remark}
%----------------------------------------------------------------------
\begin{proof}
  The proof is very similar to the ones of
  \cite[Propositions~4.7,~4.8]{imbert2018decay}. 
Using \eqref{e:Kernel-approx}, we first write for $v \in B_r(v_0)$, 
\begin{align*}
  \int_{\substack{
|w| > 1+|v_0|/8, \\
|v+w| < |v_0|/8 
}} &|g(v+w)| K_f (v,v+w) \dw, \\
  & = \int_{\substack{
|w| > 1+|v_0|/8, \\
|v+w| < |v_0|/8 
}} |g(v+w)| |w|^{-d-2s} \left\{ \int_{u \perp w} f(v+u) A(|w|,|u|)|u|^{\gamma+2s+1} \du \right\} \dw,  \\
  & \lesssim  N_q \|g\|_{L^1(B_{|v_0|/8})} (1+|v_0|/8)^{-d-2s} \max_{\substack{
|w| > 1+|v_0|/8, \\
|v+w| < |v_0|/8 
}}  \left\{  \int_{u \perp w} (1+|v+u|)^{-q} |u|^{\gamma+2s+1} \dd u  \right\}.
\end{align*}

Since $|v+w| \leq |v_0|/8$, we also have
\[ |w| \geq |v_0| - |v-v_0| - |v+w| \geq \frac 78 |v_0| - 1.\]
Thus
\[ \frac{|w|+1}7 \geq |v+w|.\]
Therefore, we get for $u \perp w$,
\begin{align*} 
 1+ |v+u| &\geq 1 + |u-w| - |v+w|, \\
&\geq 1 + |u-w| - \frac{|w|+1}7 , \\
&= \frac 67  + (|u|^2+|w|^2)^{1/2} - \frac{|w|}7, \\
&\geq \frac 67 + \frac{|u|}{\sqrt 2} + \frac{|w|}{\sqrt 2}  -\frac{|w|}7,\\
&\gtrsim (1+|u|+|w|).
\end{align*}
We use this inequality to continue our estimate from the beginning of this proof.

\begin{align*}
  \int_{\substack{
|w| > 1+|v_0|/8, \\
|v+w| < |v_0|/8 
}} &|g(v+w)|  K_f (v,v+w) \dw, \\
  & \lesssim N_q \|g\|_{L^1(B_{|v_0|/8})} (1+|v_0|)^{-d-2s} \max_{|w| > 1+|v_0|/8}  \left\{ \left( \int_{u \perp w} (1+ |u| +|w|)^{-q} |u|^{\gamma+2s+1} \dd u \right) \right\} , \\
  & \lesssim N_q \|g\|_{L^1(B_{|v_0|/8})} (1+|v_0|)^{-d-2s} \max_{|w| \geq 1+|v_0|/8} (1+|w|)^{-q+\gamma+2s+d}, \\
  & \lesssim N_q \|g\|_{L^1(B_{|v_0|/8})} (1+|v_0|)^{-q+\gamma}.
\end{align*}

The implicit constant in $\lesssim$ depends only on $d,q,s,\gamma$. The proof is now complete. 
\end{proof}

For the next lemmas in this section, we define a cutoff function in the following way. Let $\bar \varphi$ be a fixed smooth nonnegative bump function supported in $B_{1/8}$ so that $\bar \varphi=1$ in $B_{1/9}$. For any given value of $v_0 \neq 0$, we define the cutoff function $\varphi$ as
\begin{equation} \label{e:bump_function}
 \varphi(v) := \bar \varphi( |v_0|^{-1} v).
\end{equation}
This function $\varphi$ is supported in $B_{|v_0|/8}$ and is identically $1$ in $B_{|v_0|/9}$. Note that Lemma~\ref{l:bump_integral_C0} can be reformulated using $\varphi$ in the following way (at least for $|v_0|$ large, see Remark~\ref{r:bump}), 
\begin{equation}
  \label{e:reformulated}
  \forall v \in B_1 (v_0), \qquad \Q_1(f,\varphi g)(v) = \int_{\R^d}g(v') \varphi(v') K_f(v,v') \dd v' \lesssim N_q \|\varphi g\|_{L^1(\R^d)} (1+|v_0|)^{-q+\gamma}.
\end{equation}
By the definition of $\|g\|_{C^\alpha_{\ell,q}}$, we see that for any $\alpha \geq 0$ and $q>0$,
\begin{equation} \label{e:bump_norm_decay}
 \|(1-\varphi) g\|_{C^\alpha_{\ell}(Q_1(z_0))} \lesssim (1+|v_0|)^{-q} \|g\|_{C^\alpha_{\ell,q}}.
\end{equation}

In the following lemma, we establish the upper bound and decay as $|v|\to \infty$ for $\Q_1(f,g)$ in terms of corresponding norms of $f$ and $g$.

\begin{lemma}[Pointwise upper bound for $\Q_1$] \label{l:Q1_C0}
Let $f \in \Cpol^0([0,T] \times \R^d \times \R^d)$ and $g \in \Cpol^{2s+\alpha}([0,T] \times \R^d \times \R^d)$ for some $\alpha>0$. Then $\Q_1(f,g) \in \Cpol^0([0,T] \times \R^d \times \R^d)$. Moreover, for any $q>d+\gamma+2s$,
\begin{equation} \label{e:cutoff_decay}
 \|\Q_1(f,g)\|_{C^0_{\ell,{q-\gamma-2s}}} \leq C \|f\|_{C^0_{\ell,q}}  \|g\|_{C^{2s+\alpha}_{\ell,q}}.
\end{equation}
Here, the constant $C$ depends only on $\alpha$, dimension, $s$, $\gamma$, and the constant in $B$ (the Boltzmann kernel).
\end{lemma}

\begin{proof}
Let us start by recalling the formula for $\Q_1(f,g)$.
\begin{align*} 
 \Q_1(f,g) &= \int_{\R^d} (g'-g) K_f(t,x,v,v') \dd v', \\
&= \int_{\R_d} (g(z\circ (0,0,w)) - g(z)) K_{f,z}(w) \dd w.
\end{align*}
Here, $K_f$ is the Boltmann kernel depending on the function $f$ as in \eqref{e:kf}. As usual, we write $K_{f,z}(w) = K_f(t,x,v,v+w)$ for $z=(t,x,v)$.

We need to establish an upper bound for $\Q_1(f,g)(z_0)$ for any given $z_0 = (t_0,x_0,v_0) \in [0,T] \times \R^d \times \R^d$.

If $|v_0| \leq 2$, we use \eqref{e:upper_bound_for_Kf} together with Lemma~\ref{l:pointwise_operator_bound} and conclude the inequality immediately.

If $|v_0| > 2$, we decompose $\Q_1(f,g) = \Q_1(f,\varphi g)+ \Q_1(f,(1-\varphi)g)$. Here, $\varphi$ is the cutoff function defined in \eqref{e:bump_function}.

For $\Q_1(f,\varphi g)$, we use Lemma~\ref{l:bump_integral_C0} to get \eqref{e:reformulated} and obtain
\[ |\Q_1(f,\varphi g)(z_0)| \lesssim \|f\|_{C^0_{\ell,q}} \|g\|_{C^0_{\ell,q}} (1+|v_0|)^{-q+\gamma}.\]

For $\Q_1(f,(1-\varphi) g)$, we apply \eqref{e:upper_bound_for_Kf} and Lemma~\ref{l:pointwise_operator_bound}. We get
\begin{align*} 
| \Q_1(f,(1-\varphi) g)(z_0)| &\lesssim \left( \int_{\R^d} f(t_0,x_0,v_0+w) |w|^{\gamma+2s} \dd w \right) \|(1-\varphi) g\|_{C_\ell^{2s+\alpha}} , \\
&\lesssim \|g\|_{C^{2s+\alpha}_{\ell,q}} \|f\|_{C^0_{\ell,q}} (1+|v_0|)^{-q+\gamma+2s} \qquad \text{using  Lemma~\ref{l:convolution-C0}.}
\end{align*}
Adding the estimates for $|\Q_1(f,\varphi g)(z_0)|$ and $|\Q_1(f,(1-\varphi) g)(z_0)|$, we conclude the proof.
\end{proof}

In order to estimate the H\"older semi-norm of $\Q_1 (f,g)$, we need the following lemma that is the $C^\alpha$ version of Lemma~\ref{l:bump_integral_C0}.

\begin{lemma} \label{l:bump_integral_Calpha_pre}
Let $f,g \in \Cpol^\alpha$ for some $\alpha \in (0,\min(1,2s)]$. Let $z_0 = (t_0,x_0,v_0)$ such that $|v_0| > 2$ and $Q_1(z_0) \subset [0,\infty) \times \R^d \times \R^d$. Let $\varphi$ be the smooth bump function supported in $B_{|v_0|/8}$ with $\varphi=1$ in $B_{|v_0|/9}$. Let $h: Q_1(z_0) \to \R$ be given by
\[ h(z) := \int_{\R^d} \varphi(v+w) g(z \circ (0,0,w)) K_{f,z} (w) \dd w.\]
Then $h \in C^{\alpha'}_\ell(Q_1(z_0))$ and for any $q>d+\gamma+2s$,
\[ [h]_{C^{\alpha'}_\ell(Q_1(z_0)) } \leq C \|g\|_{C^\alpha_{\ell,q+\alpha/(1+2s)}} \|f\|_{C^\alpha_{\ell,q+\alpha/(1+2s)}} (1+|v_0|)^{-q+\gamma+\alpha/(1+2s)},\]
for a constant $C$ depending on $q$, $d$, $\gamma$ and $s$.
\end{lemma}

\begin{proof}
Let $z_1, z_2 \in Q_1(z_0)$. We need to estimate an upper bound for $|h(z_1)-h(z_2)|$. We write $\xi = z_2 \circ z_1^{-1}$. As usual,  $\tau_\xi$  denotes the right translation operator $\tau_\xi f(z) := f(\xi \circ z)$.

Note that $d_\ell(z_1,z_2) \approx \|z_2^{-1} \circ z_1 \| \neq \|\xi\|$. In fact, $\|\xi\|$ can be large.

We have
\begin{align*} 
 h(z_2) - h(z_1) &= \int_{\R^d} \bigg( \varphi(v_2+w) g(z_2 \circ (0,0,w)) K_{f,z_2}(w) - \varphi(v_1+w) g(z_1 \circ (0,0,w)) K_{f,z_1}(w) \bigg) \dd w, \\ 
&= \int_{\R^d} \bigg( \varphi(v_2+w) g(z_2 \circ (0,0,w)) - \varphi(v_1+w) g(z_1 \circ (0,0,w)) \bigg) K_{f,z_1}(w) \dd w \\
& \phantom{=} + \int_{\R^d} \varphi(v_2+w) g(z_2 \circ (0,0,w)) \bigg( K_{f,z_2}(w) - K_{f,z_1}(w) \bigg) \dd w, \\
&= \int_{\R^d} \left( \tau_\xi \varphi g - \varphi g) (z_1 \circ (0,0,w))  \right) K_{f,z_1}(w) \dd w \\
                 & \phantom{=} + \int_{\R^d} \varphi(v_2+w) g(z_2 \circ (0,0,w)) \left( K_{(f - \tau_{\xi^{-1}} f),z_2}(w)\right) \dd w. \\
  \intertext{This implies in particular that,}
| h(z_2) - h(z_1) | &\leq \int_{\R^d} \left\vert (\tau_\xi \varphi g - \varphi g) (z_1 \circ (0,0,w))  \right \vert K_{f,z_1}(w) \dd w \\
& \phantom{=} + \int_{\R^d} |\varphi(v_2+w) g(z_2 \circ (0,0,w))| \left( K_{ |f - \tau_{\xi^{-1}} f| ,z_2}(w)\right) \dd w, \\
\intertext{applying Lemma~\ref{l:bump_integral_C0} and observing $|v_1| \approx |v_2| \approx |v_0|$, we get}
&\lesssim \left( \|f\|_{C^0_{\ell,q}}N_1  + \|\varphi g \|_{C^0_{\ell,q}} N_2 \right) (1+|v_0|)^{-q+\gamma},
\end{align*}
where
\begin{align*}
N_1 &:= \|\tau_\xi \varphi g - \varphi g \|_{L^1(B_{|v_0|/8})}, \\
N_2 &:= \sup_{w \in \R^d} \left( |f - \tau_{\xi^{-1}} f| (z_2 \circ (0,0,w)) |v_2 + w|^q \right).
\end{align*}

Let us first analyze $N_2$. 
\begin{align*} 
N_2 & = \sup_{w \in \R^d}  | f(z_2 \circ (0,0,w)) - f(z_1 \circ (0,0,w))| (1+|v_2+w|)^q,  \\
\intertext{ using that $1+| v_1| \approx 1+|v_2| \approx 1+|v_0|$, we compute}
&\lesssim \sup_{w \in \R^d} \min \left( 1, d_\ell( z_1 \circ (0,0,w), z_2 \circ (0,0,w))^\alpha \right) (1+|v_0+w|)^{-\alpha/(1+2s)} \|f\|_{C^\alpha_{\ell,q+\alpha/(1+2s)}},\\
\intertext{using \eqref{e:repeatedly},}
&\lesssim \sup_{w \in \R^d} \min \left( 1, d_\ell( z_1, z_2)^{\alpha'} (1+|w|)^{\alpha/(1+2s)} \right) (1+|v_0+w|)^{-\alpha/(1+2s)} \|f\|_{C^\alpha_{\ell,q+\alpha/(1+2s)}},\\
&\lesssim  d_\ell( z_1, z_2)^{\alpha'} (1+|v_0|)^{\alpha/(1+2s)} \|f\|_{C^\alpha_{\ell,q+\alpha/(1+2s)}}.
\end{align*}

The factor $N_1$ is estimated similarly in terms of $\|\varphi g\|_{C^\alpha_{\ell,q+\alpha/(1+2s)}} \lesssim \|g\|_{C^\alpha_{\ell,q+\alpha/(1+2s)}}$.

Thus, we conclude
\[ |h(z_1) - h(z_2)| \lesssim \|g\|_{C^\alpha_{\ell,q+\alpha/(1+2s)}} \|f\|_{C^\alpha_{\ell,q+\alpha/(1+2s)}} (1+|v_0|)^{-q+\gamma+\alpha/(1+2s)} d_\ell( z_1,  z_2)^{\alpha'}\]
from which we get the desired estimate. 
\end{proof}

\begin{cor} \label{c:bump_integral_Calpha}
Let $f,g \in \Cpol^\alpha$ for some $\alpha \in (0,\min(1,2s)]$. Let $z_0$ such that $\mathcal E_1(z_0) \subset [0,\infty) \times \R^d \times \R^d$. Let $\varphi$ be the smooth bump function supported in $B_{|v_0|/8}$ with $\varphi=1$ in $B_{|v_0|/9}$. Let $\bar h: Q_1 \to \R$ be given by
\[ \bar h(z) := \int_{\R^d} \varphi(\bar v+w) g((\mathcal T_0 z) \circ (0,0,w)) K_{f,\mathcal T_0 z} (w) \dd w\]
with $\mathcal{T}_0 z  =(\bar t,\bar x,\bar v)$. 
Then $h \in C^{\alpha'}_\ell(Q_1)$ and for any $q>d+(\gamma+2s)$,
\[ [\bar h]_{C^{\alpha'}_\ell(Q_1)} \leq C \|g\|_{C^\alpha_{\ell,q+\alpha/(1+2s)}} \|f\|_{C^\alpha_{\ell,q+\alpha/(1+2s)}} (1+|v_0|)^{-q+\gamma+\alpha/(1+2s)},\]
for a constant $C$ depending on $q$, $d$, $\gamma$ and $s$.
\end{cor}
\begin{proof} Apply Lemma~\ref{l:bump_integral_Calpha_pre} and observe $\bar h = h \circ \mathcal T_0$. Then use Lemma~\ref{l:holder-cov} to conclude.
\end{proof}

We are now in position to derive the desired estimate for $\Q_1 (f,g)$. 

\begin{lemma}[Bound for $\Q_1$] \label{l:Q1_Calpha}
Let $f \in \Cpol^\alpha$ and $g \in \Cpol^{2s+\alpha}$ for some $0<\alpha \leq \min(1,2s)$. Then $\Q_1(f,g) \in \Cpol^{\alpha'}$ for $\alpha' = 2s \alpha / (1+2s)$. Moreover, for any $q>d+\gamma+2s$,
\[ \|\Q_1(f,g)\|_{C^{\alpha'}_{\ell,{q-\gamma-2s-\alpha/(1+2s)}}} \leq C  \|f\|_{C^\alpha_{\ell,q}}  \|g\|_{C^{2s+\alpha}_{\ell,q}} .\]
Here, the constant $C$ depends only on $\alpha$, dimension, $s$, $\gamma$, and the constants in $B$ (the Boltzmann kernel).
\end{lemma}

\begin{proof}
  The norm $\|\Q_1(f,g)\|_{C^0_{\ell,{q-\gamma-2s}}}$ is already controlled by Lemma~\ref{l:Q1_C0}. We are left with estimating the semi-norm $[\Q_1(f,g)]_{C^{\alpha'}_{\ell,{q-\gamma-2s}}}$.

  Let $z_0=(t_0,x_0,v_0)$ be so that $Q_r(z_0) \subset [0,T] \times \R^d \times \R^d$ for some $r\in (0,1)$ like in Definition~\ref{d:holder-space-fast} and
let $\varphi$ be the cutoff function as in \eqref{e:bump_function}. From Lemma~\ref{l:bump_integral_Calpha_pre}, we know that 
\begin{equation} \label{e:q1_ca1}
 [\Q_1(f,\varphi g)]_{C^{\alpha'}_{\ell}(Q_r(z_0))} \lesssim \|g\|_{C^\alpha_{\ell,q}} \|f\|_{C^\alpha_{\ell,q}} (1+|v_0|)^{-q+\gamma+2\alpha/(1+2s)}.
\end{equation}

In order to estimate $\|\Q_1(f,(1-\varphi) g)\|_{C^{\alpha'}_{\ell}(Q_r(z_0))}$, let us consider two points $z_1, z_2 \in Q_r(z_0)$. We have
\begin{align*} 
 \Q_1(f,(1-\varphi) g)(z_2) - \Q_1(f,(1-\varphi) g)(z_1) &= \mathcal L_{K_{f,z_2}} [(1-\varphi)g](z_2) -  \mathcal L_{K_{f,z_1}} [(1-\varphi)g](z_1), \\
&= \left( \mathcal L_{K_{f,z_2}} [(1-\varphi)g](z_2) - \mathcal L_{K_{f,z_2}} [(1-\varphi)g](z_1) \right)\\
& \phantom{=} + \left( \mathcal L_{K_{f,z_2}} [(1-\varphi)g](z_1) -  \mathcal L_{K_{f,z_1}} [(1-\varphi)]g(z_1) \right).
\end{align*}
In the first term we are fixing the kernel $K_{f,z_2}$ (freezing coefficients) and evaluating in the function $(1-\varphi)g$ at two points $z_2$ and $z_1$. In the second term, we are evaluating the operator at the same point $z_1$, for the same function $(1-\varphi)g$, and we will obtain cancellation from the two kernels $K_{f,z_2} - K_{f,z_1}$.

For estimating the first term, we use  \cite[Lemma 3.6]{schauder}. It gives us that
\begin{align*} 
| \mathcal L_{K_{f,z_2}} [(1-\varphi)g](z_2) - \mathcal L_{K_{f,z_2}} [(1-\varphi)g](z_1)| &\lesssim \Lambda_{K_{f,z_2}} [(1-\varphi) g]_{C^{2s+\alpha'}_\ell} d_\ell(z_1,z_2)^{\alpha'}, \\
&\lesssim (1+|v_0|)^{-q+\gamma+2s} \|f\|_{C^0_{\ell,q}} \|g\|_{C^{2s+\alpha'}_{\ell,q}} d_\ell(z_1,z_2)^{\alpha'},
\end{align*}
using \eqref{e:upper_bound_for_Kf} and Lemma~\ref{l:convolution-C0} to get the second inequality.

For the second term, we use Lemma~\ref{l:pointwise_operator_bound}, and compute
\begin{align*} 
| \mathcal L_{K_{f,z_2}} &[(1-\varphi)g](z_1) - \mathcal L_{K_{f,z_1}} [(1-\varphi)g](z_1) | \lesssim \Lambda_{(K_{f,z_2}-K_{f,z_1})} \|(1-\varphi) g\|_{C^{2s+\alpha'}_\ell}, \\
\intertext{using \eqref{e:bump_norm_decay} and estimating $\Lambda_{(K_{f,z_2}-K_{f,z_1})}$ from \eqref{e:upper_bound_for_Kf},}
& \lesssim (1+|v_0|)^{-q} \|g\|_{C^{2s+\alpha'}_{\ell,q}} \left( \int_{\R^d} | f(z_2 \circ (0,0,w))-f(z_1 \circ (0,0,w)) | |w|^{\gamma+2s} \dd w \right) 
\end{align*}

We proceed like in the proof of Lemma~\ref{l:cdv-coef-holder} to estimate the integral. Using \eqref{e:repeatedly}, we have that
\[ | f(z_2 \circ (0,0,w))-f(z_1 \circ (0,0,w)) | \lesssim \min( 1, (d_\ell(z_1,z_2) + d_\ell(z_1,z_2)^{2s/(1+2s)} |w|^{1/(1+2s)})^\alpha ) (1+|v_1+w|)^{-q} \|f\|_{C^\alpha_{\ell,q}}\]
from which we get,
\[ \int_{\R^d} | f(z_2 \circ (0,0,w))-f(z_1 \circ (0,0,w)) | |w|^{\gamma+2s} \dd w \lesssim d_\ell(z_1,z_2)^{\alpha'} \|f\|_{C^\alpha_{\ell,q}} (1+|v_0|)^{\alpha/(1+2s)},\]
provided that $q > d+\gamma+2s+\alpha/(1+2s)$. Incorporating this inequality in the computation above, we get
\begin{align*} 
\big| \mathcal L_{K_{f,z_2}} &[(1-\varphi)g](z_1) - \mathcal L_{K_{f,z_1}} [(1-\varphi)g](z_1) \big|\lesssim
(1+|v_0|)^{-q+\alpha/(1+2s)} \|g\|_{C^{2s+\alpha'}_{\ell,q}} \|f\|_{C^\alpha_{\ell,q}} d_\ell(z_1,z_2)^{\alpha'}. \end{align*}
Collecting the two upper bounds above,
\[  |\Q_1(f,(1-\varphi) g)(z_2) - \Q_1(f,(1-\varphi) g)(z_1)| \lesssim (1+|v_0|)^{-q+\gamma+2s+\alpha/(1+2s)} d_\ell(z_1,z_2)^{\alpha'} \|f\|_{C^\alpha_{\ell,q}} \|g\|_{C^{2s+\alpha'}_{\ell,q}}. \]
Combining with \eqref{e:q1_ca1} and using that $\alpha/(1+2s) < 2s$,
\[  |\Q_1(f,g)(z_2) - \Q_1(f,g)(z_1)| \lesssim (1+|v_0|)^{-q+\gamma+2s+\alpha/(1+2s)} d_\ell(z_1,z_2)^{\alpha'} \|f\|_{C^\alpha_{\ell,q}} \|g\|_{C^{2s+\alpha'}_{\ell,q}}. \]
Thus, $[ \Q_1(f,g) ]_{C^{\alpha'}_\ell(Q_1(z_0))} \leq (1+|v_0|)^{-q+\gamma+2s+\alpha/(1+2s)} \|f\|_{C^\alpha_{\ell,q}} \|g\|_{C^{2s+\alpha'}_{\ell,q}}$ and we concluded the proof.
\end{proof}

\begin{remark}
The lemmas in Section~\ref{s:bilinear} allow us to estimate the H\"oder norm of $Q(f,g)$ in terms of the norms of $f$ and $g$. When it comes to global H\"older norms, we obtain certain decay exponent ``$q$'' in each of the lemmas. The precise value (and loss) of decay exponent is not computed sharply because it was not necessary for the purpose of the result in this paper. It might make sense to investigate sharper version of the lemmas in this section if one tries to obtain $C^\infty$ estimates as in Theorem~\ref{t:main2} for $\gamma \leq 0$ but for solutions that do not decay rapidly as $|v| \to \infty$.
\end{remark}

\section{Global H\"older estimates for the Boltzmann equation}

In this section, we derive global H\"older estimates and global Schauder estimates from the local ones derived in the previous sections. 

\subsection{Global H\"older estimates}

The H\"older estimate in Theorem~\ref{t:local-holder} applies directly to the Boltzmann equation. However, in doing so, it leads to a H\"older estimate only locally, that is to say for a compact set of velocities. In order to obtain a global estimate (that holds for $v \in \R^d$) we combine Theorem~\ref{t:local-holder} with the change of variables described in Section~\ref{s:cov}.

%-----------------------------------------------------------------------
\begin{prop}[Global H\"older estimate] \label{p:global-holder} Let $f$
  be a solution of the Boltzmann
  equation~\eqref{e:boltzmann} in $(0,T) \times \R^d \times \R^d$ so that Assumption~\ref{a:hydro-assumption} holds in its full domain. Assume that for some $q>d$, there exists  $N_q>0$ such that for all $(t,x,v) \in (0,T) \times \R^d \times \R^d$, 
  \[ f(t,x,v) \leq N_q (1+|v|)^{-q}\]
(this is the same as to say $N_q = \|f\|_{C^0_{\ell,q}}$).

  Let us set $z_0=(t_0,x_0,v_0) \in (\tau,T) \times \R^d \times \R^d$ with
  $|v_0| \ge 2$, to be the center of the change of variables $\mathcal T_0$ from Section~\ref{s:cov}.  Then for all $r \in (0,1)$ such that
  $\mathcal{E}_r (z_0) \subset (\tau,T) \times \R^d \times \R^d$ and
  all $\bar z_1, \bar z_2 \in \mathcal{E}_{r/2}(z_0)$,
  \[
    | f(\bar z_1) - f(\bar z_2)| \leq C(N_q)  (1+|v_0|)^{-q+\gamma_+} d_\ell(z_1,z_2)^\alpha
  \]
  where $C>0$ and $\alpha \in (0,1)$ only depend on $q$, $N_q$, the parameters in Assumption~\ref{a:hydro-assumption}, dimension $d$, $\gamma$, $s$ from
  \eqref{assum:B} and $\tau$. Here $\bar z_i = \mathcal T_0 z_i$.
\end{prop}
%-----------------------------------------------------------------------
\begin{remark}
Our solutions $f$ will be in $\Cpol^0$ with semi-norms controlled by Theorem~\ref{t:decay}.
\end{remark}

\begin{remark}
  When $t_1 = t_2$ and $x_1=x_2$, $d_\ell(z_1,z_2)$ is the same as $|v_1-v_2|$. It is exactly comparable to the non-isotropic
  distance $\dgs(\bar v_1, \bar v_2)$ of Gressman and Strain as defined in \cite{gressman2011global} (since $r<1$); see Lemma~\ref{l:da}
  in the appendix.

Proposition~\ref{p:global-holder} gives us a global H\"older estimate in all variables $t$, $x$ and $v$. 
\end{remark}
% -------------
We will use in Corollary~\ref{c:schauder_boltzmann} below the following straightforward consequence of the sharp global H\"older estimate from Proposition~\ref{p:global-holder}.
%--------------------
\begin{cor}[H\"older estimate with fast decay]\label{c:holder}
  Let $f$ be a solution of the Boltzmann
  equation~\eqref{e:boltzmann} in $(0,T) \times \R^d \times \R^d$ so that Assumption~\ref{a:hydro-assumption} holds for all
  $(t,x) \in (0,T) \times \R^d$. Moreover, assume $f \in \Cpol^0$.

  Then, there is an $\alpha>0$ so that for all $\tau \in (0,T)$, $f \in \Cpol^\alpha$ with
\[ \|f\|_{C_{\ell,q}^\alpha( (\tau,T) \times \R^d \times \R^d)} \le C_q \|f\|_{C^0_{\ell,q+\tilde q(\gamma,s)} ( (0,T) \times \R^d \times \R^d)} \qquad \text{ for all $q>d$.}\]
The constant $ C_q>0$ depends on $q$, the parameters from Assumption~\ref{a:hydro-assumption}, dimension $d$, $\gamma$ and $s$ from
  \eqref{assum:B} and time $\tau$. The value of $\alpha>0$  depends on the constants in Assumption~\ref{a:hydro-assumption}, $d$, $\gamma$ and  $s$. The value of $\tilde q(\gamma,s)$ depends on $s$ and $\gamma$ only.
\end{cor}
% ---------------------
We now turn to the proof of the previous proposition. 
%---------------------
\begin{proof}[Proof of Proposition~\ref{p:global-holder}]
Without loss of generality, we assume $\tau \geq 1$ and take $r=1$. Otherwise we would have to adjust the choice of $r$ to be $\leq \tau^{1/(2s)}$ so that $Q_r(z_0) \subset (0,T) \times \R^d \times \R^d$. The constants in the result would be affected by this value accordingly.

  Let $\varphi$ be the cutoff function as in \eqref{e:bump_function}. In particular $\varphi =0$ in $E_1(v_0)$ for $|v_0|\ge 2$ because $E_1 (v_0) \subset B_1 (v_0)$. Let $g(t,x,v) = (1-\varphi(v))  f(t,x,v)$. Thus, $\|g\|_{L^\infty} \le N_q (1+|v_0|)^{-q}$.

By a direct computation, we observe that $g$ solves the equation
\[ 
\partial_t g + v \cdot \nabla_x g = \LL_{K_f}(g) 
+ h_1 + h_2 \qquad \text{ in }  (0,T) \times \R^d \times E_1(v_0) ,
\]
where
\[
h_1 = \int_{\R^d} \varphi(v') f(v') K_f(v,v') \dd v'  
\qquad \text{and} \qquad  
h_2 = \Q_2(f,f) =c_b (f \ast |\cdot|^\gamma)f.
\]

Recall that $\varphi$ is supported in $B_{|v_0|/8}$. Thus, we apply Lemma~\ref{l:bump_integral_C0} and obtain
\begin{align*}
 |h_1| & \lesssim \|\varphi f\|_{L^1_v} N_q  (1+|v_0|)^{-q+\gamma} \\
 & \lesssim M_0   N_q  (1+|v_0|)^{-q+\gamma}
\end{align*}
in $\mathcal E_1(z_0) \subset Q_1 (z_0)$. 

In order to estimate $h_2$, we note that if $\gamma \geq 0$, then Lemma~\ref{l:convolution-moments} implies that $|\cdot|^\gamma \ast_v f \lesssim (1+|v_0|)^\gamma (M_0+E_0)$. On the other hand, if $\gamma < 0$, then Lemma~\ref{l:convolution-C0} implies that $|\cdot|^\gamma \ast_v f \lesssim N_q (1+|v_0|)^{\gamma}$, provided that $q>d$. Thus
\[
|h_2| \lesssim \begin{cases}
(1+|v_0|)^{-q+\gamma} N_q, & \text{ if } \gamma \geq 0,\\
(1+|v_0|)^{-q+\gamma} N_q^2 & \text{ if } \gamma < 0,
\end{cases}
\] 
in $\mathcal{E}_1(z_0) \subset Q_1 (z_0)$.  

Applying the change of variables $\To$ from \eqref{e:T0} and Theorem~\ref{t:change-of-vars}, we have that the function
$\bar g = g \circ \To $ solves
\[ 
\partial_t \bar g + v \nabla_x \bar g =  \LL_{\bar K_f} \bar g + \bar h \qquad \text{ in }  Q_1
\]
with 
\[  \bar h := |v_0|^{-\gamma-2s} \left(  h_1(\To(t,x,v)) + h_2( \To(t,x,v) \right)  \]
and $\bar K_f$ satisfies  ellipticity conditions~\eqref{e:nondeg1} (only if $s < \frac12$), \eqref{e:nondeg2}, \eqref{e:bounded1}, \eqref{e:bounded2}, \eqref{e:cancellation1}, \eqref{e:cancellation2} (only if $s \geq \frac12$). Note that Assumption~\ref{e:extra-int} in Theorem~\ref{t:change-of-vars} holds with $C_\gamma \lesssim N_q$ because $q > d$, recall Lemma~\ref{l:convolution-C0}.

Applying Theorem~\ref{t:local-holder} to $\bar g$, we get for all
$z_1, z_2 \in Q_{1/2}$,
\begin{align*} 
 |\bar g(z_1) - \bar g(z_2)| &
\le C  \left( \|\bar g\|_{L^\infty([-1,0]\times B_1 \times \R^d)} + \|\bar h\|_{L^\infty(Q_1)} \right) d_\ell(z_1, z_2)^\alpha, \\
&\le C \bigg( N_q (1+|v_0|)^{-q} + (N_q + N_q^2) (1+|v_0|)^{-q+\gamma} \bigg) d_\ell(z_1, z_2)^\alpha \\
&\le C (N_q) (1+|v_0|)^{-q+\gamma_+} d_\ell(z_1, z_2)^\alpha.
\end{align*}

This estimate yields the result since $\bar g = g \circ \To$ and $\To z_i = \bar z_i$ for $i=1,2$.
\end{proof}

\subsection{Global Schauder estimates for the Boltzmann equation}
\label{s:initial_schauder_for_Boltzmann}

Our next task is to use the change of variables in order to derive global Schauder estimates for the Boltzmann equation. In this case, we work with a more general equation, the \emph{linear} Boltzmann equation,
\begin{equation}
  \label{e:lin-boltzmann}
  (\partial_t + v \cdot \nabla_x) g = \Q_1(f,g) +h \text{ in } (0,T) \times \R^d \times \R^d.
\end{equation}
Theorem~\ref{t:local-schauder} gives us local Schauder estimates for the solution $g$, with a precise exponent, in terms of H\"older norms of $h$, $g$ and the kernel $K_f$. We will combine it with the change of variables described in Section~\ref{s:cov} in order to obtain global Schauder estimates.

We should not be deceived by the description of \eqref{e:lin-boltzmann} as a \emph{linear} equation. Proposition~\ref{p:global_schauder} below applies whenever a function $g$ satisfies any equation of that form, for any functions $f$ and $h$. Whether the functions $f$, $g$ and $h$ are related to each other or not is irrelevant for the estimates. In particular, if $f=g$ and $h = \Q_2(f,f)$, the estimate in Proposition~\ref{p:global_schauder} applies to the original (nonlinear) Boltzmann equation. Equation \eqref{e:lin-boltzmann} will also be satisfied when $g$ is a directional derivative of $f$ or some incremental quotient, for an appropriate $h$ in each case. In that sense, an estimate for \eqref{e:lin-boltzmann} as in Proposition~\ref{p:global_schauder} is more general than a Schauder estimate for merely the original Boltzmann equation \eqref{e:boltzmann}.

% --------------------------------------------------------------------------------------------------
\begin{prop}[Global Schauder estimates]\label{p:global_schauder}
  Let $f : (0,T) \times \R^d \times \R^d \to [0,\infty)$ be such that Assumption~\ref{a:hydro-assumption} holds. Assume also that $f \in \Cpol^{\alpha} $ for some  $\alpha \in (0,\min (1,2s))$.  Let $g \in \Cpol^{\alpha} $ be a
  solution of \eqref{e:lin-boltzmann} with $h \in \Cpol^{\alpha'}$
  with $\alpha' = \frac{2s}{1+2s} \alpha$ and $2s+\alpha' \notin \{1,2\}$. Then for all $\tau > 0$, we have the following a priori estimate for $g$ in $\Cpol^{2s+\alpha'}([\tau,T] \times \R^d \times \R^d)$, for each $q > d+2+2s$,
  \[ 
    \|g\|_{C_{\ell,q}^{2s+\alpha'} ([\tau,T] \times \R^d \times \R^d) } \le C \left( \|g\|_{C_{\ell,q+\kappa}^\alpha ([0,T] \times \R^d \times \R^d))}
    + \|h\|_{C_{\ell,q+\kappa}^{\alpha'} ([0,T] \times \R^d \times \R^d) }  \right) 
  \]
  where the constant $\kappa$ depends on $s$ and $\gamma$ only, and $C$ depends on $r,q$, dimension $d$, parameters
  $s, \gamma$ in \eqref{assum:B}, $m_0,M_0,E_0,H_0$ from Assumption~\ref{a:hydro-assumption}, $\tau$ and $\|f\|_{C_{\ell,q+\kappa}^{\alpha}(Q_r(z_0))}$.
\end{prop}
% --------------------------------------------------------------------------------------------------
\begin{proof}

Like in the proof of Proposition~\ref{p:global-holder}, we concentrate on $|v_0|>2$ and assume without loss of generality that $\tau \geq 1$. Let us pick any $z_0$ so that $\mathcal E_1(z_0) \subset [0,T] \times \R^d \times \R^d$.

Let $\varphi$ be the cutoff function as in \eqref{e:bump_function}.

We multiply $g$ by $(1-\varphi)$ in order to concentrate on velocities $|v| \geq |v_0|/9$. Then, we change variables by looking at $\bar g = [(1-\varphi) g] \circ \mathcal T_0$. Recall that $\mathcal T_0$ maps $Q_1$ into the slanted ellipsoidal cylinder $\mathcal E_1(z_0)$. The function $\bar g$ satisfies the following equation in $Q_1$,
\[ \partial_t \bar g + v \cdot \nabla_x \bar g = \left( \int_{\R^d}(\bar g'- \bar g) \bar K_f(t,x,v,v') \dd v' \right) + \bar h + \bar h_2.\]
Here, $\bar K_f$ is the kernel after the change of variables, as in \eqref{e:barKf}. 

The function $\bar h$ corresponds to the source term $h$ after the change of variables. The new source term $\bar h_2$ is the result of  our application of the cutoff factor $(1-\varphi)$. The functions $\bar h$ and $\bar h_2$ are given by
\begin{align*}
\bar h &:= |v_0|^{-\gamma-2s} h \circ \mathcal T_0, \\
\bar h_2 &:= |v_0|^{-\gamma-2s} \int_{\R^d}\varphi(v') g(\bar t, \bar x, v') K_f(\bar t, \bar x, \bar v, v') \dd v'.
\end{align*}
As usual, we write $\bar z = (\bar t, \bar x, \bar v) = \mathcal T_0 z$.

According to the Schauder estimates of Theorem~\ref{t:local-schauder}, we get
\begin{align*}
 [\bar g]_{C^{2s+\alpha'}_\ell(Q_1)} & \lesssim \left(1+\bar A_0^{\frac{2s+\alpha'-\alpha}{\alpha'}} \right) [\bar g]_{C^\alpha_\ell([-2^{2s},0] \times B_2 \times \R^d)}  + [\bar h + \bar h_2]_{C^{\alpha'}_\ell(Q_2)} + (1+\bar A_0) \|\bar h + \bar h_2\|_{C^0(Q_2)} , \\
&=: T_1 + T_2 + T_3.
\end{align*}

Since we know from Lemma~\ref{l:holder-cov} that $[g]_{C^{2s+\alpha'}_\ell(\mathcal E_1(z_0))} \lesssim {|v_0|^{\bar c (2s+\alpha')}}  [\bar g]_{C^{2s+\alpha'}_\ell(Q_1)}$, the proof of this proposition will proceed by estimating the right hand side in the inequality above. Let us estimate the three terms $T_1$, $T_2$ and $T_3$, one by one.

For the first term, let us observe that by the construction of $\varphi$ and the definition of the norm $C^\alpha_{\ell,q}$,
\[ [(1-\varphi) g]_{C^\alpha_\ell([0,T]\times \R^d \times \R^d)} \lesssim (1+|v_0|)^{-q}  \|g\|_{C^\alpha_{\ell,q}([0,T]\times \R^d \times \R^d)}.\]
Combining with the change of variables and using Lemma~\ref{l:holder-cov}, 
\begin{equation} \label{e:gs_e1}
 [\bar g]_{C^\alpha_\ell([-2^{2s},0] \times B_2 \times \R^d)} \leq (1+|v_0|)^{-q_1}  \|g\|_{C^\alpha_{\ell,q_1}([0,T]\times \R^d \times \R^d)}.
\end{equation}
The estimate in \eqref{e:gs_e1} holds for any value of $q_1>0$.

Using Lemma~\ref{l:cdv-coef-holder}, we have that, for any $q_1 > d+2+\alpha/(1+2s)$
\[ \bar A_0 \lesssim (1+|v_0|)^{\frac \alpha {1+2s}(1-2s-\gamma)_+} \|f\|_{C^\alpha_{\ell,q_1}}.\]
Combining it with \eqref{e:gs_e1}, we estimate the first term $T_1$ as
\begin{align*} 
 T_1 &\lesssim |v_0|^{-q_1}  \|g\|_{C^\alpha_{\ell,q_1}} + |v_0|^{-q_1 + \frac \alpha {1+2s}(1-2s-\gamma)_+ \frac{2s+\alpha'-\alpha}{\alpha'}} \|f\|_{C^\alpha_{\ell,q_1}}^{\frac{2s+\alpha'-\alpha}{\alpha'}} \|g\|_{C^\alpha_{\ell,q_1}}, \\
&= |v_0|^{-q_1}  \|g\|_{C^\alpha_{\ell,q_1}} + |v_0|^{-q_1 + (1-2s-\gamma)_+ \left( 1 - \frac{\alpha}{2s (1+2s)} \right) } \|f\|_{C^\alpha_{\ell,q_1}}^{\frac{2s+\alpha'-\alpha}{\alpha'}} \|g\|_{C^\alpha_{\ell,q_1}},\\
&\le |v_0|^{-q_1}  \|g\|_{C^\alpha_{\ell,q_1}} + |v_0|^{-q_1 + 1 } \|f\|_{C^\alpha_{\ell,q_1}}^{\frac{2s+\alpha'-\alpha}{\alpha'}} \|g\|_{C^\alpha_{\ell,q_1}}.
\end{align*}
This is true for any $q_1 > d+2+\alpha/(1+2s)$.

For the other terms, we must estimate the $C^{\alpha'}_\ell$ norms of $\bar h$ and $\bar h_2$. In the case of $\bar h$, we simply observe that
\begin{align*} 
 \|\bar h\|_{C^0(Q_2)} &= |v_0|^{-\gamma-2s} \|h\|_{C^0(\mathcal E_2(z_0))} , \\ 
 &\leq |v_0|^{-\gamma-2s-q_1} \|h\|_{C^0_{\ell,q_1}}. \\
[ \bar h ]_{C^{\alpha'}_\ell(Q_2)} &\leq |v_0|^{-\gamma-2s} [h]_{C^{\alpha'}_\ell(\mathcal E_2(z_0))} , \\ 
 &\leq |v_0|^{-\gamma-2s-q_1} [h]_{C^{\alpha'}_{\ell,q_1}}.
\end{align*}

In the case of $\bar h_2$, we apply Lemma~\ref{l:bump_integral_C0} (with $q=q_1$) and Corollary~\ref{c:bump_integral_Calpha} (with $q=q_1-\alpha / (1+2s)$) and obtain  for any $q_1 > d+(\gamma+2s)$,
\begin{align*}
\|\bar h_2\|_{C^0_\ell(Q_2)} &\lesssim |v_0|^{-q_1-2s} \|g\|_{C^0_{\ell,q_1}} \|f\|_{C^0_{\ell,q_1}} , \\
  [\bar h_2]_{C^{\alpha'}_\ell(Q_2)} &\lesssim |v_0|^{-q_1-2s+2\alpha/(1+2s)} \|f\|_{C^\alpha_{\ell,q_1}} \|g\|_{C^\alpha_{\ell,q_1}}  .
\end{align*}
We use these estimates to obtain upper bounds for $T_2$ and $T_3$.
\begin{align*}
T_2 &\lesssim |v_0|^{-\gamma-2s-q_1} [h]_{C^{\alpha'}_{\ell,q_1}} + |v_0|^{-q_1-2s+2\alpha/(1+2s)} \|f\|_{C^\alpha_{\ell,q_1}} \|g\|_{C^\alpha_{\ell,q_1}} ,\\
T_3 &\lesssim \left(1+ |v_0|^{\frac \alpha {1+2s}(1-2s-\gamma)_+} \|f\|_{C^\alpha_{\ell,q_1}} \right) \left(  |v_0|^{-\gamma-2s-q_1} \|h\|_{C^0_{\ell,q_1}}  + |v_0|^{-q_1-2s} \|g\|_{C^0_{\ell,q_1}} \|f\|_{C^0_{\ell,q_1}} \right).
\end{align*}

Finally, using Lemma~\ref{l:holder-cov},
\begin{equation} \label{e:gs-a1}
 [g]_{C^{2s+\alpha'}_\ell (\mathcal E_1(z_0))} \lesssim |v_0|^{\bar c(2s+\alpha')} [\bar g]_{C^{2s+\alpha'}_\ell(Q_1)} \lesssim |v_0|^{\bar c(2s+\alpha')}(T_1+T_2+T_3).
\end{equation}

Note that $\mathcal E_1(z_0) \supset Q_{|v_0|^{-\tilde c}} (z_0)$ with $\tilde c = \max (1, (\gamma+2s)/(2s))$ and \eqref{e:gs-a1} holds at any point $z_0$. Using Lemma~\ref{l:Calpha_local}, we extend the inequality to the larger domain $Q_1(z_0)$,
\begin{align*} 
 [g]_{C^{2s+\alpha'}_\ell ( Q_1(z_0))} &\lesssim |v_0|^{\bar c (2s+\alpha')} (T_1+T_2+T_3) + |v_0|^{\tilde c (2s+\alpha')} \|g\|_{C^0_{\ell}(Q_1(z_0))}.  \\
\intertext{Collecting all inequalities, not tracking the dependence on $\|f\|_{C^{\alpha}_{\ell,q_1}}$, and keeping only the largest exponents of $|v_0|$, we are left with }
&\leq C  \left(  |v_0|^{-q_1+\kappa} \|g\|_{C^\alpha_{\ell,q_1}} +  |v_0|^{-q_1+\kappa} \|h\|_{C^{\alpha'}_{\ell,q_1}} \right).
\end{align*}
Here, the constant $C$ depends on $\|f\|_{C^{\alpha}_{\ell,q_1}}$ and $\kappa$ depends on $s$ and $\gamma$ only.

For any given value of $q$, we pick $q_1 = q + \kappa$ and conclude the proof of the lemma.
\end{proof}

\begin{remark}
In Proposition~\ref{p:global_schauder}, we obtain a priori estimates for the norms $\|g\|_{C^{2s+\alpha'}_{\ell,q}}$ in terms of $\|g\|_{C^{\alpha}_{\ell,q_1}}$, $\|f\|_{C^{\alpha}_{\ell,q_1}}$ and $\|h\|_{C^{\alpha'}_{\ell,q_1}}$ for $q_1 = q+\kappa$. Note that we gain some regularity in the estimate but we lose some decay from $q$ to $q_1$. We have made no effort to make the choice of $q_1$ as $q+\kappa$ optimal. Since we work with functions that have a rapid decay as $|v| \to \infty$, the precise exponent in the loss of decay in the estimate has no consequence for our proof.
\end{remark}

\begin{remark}
Following the proof of Proposition~\ref{p:global_schauder} one can compute how the constant $C$ depends on $\|f\|_{C^\alpha_{\ell,q+\kappa}}$. We get $C \approx \|f\|_{C^\alpha_{\ell,q+\kappa}}^{\frac{\alpha+2s-\alpha'}{\alpha'}} = \|f\|_{C^\alpha_{\ell,q+\kappa}}^{1/(2s) + (1+2s)/\alpha}$.
\end{remark}

\begin{cor} \label{c:schauder_boltzmann}
Let $f$  be a  solution of the Boltzmann  equation~\eqref{e:boltzmann} in $(0,T) \times \R^d \times \R^d$ so that Assumption~\ref{a:hydro-assumption} holds. If $\gamma \leq 0$, assume further that $f(0,x,v) = f_0(x,v)$ with
\[ 0 \leq f_0(x,v) \leq N_q (1+|v|)^{-q},\]
for all non-negative integer $q$.

Then, for some $\alpha>0$ and every $q \in \mathbb N$, the norm $\|f\|_{C^{2s+\alpha}_{\ell,q}((\tau,T)\times \R^d \times \R^d)}$ is bounded depending only $d$, $\gamma$, $s$, $\tau$, the parameters in Assumption~\ref{a:hydro-assumption}, and the values of $N_q$ (if $\gamma \leq 0$).
\end{cor}

\begin{proof}
Applying Theorem~\ref{t:decay}, we get an estimate for the norms $\|f\|_{C^0_{\ell,q}(\tau/3,T) \times \R^d \times \R^d)}$ for any value of $q \in \mathbb N$.

Applying Corollary~\ref{c:holder}, we get an estimate for the norms $\|f\|_{C^\alpha_{\ell,q}((\tau/2,T) \times \R^d \times \R^d)}$ for any value of $q \in \mathbb N$, and some small $\alpha>0$.

Applying Proposition~\ref{p:global_schauder} to $g=f$ and $h=f$, we conclude the proof of the corollary.
\end{proof}

\section{Increments}

\label{s:incremental_quotients}

In order to bootstrap the regularity estimate from Corollary~\ref{c:schauder_boltzmann}, we will apply the global Schauder estimates from Proposition~\ref{p:global_schauder} to derivatives and increments of the solution $f$ iteratively.

Before doing that, we develop some technical lemmas about increments and H\"older norms in this section.

Let us write
\[
  \Delta_y f (z) = f(z \circ (0,y,0)) - f(z)  \quad \text{ and } \quad \Delta_w f (z) = f(z\circ (0,0,w))-f(z)
\]
for some small increments $y \in \R^d$ and $w \in \R^d$.
Roughly speaking, the global Schauder estimate from Proposition~\ref{p:global_schauder} allows us to gain only $2s$ derivatives at each iteration, which can be less than $1$ if $s<1/2$. In order to gain one full derivative in each variable, we will apply this estimate to increments of $f$ as above.

The following fact about (usual) H\"older spaces is commonly used to study the regularity of solutions to nonlinear equations (see \cite[Lemma~5.6]{cc}). If $f: \R \to \R$ is a $C^\alpha$ function, and the $C^\alpha$ semi-norm of the increments $f(x+h)-f(x)$ is bounded above by $\lesssim |h|^\beta$, then $f$ is H\"older continuous with the larger exponent $\min(\alpha+\beta,1)$. In our current context of kinetic equations, the underlying geometry and Galilean invariance make the procedure more complicated. Here, we present two separate lemmas that involve increments in space and velocity respectively. They allow us to transfer a regularity estimate for an increment, into a higher order of differentiation. 

In spite of the apparent simplicity of the statement, the proof is rather involved. The first step in the proof is inspired by \cite[Lemma~5.6]{cc}.

%-------------------------------------------------------------------------
\begin{lemma}[Gaining regularity with $x$-increments]\label{l:cc-x}
  Let $\alpha_1,\alpha_2 >0$ and $\beta \ge 0$.  Given a cylinder  $Q= Q_R(z_0)$ with $R \in (0,1)$ and a bounded continuous function  $f$ defined in $Q$, we consider for any  $y \in B_{R^{1+2s}/2}$  the following  function,
  \[ \Delta_y f(z) = f(z \circ (0,y,0)) -f(z).\]
It is defined in $\Qint = Q_{R/2}(z_0)$.

We assume there exists an $N>0$ such that for all  $y \in B_{R^{1+2s}/2}$
\begin{equation} \label{e:Lemma5-ass-x}
    \|\Delta_y f\|_{C^0(\Qint)} \le N, \qquad \; [\Delta_y f]_{C_\ell^{\alpha_1 + \alpha_2} (\Qint)} \le N \|(0,y,0)\|^\beta.
\end{equation}

We assume that $\alpha_2 \in (0,\min (1,2s)]$, $\alpha_1+\alpha_2 \le 1+  2s$, $\alpha_1+\beta \leq 1+  2s$.  Then for all  $y \in B_{R^{1+2s}/2}$, 
  \[ \begin{cases}
      \big\| \Delta_y f \big\|_{C_\ell^{\alpha_2} (\Qint)} \lesssim N \|(0,y,0)\|^{\alpha_1+\beta} &  \text{ if } \alpha_1+\alpha_2+\beta \leq 1+ 2s, \\
      \big\| \Delta_y f \big\|_{C_\ell^{\eta} (\Qint)} \lesssim N \|(0,y,0)\|^{1+ 2s} & \text{ if $\alpha_1+\alpha_2+\beta>1+ 2s$},
\end{cases}
\]
 for some $\eta=\eta(\alpha_1,\alpha_2,\beta,s)>0$. 
\end{lemma}
%----------------------------------------------------
\begin{remark}
  One might expect that $\eta = \alpha_1 + \alpha_2 + \beta - (1+2s)$.
  Our proof gives us a smaller number $\eta>0$, with an explicit
  formula. We do not know if the value we obtain is sharp.
\end{remark}
%-------------------------------------------------------------------------
\begin{lemma}[Gaining regularity with $v$-increments]\label{l:cc-v}
  Let $\alpha_1,\alpha_2 >0$ and $\beta \ge 0$.  Given a cylinder   $Q= Q_R(z_0)$ with $R \in (0,1)$ and a bounded continuous function  $f$ defined in $Q$, we consider for any $w \in B_{R/2}$ the  following function,
  \[ \Delta_w f(z) = f(z \circ (0,0,w)) -f(z).\]
  It is defined in $\Qint = Q_{R/2}(z_0)$.

We assume there exists an $N>0$ such that for all  $w \in B_{R/2}$,
  \begin{equation} \label{e:Lemma5-ass-v}
    \|\Delta_w f\|_{C^0(\Qint)} \le N, \qquad \; [\Delta_w f]_{C_\ell^{\alpha_1 + \alpha_2} (\Qint)} \le N \|(0,0,w)\|^\beta.
\end{equation}

We assume that $\alpha_2 \in (0,\min (1,2s)]$, $\alpha_1+\alpha_2 \le 1$, $\alpha_1+\beta \leq 1$.  Then  for all  $w \in B_{R/2}$, 
  \[ \begin{cases}
      \big\| \Delta_w f \big\|_{C_\ell^{\alpha_2} (\Qint)} \lesssim N \|(0,0,w)\|^{\alpha_1+\beta} &  \text{ if } \alpha_1+\alpha_2+\beta \leq 1, \\ 
      \big\| \Delta_w f \big\|_{C_\ell^{\eta} (\Qint)} \lesssim N \|(0,0,w)\| & \text{ if $\alpha_1+\alpha_2+\beta>1$}
\end{cases}
\]
 for some $\eta=\eta(\alpha_1,\alpha_2,\beta,s)>0$. 
\end{lemma}
%----------------------------------------------------
\begin{proof}[Proof of Lemmas~\ref{l:cc-x} and \ref{l:cc-v}]
  Let $\iota= 1$ if $a=(0,y,0)$ and $\iota=0$ if $a=(0,0,w)$. We then consider  \( \Delta_a f (z) = f(z \circ a) -f(z) \)  so that $\Delta_a f$ equals either $\Delta_y f$ or $\Delta_w f$ if $a =(0,y,0)$ or $a =(0,0,w)$. 
  
  Let $p_z$ denote the polynomial expansion of $f$ at $z$ of kinetic  degree strictly smaller than $\alpha_1 + \alpha_2$. The assumptions   \eqref{e:Lemma5-ass-x} and \eqref{e:Lemma5-ass-v} implies in particular the  following: for all $z \in \Qint$ and $\xi$ such that  $z \circ \xi \in \Qint$,
\begin{equation}\label{e:assum-g}
    |\Delta_a f (z \circ \xi) - \delta_a p_z (\xi)| \leq N \|a\|^\beta \|\xi \|^{\alpha_1+\alpha_2}
\end{equation}
where $\Delta_a f (z) = f(z \circ a) -f(z)$ and where  $\delta_a p_z$ is the polynomial expansion of $\Delta_a f$ at the point $z$.

We abuse notation by writing $\Delta_y f = \Delta_{(0,y,0)} f$,  $\delta_y p_z = \delta_{(0,y,0)} p_z$,  $\Delta_w f = \Delta_{(0,0,w)}f$ and  $\delta_w p_z = \delta_{(0,0,w)}p_z$. Since $\alpha_2 \in (0,\min (1,2s))$, we aim at proving that for  $z \in \Qint$ and $a \in Q_{R/2}$ and $\xi$ such that  $z \circ \xi \in \Qint$,
\begin{eqnarray}
\label{e:conclusion}
    |\Delta_a f (z \circ \xi) - \Delta_a f (z)| \lesssim N \|a\|^{\alpha_1+\beta} \|\xi \|^{\alpha_2} , &\qquad \text{ if } \alpha_1+\alpha_2 + \beta \le 1 + \iota 2s  \\
\label{e:conclusion-weak}
    |\Delta_a f (z \circ \xi) - \Delta_a f (z)| \lesssim N \|a\|^{1 + \iota 2s} \|\xi \|^{\eta}, &\qquad \text{ if } \alpha_1+\alpha_2 + \beta > 1 + \iota 2s 
\end{eqnarray}
where $\iota = 1$ for $a =(0,y,0)$ and $\iota = 0$ if $a= (0,0,w)$.

The remainder of the proof proceeds in several steps. The first one is reminiscent of the proof of \cite[Lemma~5.6]{cc}.  \medskip

\textsc{Step~1.} We claim that for all  $z \in \Qint$ and all $k \in \mathbb{N}$ such that $z \circ (2^k a) \in \Qint$, we have
  \begin{align}
    \label{e:step1}
    |\Delta_a f (z) - 2^{-k} \Delta_{2^k a} f(z) | \lesssim &N \|a\|^{\beta+\alpha_1+\alpha_2} 2^{k \left( \frac{\beta+\alpha_1+\alpha_2}{1+ \iota 2s} - 1 \right)_+}.
  \end{align}
  
  In order to get such an estimate, we remark that
  \[ \Delta_{2a} f(z) =\Delta_a f(z) + \Delta_a f (z\circ a).\]
  Using \eqref{e:assum-g}, we thus get
  \begin{align*}
    |\Delta_{2a} f(z) - 2 \Delta_a f(z)|  & = |\Delta_a f (z\circ a) - \Delta_a f(z)| \\
                                          & \lesssim N \|a\|^{\alpha_1+\alpha_2+\beta} + |\delta_a p_z (a) - \Delta_a f (z)| .
  \end{align*}
  Since the polynomial $p_z$ is of degree strictly less than $\alpha_1+\alpha_2$, we have for $\xi = (\xi_t,\xi_x,\xi_v) \in \R^{1+2d}$, 
  \begin{equation}
    \label{e:polyz}
    \delta_a p_z (\xi) = \Delta_a f(z) + \underset{\text{if } \alpha_1 + \alpha_2 > 2s}{\underbrace{(\partial_t + v\cdot \nabla_x)\Delta_a f (z)\xi_t}}
    + \underset{\text{if } \alpha_1 + \alpha_2 > 1}{\underbrace{D_v \Delta_a f (z)\cdot \xi_v}}
    + \frac12\underset{\text{if } \alpha_1 + \alpha_2 > 2}{\underbrace{D_v^2 \Delta_a f (z) \xi_v \cdot \xi_v}}.
  \end{equation}
  In particular, since $\alpha_1 + \alpha_2 \le 1$ when $a = (0,0,w)$, we remark that, when evaluating the previous expression with $a = (0,y,0)$ at $\xi = (0,y,0)$ or with $a =(0,0,w)$ at $\xi =(0,0,w)$, 
  \[ \delta_y p_z ((0,y,0)) = \Delta_y f(z) .\]
  In the case $a = (0,0,w)$, we used the assumption $\alpha_1+\alpha_2 \le 1$. 
  
We thus conclude that
  \[ |\Delta_{2a} f(z) - 2 \Delta_a f(z)| \lesssim N \|a\|^{\alpha_1+\alpha_2+\beta}\]
  or equivalently
  \[ |\Delta_{a} f(z) - 2^{-1} \Delta_ {2a} f(z)| \lesssim 2^{-1} N \|a\|^{\alpha_1+\alpha_2+\beta}.\]
  By induction, we get
  \begin{align*}
    |\Delta_{a} f(z) - 2^{-k} \Delta_{2^k a} f(z)| & \lesssim N \sum_{j=1}^k2^{-j} \|2^{j-1}a\|^{\alpha_1+\alpha_2+\beta} \\
                                                   & \lesssim N \|a\|^{\alpha_1+\alpha_2+\beta} \sum_{j=1}^k 2^{-j + (j-1)\frac{\alpha_1+\alpha_2 + \beta}{1+\iota 2s}} \\
    & \lesssim N \|a\|^{\alpha_1+\alpha_2+\beta} 2^{k\left(\frac{\alpha_1+\alpha_2 + \beta}{1+ \iota 2s} -1\right)_+}.
  \end{align*}
  This achieves the proof of the claim.
  \medskip

  \textsc{Step~2.} We claim now that for $z \in \Qint$ and $a \in Q_{R/2}$, 
  \begin{equation}
    \label{e:step2}
    |\Delta_a f(z)| \lesssim N \|a\|^m
  \end{equation}
  with $m={\min (\alpha_1+\alpha_2+\beta,1+ \iota 2s)}$.   It is enough to pick an integer $k=k(a)$ such that  $\|2^k a\| \simeq 1$ (or equivalently $2^{-k} \simeq \|a\|^{1+\iota 2s}$), and apply Claim~\eqref{e:step1} from Step~1.

  Indeed, using the assumption $\alpha_1+\alpha_2 \le 1$ in the case $a=(0,0,w)$, we can write in both cases 
  \begin{align*}
    |\Delta_{a} f(z) | &\lesssim 2^{-k} | \Delta_{2^k a} f (z)| + N \|a\|^{\alpha_1+\alpha_2+\beta} 2^{k\left(\frac{\alpha_1+\alpha_2 + \beta}{1+\iota 2s} -1\right)_+} \\
                       &\lesssim \|\Delta_{2^k a} f\|_{C^0}  \|a\|^{1+\iota 2s} + N \|a\|^{\min (\alpha_1+\alpha_2+\beta,1+ \iota 2s)} \\
                   &\lesssim  N \|a\|^{\min (\alpha_1+\alpha_2+\beta,1+\iota 2s)}.
  \end{align*}

 Because $\alpha_1 + \beta \le 1+\iota 2s$ and $a \in Q_1$, we remark that this implies
  \(
    |\Delta_a f(z)| \lesssim N \|a\|^{\alpha_1 + \beta}.
  \)
We are thus left with estimating the semi-norm.
  \medskip

  \textsc{Step~3.}
  We next claim that for $z \in \Qint$ and  $a \in Q_{R/2}$,
  \begin{equation}
    \label{e:step3}
    \left\{
    \begin{aligned}
    | (\partial_t + v \cdot \nabla_x) \Delta_a f (z)| &\lesssim N \| a \|^{m -  (2s) \theta} & \text{if } \alpha_1 + \alpha_2 > 2s \\
    | D_v \Delta_a f (z)| &\lesssim N \| a \|^{m-\theta} & \text{if } \alpha_1 + \alpha_2 > 1 \\
      | D_v^2 \Delta_a f (z)| &\lesssim N \| a \|^{m-2\theta} & \text{if } \alpha_1 + \alpha_2 > 2
    \end{aligned}
    \right.
  \end{equation}
  where
  \begin{equation}\label{e:m-theta}
      m = \min (\alpha_1+\alpha_2 + \beta, 1+ \iota 2s) \qquad \text{ and } \qquad 
      \theta = \min \left(1,\frac{1+\iota 2s-\beta}{\alpha_1+\alpha_2} \right).
  \end{equation}

  It is a consequence of the assumption~\eqref{e:assum-g}, the
  estimate~\eqref{e:step2} from Step~2 and the interpolation
  inequality given by Proposition~\ref{p:interpol}.  For instance, in
  the case $a =(0,y,0)$ and if $\alpha_1 + \alpha_2 > 2s$, we have
  \begin{align*}
    [(\partial_t + v \cdot \nabla_x) \Delta_y f ]_{C_\ell^0 (\Qint)} & \lesssim [\Delta_y f]_{C_\ell^{2s} (\Qint)} \\
                                                                & \lesssim [\Delta_y f]_{C_\ell^0 (\Qint)}^{1- \frac{2s}{\alpha_1+\alpha_2}}
                                                                  [\Delta_y f]_{C_\ell^{\alpha_1+\alpha_2} (\Qint)}^{\frac{2s}{\alpha_1+\alpha_2}} + [\Delta_y f]_{C_\ell^0 (\Qint)}. \\
    & \lesssim N \| (0,y,0)\|^{\left(1-\frac{2s}{\alpha_1+\alpha_2}\right) m + \frac{2s}{\alpha_1+\alpha_2} \beta}.
  \end{align*}
  We now remark that $\left(1-\frac{2s}{\alpha_1+\alpha_2}\right) m + \frac{2s}{\alpha_1+\alpha_2} \beta = m -2s \theta$. 
   The other cases are treated similarly.  
  \medskip
  
  \textsc{Step 4.}
Let $z \in \Qint$ and $a \in Q_{R/2}$ and $\xi$ such that $z \circ \xi \in \Qint$. Assume $\|\xi\| \leq \|a\|$.  We derive from \eqref{e:polyz} and the previous step 
\begin{align*} 
     |\delta_a p_z (\xi) - \Delta_a f (z)| &\lesssim N \left( \|a\|^{m-2s \theta} \|\xi\|^{2s} + \|a\|^{m-\theta} \|\xi\| + \|a\|^{m-2\theta} \|\xi\|^2 \right),  \\
& = N \|a\|^m \left( \left( \frac {\|\xi\|}{\|a\|^\theta} \right)^{2s} + \left( \frac {\|\xi\|}{\|a\|^\theta} \right) + \left( \frac {\|\xi\|}{\|a\|^\theta} \right)^2 \right).
\end{align*}

Since   $\theta \leq 1$ and we are now focusing on the case $\|\xi\| \leq \|a\|$,
\begin{equation}
  \label{e:step4}
  |\delta_a p_z (\xi) - \Delta_a f (z)| \lesssim N \|a\|^m  \left( \frac{\|\xi\|}{\|a\|^\theta} \right)^{\min(1,2s)} = N \|a\|^{m-\theta \min(1,2s)} \|\xi\|^{\min(1,2s)}.
\end{equation}
Here $m$ and $\theta$ are given in \eqref{e:m-theta}.

\medskip

\textsc{Step 5.} Assume $\|a\| \gtrsim \|\xi\|$.

We first use \eqref{e:assum-g} with $\|a\| \gtrsim \|\xi\|$ to get
\begin{equation}
  \label{e:step5}
  |\Delta_a f (z\circ \xi) - \delta_a p_z(\xi)| \lesssim  N \|a\|^{\alpha_1+\beta} \|\xi \|^{\alpha_2}. 
\end{equation}

We claim that in the case $\alpha_1 + \alpha_2 + \beta \leq 1+2s$ for $a=(0,y,0)$ or $\alpha_1 + \alpha_2 + \beta \leq 1$ for $a=(0,0,w)$, \eqref{e:conclusion} holds true. Indeed, since  in this case $m = \alpha_1+\alpha_2+\beta$ and $\theta=1$ and we also have $\alpha_2 < \min (1,2s)$, in view of \eqref{e:step4} we get 
\[
  |\delta_a p_z (\xi) - \Delta_a f (z)|  \lesssim  N \|a\|^{\alpha_1+\alpha_2+\beta-\min(1,2s)} \|\xi\|^{\min(1,2s)} \lesssim N \|a\|^{\beta+\alpha_1} \|\xi \|^{\alpha_2}.
\]
Adding the previous two inequalities, we get \eqref{e:conclusion} for such $a$'s and $\xi$'s. 

For those values of $\alpha_1$, $\alpha_2$ and $\beta$ so that
$\theta \neq 1$, we obtain a somewhat weaker estimate. In this case, \eqref{e:step4} tells us that
\[
  |\delta_a p_z (\xi) - \Delta_a f (z)|  \lesssim N \|a\|^{1+\iota 2s - \min(1,2s) \theta} \|\xi\|^{\min(1,2s)} \lesssim N \|a\|^{1+\iota 2s} \|\xi\|^{\min(1,2s)(1-\theta)} .
\]
The last inequality holds because $\|\xi\| \lesssim \|a\|$.

In this case, using again that $\|\xi\| \lesssim \|a\|$, \eqref{e:step5} implies that
\[  |\Delta_a f (z\circ \xi) - \delta_a p_z(\xi)| \lesssim  N \|a\|^{\alpha_1+\beta} \|\xi \|^{\alpha_2} \lesssim N \|a\|^{1+ \iota 2s} \|\xi\|^{\alpha_2 -(1+\iota 2s-(\beta +\alpha_1))} \lesssim  N \|a\|^{1+ \iota 2s} \|\xi\|^{(1-\theta)(\alpha_1+\alpha_2)}. \]

Combining the two inequalities above, we get
  \[
    |\Delta_a f (z\circ \xi) - \Delta_a f(z) | \lesssim N \|a\|^{1+\iota 2s} \|\xi\|^\eta,
  \]
  where
    \begin{equation}
      \label{e:eta} \eta = \min(1,2s,\alpha_1+\alpha_2)(1-\theta).
    \end{equation}

  \textsc{Step 6.}  We finally claim that \eqref{e:conclusion},  \eqref{e:conclusion-weak} hold true in all cases. In order to prove  it, we only have to deal with the case $\|a\| \lesssim \|\xi\|$ in  which we pick $k \in \mathbb{N}$ such that  \(\|2^k a\| \simeq \|\xi\|\).  In this case, we can use  \eqref{e:conclusion}, \eqref{e:conclusion-weak} with $2^k a$ and $\xi$ (from Step~5) and  $\alpha_1 +\beta \le 1+ \iota 2s$ and get
   \begin{eqnarray}
     \label{e:int1}
     | 2^{-k}\Delta_{2^k a} f(z \circ \xi) - 2^{-k} \Delta_{2^k a} f(z)|
     \lesssim N \| a\|^{\alpha_1+\beta} \|\xi\|^{\alpha_2}  &\qquad \text{ if } \alpha_1+\alpha_2 + \beta \le 1 + \iota 2s,  \\
    \label{e:int1-weak}
     | 2^{-k}\Delta_{2^k a} f(z \circ \xi) - 2^{-k} \Delta_{2^k a} f(z)|
     \lesssim N \|a\|^{1 + \iota 2s} \|\xi \|^{\eta}, &\qquad \text{ if } \alpha_1+\alpha_2 + \beta > 1 + \iota 2s .
   \end{eqnarray}
   We now use twice what we obtained in Step~1, to $z$ and $z\circ \xi$, and get for $\alpha_1 + \alpha_2 + \beta \le 1+ \iota 2s$
   (using $\alpha_1+\beta \le 1+\iota 2s$ once again),
  \begin{align}
    \label{e:int2}    |\Delta_a f (z \circ \xi) - 2^{-k} \Delta_{2^k a} f(z \circ \xi ) | &\lesssim N \|a\|^{\alpha_1+\beta} \|\xi\|^{\alpha_2 }, \\
    \label{e:int3}     |\Delta_a f (z) - 2^{-k} \Delta_{2^k a} f(z) |  &\lesssim N \|a\|^{\alpha_1+\beta} \|\xi\|^{\alpha_2 } ,
  \end{align}
  and for $\alpha_1 + \alpha_2 + \beta > 1 + \iota 2s$,
  \begin{align}
    \label{e:int2-weak}    |\Delta_a f (z \circ \xi) - 2^{-k} \Delta_{2^k a} f(z \circ \xi ) | &\lesssim N \|a\|^{1 + \iota 2s} \|\xi\|^{\eta} ,\\
    \label{e:int3-weak}     |\Delta_a f (z) - 2^{-k} \Delta_{2^k a} f(z) |  &\lesssim N \|a\|^{1+\iota 2s} \|\xi\|^{\eta} .
  \end{align}
  
  Summing \eqref{e:int1}, \eqref{e:int2} and \eqref{e:int3} yields
  \eqref{e:conclusion} for all $y$ and $\xi$.  In the same way,
  Summing \eqref{e:int1-weak}, \eqref{e:int2-weak} and \eqref{e:int3-weak} yields
  \eqref{e:conclusion-weak} for all $y$ and $\xi$. This achieves the proof
  of the lemma.
\end{proof}

%---------------------------------------------------------------------
\begin{lemma}[H\"older continuous increments in $x$]\label{l:4-x}
  Given $y \in B_{R^{1+2s}/2}$ with $R \le 1$ and  $\alpha \in (0,\min (1,2s)]$ and some cylinder $Q=Q_R(z_0)$, let  $f \in C_\ell^{2s+\alpha} (Q)$. Then $\Delta_y f$ lies in $C_\ell^{\alpha} (\Qint)$ with  $\Qint= Q_{R/2}(z_0)$ and
 \begin{equation}\label{e:dg}
   \|\Delta_y f \|_{C_\ell^{\alpha} (\Qint)} \le C \|f\|_{C_\ell^{2s+\alpha} (Q)} \|(0,y,0)\|^{2s}
 \end{equation}
 for some constant $C$ only depending on dimension and $s$.
\end{lemma}
% ---------------------------------------------------------------------
\begin{remark}
  This lemma and the following one can be seen as discrete
  counterparts of \cite[Lemma~2.5]{schauder}.
\end{remark}
%-----------
\begin{proof}
  We remark that the assumption of the lemma implies that the
  assumptions of Lemma~\ref{l:cc-x} holds true with $\beta=0$ and
  $\alpha_1 =2s$ and $\alpha_2 = \alpha$ with
  $N = 2\|f\|_{C_\ell^\alpha} (Q)$. Applying Lemma~\ref{l:cc-x} yields
  the desired result. 
\end{proof}

Lemma~\ref{l:4-x} can also be proved directly along the lines of the proof of Lemma~\ref{l:4-v} below. The proof would be easier because $(0,y,0)$ belongs to the center of the Lie group and thus $z \circ (0,y,0) \circ \xi = z \circ \xi \circ (0,y,0)$.

% ---------------------------------------------------------------------
\begin{lemma}[H\"older continuous increments in $v$]\label{l:4-v}
  Given $w \in B_{R/2}$ with $R \le 1$, and $2s + \alpha < 1$ and $\alpha \le \min (1,2s)$ and some cylinder $Q=Q_R(z_0)$,  let  $f \in C_\ell^{2s+\alpha} (Q)$ such that $\nabla_x f \in C^0(Q)$. Then $\Delta_w f$ lies in $C_\ell^{\alpha} (\Qint)$ with $\Qint = Q_{R/2} (z_0)$ and 
 \begin{equation}\label{e:dg-v}
   [\Delta_w f ]_{C_\ell^{\alpha} (\Qint)} \le C ( [f]_{C_\ell^{2s+\alpha} (Q)} + |w|^{1-\alpha} \|\nabla_x f\|_{C^0(Q)} )\|(0,0,w)\|^{2s}
   \end{equation}
   for some constant $C$ only depending on dimension and $s$.
\end{lemma}
% ---------------------------------------------------------------------
\begin{proof}
  It is convenient to write $a=(0,0,w)$. We need to estimate the quantity
\[ W := |\Delta_a f(z \circ \xi) - \Delta_a f(z)| = |f(z \circ \xi \circ a) - f(z \circ \xi) - f(z \circ a) + f(z)|.\]
The easiest case is when $\|a\| \leq \|\xi\|$. In this case, we apply Definition~\ref{d:holder-space} at the point $z$ and $z \circ \xi$ with increment $a$. We get
\[ |f(z \circ \xi \circ a) - p_{z \circ \xi}(a)| \leq [f]_{C_\ell^{2s+\alpha}} \|a\|^{2s+\alpha}, \qquad |f(z \circ a) - p_{z}(a)| \leq [f]_{C_\ell^{2s+\alpha}} \|a\|^{2s+\alpha}.\]
The polynomials $p_z$ and $p_{z \circ \xi}$ are of kinetic degree less than $2s+\alpha < 1$. Thus, they do not have any component in the ``$v$'' variable: $p_z(0,0,w)=f(z)$ and $p_{z \circ \xi}(0,0,w) = f(z \circ \xi)$. Thus,
\begin{align*} 
 W &\leq 2 [f]_{C_\ell^{2s+\alpha}} \|a\|^{2s+\alpha} \lesssim  [f]_{C_\ell^{2s+\alpha}} \|a\|^{2s} \|\xi\|^{\alpha}.
\end{align*}
The last inequality holds when $\|a\| \leq \|\xi\|$. Note that for this case, we did not need a correction in terms of $\|\nabla_x f\|_{C^0(Q)}$. For $\|a\| > \|\xi\|$, we will need an alternative chain of inequalities.

When $\|a\| > \|\xi\|$, we apply Definition~\ref{d:holder-space} at the point $z$ and $z \circ a$ with increment $\xi$. We get
\begin{equation} \label{e:ic1}
 |f(z \circ a \circ \xi) - p_{z \circ a}(\xi)| \leq [f]_{C_\ell^{2s+\alpha}} \|\xi\|^{2s+\alpha}, \qquad |f(z \circ \xi) - p_{z}(\xi)| \leq [f]_{C_\ell^{2s+\alpha}} \|\xi\|^{2s+\alpha}.
\end{equation}
The polynomials $p_z$ and $p_{z \circ a}$ have kinetic degree less than $2s+\alpha < 1$. Thus, they have at most two nonzero terms, the constant one, and the one in the ``$t$'' variable. They are (see \cite{schauder}),
\begin{equation} \label{e:ic2}
 p_z(\xi) = f(z) + (\partial_t + v \cdot \nabla_x) f(z) \xi_t, \qquad p_{z\circ a}(\xi) = f(z \circ a) + (\partial_t + (v+w) \cdot \nabla_x) f(z \circ a) \xi_t.
\end{equation}

 Note that $z \circ a \circ \xi$ differs
  from $ z\circ \xi \circ a$. We estimate this discrepancy. If $z = (t,x,v)$ and $\xi = (\xi_t,\xi_x,\xi_v)$, we have
  \begin{align*}
   f (z \circ a \circ \xi)- f(z\circ \xi \circ a)  = &  f(t+\xi_t, x + \xi_x + \xi_t (v+w), v + \xi_v + w) - f(t+\xi_t, x+\xi_x + \xi_t v, v + \xi_v + w) \\
    = &  \int_0^1 \nabla_x f (t+\xi_t, x+\xi_x + \xi_t v + \theta \xi_t w,v +\xi_v + w) \cdot \xi_t w \dd \theta. 
  \end{align*}
  This implies that
  \begin{align*}
| f (z \circ a \circ \xi)& - f(z\circ \xi \circ a) - (w \cdot \nabla_x f (z \circ a)) \xi_t | \\
  &\le |\xi_t||w| \int_0^1 |\nabla_x f (t+\xi_t, x+\xi_x + \xi_t v + \theta \xi_t w,v +\xi_v + w) - \nabla_x f (t,x,v+w)| \dd \theta \\
  & \le 2 \|\nabla_x f\|_{C^0} \|\xi\|^{2s} |w|.
\end{align*}  

We combine this with \eqref{e:ic1} and \eqref{e:ic2} to obtain the following upper bound for $W$,
\begin{align*} 
W &\leq 2 [f]_{C_\ell^{2s+\alpha}} \|\xi\|^{2s+\alpha} +  \|\nabla_x f\|_{C^0}  |w|  \|\xi\|^{2s} \\
&\phantom{\leq} + |  (\partial_t + v \cdot \nabla_x) f(z) -  (\partial_t + (v+w) \cdot \nabla_x) f(z \circ a)| |\xi_t|, \\
\intertext{Using \cite[Lemma 2.7 for $D = (\partial_t + v \cdot \nabla_x)$]{schauder},}
&\leq 2 [f]_{C_\ell^{2s+\alpha}} \|\xi\|^{2s+\alpha} +  \|\nabla_x f\|_{C_0}|w|  \|\xi\|^{2s} + [f]_{C_\ell^{2s+\alpha}} \|a\|^\alpha \|\xi\|^{2s}, \\
&\lesssim\left( [f]_{C_\ell^{2s+\alpha}} + |w|^{1-\alpha}  \|\nabla_x f\|_{C_0} \right) \|a\|^{2s} \|\xi\|^{\alpha},
\end{align*}
For the last inequality, we used $|w| = \|a\| \geq \|\xi\|$ and $\alpha \leq 2s$.
\end{proof}
%----------

\section{The proof of Theorem~\ref{t:main2}}
\label{s:smoothing}

This section is devoted to proving Theorem~\ref{t:main2}. By an iterative process, we will establish the following family of inequalities. For all differential operator $D = \partial_t^{k_t} D_x^{k_x} D_v^{k_v}$ with $k = (k_t,k_x,k_v) \in \N^{1+2d}$, there exists $\alpha>0$ so that for all $\tau>0$ and $q>0$ there is a constant $C_{k,q}$ (depending on $k_t,k_x,k_v$, $q$, $\tau$, and the parameters in Theorem~\ref{t:main2}) such that
\begin{equation}
    \label{e:Cpol-2s+alpha}
    \|Df\|_{C^{2s+\alpha}_{\ell,q}([\tau,\infty) \times \R^d \times \R^d)} \le C_{k,q}.
  \end{equation}
The value of $\alpha$ that we obtain in the iteration will also depend on $k$ and it will tend to be smaller as the order of differentiation increases. A posteriori, we obtain a $C^\infty$ estimate for $f$, so the particular values of $\alpha$ after each iteration do not matter.
In order to be in position to apply the Schauder estimate and gain $2s$ derivatives, we will always pick $\alpha \in (0,\min (1,2s))$ such that $2s + \alpha \notin \{1,2\}$. To do so, it is convenient to work with exponents $\alpha$ such that $\alpha < 1-2s$ if $s < 1/2$ and $\alpha < 2-2s$ if $s \ge 1/2$.

  We use the (classical) definition $D^{(k_t,k_x,k_v)} = \partial_t^{k_t} \partial_{x_1}^{k_x^1}\dots \partial_{x_d}^{k_x^d} \partial_{v_1}^{k_v^1}\dots \partial_{v_d}^{k_v^d}$ if  $k_x = (k_x^1,\dots, k_x^d)$ and $k_v = (k_v^1,\dots, k_v^d)$. We recall that the order of a multi-index $k \in \mathbb{N}^{1+2d}$ is $k_t + k_x^1 + \dots k_x^d + k_v^1 + \dots + k_v^d$ and is denoted by $|k|$.  In this section, when we refer to the order of $D$, we mean literally the classical order of differentiation (not the kinetic order as defined in \cite{schauder}).

Note that the value of $C_{k,q}$ depend on several parameters. We stress their dependence with respect to $k$ and $q$ because it affects the order in which these numbers are computed. As we said, we establish Inequalities \eqref{e:Cpol-2s+alpha} for every value of $k$ and $q$ iteratively. We first prove if for $k=(0,0,0)$ and any value of $q$. Then, we will compute $C_{k,q}$ in terms of the values of $C_{i,q+\kappa+3}$ for multi-indices $i \in \N^{1+2d}$ so that either $|i| < |k|$, or $|i|=|k|$ and $i_x > k_x$. In other words, the upper bounds for the differential operator $D^k f$ will depend on the bounds for lower order operators, and on the bounds for operators with the same total order but higher order in $x$. We observe that the computation of any of these values $C_{k,q}$ would involve finitely many iterations, starting from the family of inequalities \eqref{e:Cpol-2s+alpha} for $k=0$. Note the addition ``$+\gamma$'' in the decay exponent $q+\kappa+3$, which is not problematic since we start with the inequality $C_{0,q}$ for every value of $q$. This loss $\kappa+3$ only depends on the parameter $s$ and $\gamma$ from the collision kernel $B$, see  \eqref{assum:B}.

There are several sequential orders which we could employ in order to compute all the constants $C_{k,q}$. In this proof, we make the following (somewhat arbitrary) choice. We first establish \eqref{e:Cpol-2s+alpha} for $k=(0,k_x,0)$, with $k_t=|k_v|=0$. In the second step, we extend the inequalities \eqref{e:Cpol-2s+alpha} to indices of the form $k=(k_t,k_x,0)$, with $k_v = 0$. In the third and last step, we establish \eqref{e:Cpol-2s+alpha} for all values of $k \in \N^{1+2d}$. By proving Estimates \eqref{e:Cpol-2s+alpha} in this order, we ensure that we always have enough previous information to establish the value of $C_{k,q}$ in each step.

We start with a function $f \in \Cpol^0$ (according to Theorem \ref{t:decay}). The iteration procedure described below allows us to obtain upper bounds of the form \eqref{e:Cpol-2s+alpha} for increasingly higher values of $|k|$. If we only had an upper bound for $\|f\|_{C_{\ell,q}^0}$ for some finite exponent $q$, the iteration would provide regularity estimates only up to certain order of differentiation.

The zeroth step of the iteration is to apply Corollary~\ref{c:schauder_boltzmann}, which provides Inequality \eqref{e:Cpol-2s+alpha} for $k_t=0$, $k_x=0$, $k_v = 0$. This is the case where $Df=f$.  The remainder of the proof proceeds in three steps, as described above. 

\medskip

 \textsc{Step~1.} We prove \eqref{e:Cpol-2s+alpha} holds true for all
 differential operators of the form $D =D_x^{k_x}$. We proceed by
 induction on $n=|k_x|$. It is convenient to make the inductive statement in terms of increments. More precisely, we are going to prove
 by induction on $n \ge 1$ that there exists an $\alpha_n$ such that for any $\tau>0$, there exists a $C_{n,q}>0$ so that
 \begin{equation} \label{e:induc-n}
\forall k_x \in \mathbb{N}^d,  q>0,  y \in B_1, \qquad |k_x |\le n-1 \Rightarrow \|\Delta_y D_x^{k_x} f\|_{C_{\ell,q}^{2s+\alpha_n} ([\tau,\infty) \times \R^d \times \R^d)} \le C_{n,q} |y|
\end{equation}
where we recall that $\Delta_y f (z) = f (z \circ (0,y,0)) - f(z)$.

Note that passing to the limit as $y \to 0$, the inequality above implies that for all $|k_x| \leq n$,
\begin{equation} \label{e:induc-n2}
 \|D_x^{k_x} f\|_{C_{\ell,q}^{2s+\alpha_n} ([\tau,\infty) \times \R^d \times \R^d)} \le C_{n,q}.
\end{equation}
  
Corollary \ref{c:schauder_boltzmann} provides the case $n=0$ in \eqref{e:induc-n2}. Note that \eqref{e:induc-n} holds trivially for $n=0$ since there is no $k_x$ so that $|k_x| \leq -1$. In order to proceed by induction, we assume we know \eqref{e:induc-n} and \eqref{e:induc-n2} hold up to certain value of $n \in \mathbb N$ and we prove it for $n+1$.

Let $|k_x| = n$ and $g = \Delta_y D^{k_x}_x f$. By the inductive hypothesis \eqref{e:induc-n2} combined with Lemma~\ref{l:4-x}, we have that for any value of $\tau>0$ and $q>0$,
\begin{equation} \label{e:ss1}
 \|g\|_{C^{\alpha_n}_{\ell,q}([\tau,\infty) \times \R^d \times \R^d)} \lesssim \|(0,y,0)\|^{2s} = |y|^{2s/(1+2s)}.
\end{equation}
We want to enhance the exponent $2s/(1+2s)$ on the right hand side all the way to one. For that, we apply the following lemma successively.

\begin{lemma}[Gain of regularity in $x$]\label{l:step1}
Let $g = \Delta_y D^{k_x}_x f$ (as above), $\beta \in (0,1+2s)$ and assume that    \eqref{e:induc-n} holds true. If there exists $\bar \alpha \in (0,\alpha_n]$ such that $2s + \bar \alpha' \notin \{1,2\}$ and, 
\[
      \|g\|_{C_{\ell,q+\kappa+3}^{\bar \alpha}( [\tau,\infty) \times \R^d \times \R^d )} \lesssim \|(0,y,0)\|^\beta,
\]
then
\[ \|g\|_{C_{\ell,q}^{2s + \bar \alpha'}( [2\tau,\infty) \times \R^d \times \R^d )} \lesssim \|(0,y,0)\|^\beta, \]
with $\bar \alpha' = \frac{2s}{1+2s}\bar \alpha$. Here, $\kappa$ is the constant in Proposition \ref{p:global_schauder}.
\end{lemma}

\begin{proof}
The key to this lemma is to differentiate \eqref{e:boltzmann} and compute an equation for $g$. Then, we apply the global Schauder estimate of Proposition~\ref{p:global_schauder} together with the estimates we have for each incremental quotient.

Indeed, by a direct computation, we verify that $g$ verifies the equation
\[ (\partial_t + v \cdot \nabla_x) g - \Q_1(f,g) = h,\]
where
\[ h = \sum_{\substack{|i|<n \\ i \leq k_x}} \left\{ \Q_1(\Delta_y \hat D_i f, \tau_y D_i f) + \Q_1(\hat D_i f, \Delta_y D_i f) \right\} + \sum_{i \leq k_x} \left\{ {\Q_2(\Delta_y \hat D_i f, \tau_y D_i f)} + \Q_2(\hat D_i f, \Delta_y D_i f) \right\}.\]
Here, $i \in \mathbb N^d$ is a multi-index. When we write $i \leq k_x$, we mean that each component of $i$ is less or equal than each component of $k_x$. We write $\tau_y f(z) = f(z \circ (0,y,0)) = \Delta_y f + f$. We also write $\hat D_i$ to denote the differential operator so that $D^{k_x}_x = \hat D_i \circ D_i$.

Since the index $i$ in the first sum runs over $|i| < n$, the inductive hypothesis \eqref{e:induc-n} tells us that both $\tau_y D_i f$ and $\Delta_y D_i f$ are bounded in $C^{2s+\alpha_n}_{\ell,q+\kappa+3}$  by $\lesssim |y|$. Likewise, for every value of $i$ so that $i \leq k_x$, we have $D_i f$, $\hat D_i f$, $\Delta_y D_i f$, $\Delta_y \hat D_i f$, all bounded in $C^{\alpha_n}_{\ell,q+\kappa+3}$ by $\lesssim |y|$ except for the two extreme cases: $\Delta_y D_i f$ for $i=k_x$ and $\Delta_y \hat D_i f$ for $i=(0,0,0)$. Both functions coincide with $\Delta_y D^{k_x}_x f$. For this reason, the hypothesis of the lemma bounds these two functions in $C^{\bar \alpha}_{\ell,q+\kappa+3}$ by $\lesssim \| (0,y,0)\|^\beta$.

Taking the previous paragraph into account, we bound each term in $h$ using Lemmas~\ref{l:Q1_Calpha} and \ref{l:Q2_bound}. We obtain a bound for $\|h\|_{C^{\bar  \alpha'}_{\ell,q+\kappa+3-(\gamma+2s)-\alpha/(1+2s)}}$ and consequently a bound for $\|h\|_{C^{\bar  \alpha'}_{\ell,q+\kappa}}$ since $\gamma +2s + \alpha / (1+2s) \le 3$.

Applying Proposition~\ref{p:global_schauder}, we obtain the desired  bound for $\|g\|_{C^{2s+\bar \alpha'}_{\ell,q}}$.
\end{proof}

Note that Lemma~\ref{l:step1} provides a gain in regularity at the expense of a loss in decay, from $q+\kappa +3$ to $q$. We did not try to make $\kappa$ explicit in Proposition~\ref{p:global_schauder} and $+3$ is a rather rough overestimation of the additional loss in the decay exponent when applying Lemmas~\ref{l:Q1_Calpha} and \ref{l:Q2_bound}.  
\bigskip

Applying Lemma~\ref{l:step1} once, we transform \eqref{e:ss1} into the following inequality, for every value of $q>0$,
\begin{equation} \label{e:ss2}
 \|g\|_{C^{2s+\alpha'_n}_{\ell,q}([2\tau,\infty) \times \R^d \times \R^d)} \lesssim \|(0,y,0)\|^{2s} = |y|^{2s/(1+2s)}.
\end{equation}
Note that the time shift $\tau$ was updated to $2\tau$. This is because the application of Proposition~\ref{p:global_schauder} in the proof of Lemma~\ref{l:step1} requires a gap in time. We obtain estimate for every value of $\tau>0$ (with constants depending on $\tau$). So, the difference between $\tau$ and $2\tau$ is not relevant for the final estimates. In view of this observation, we will omit the domain dependence in the estimates below as a way to unclutter the expressions and focus on the H\"older and decay exponents.

The estimate \eqref{e:ss2} can be combined with Lemma~\ref{l:cc-x} for $(\alpha_1,\alpha_2,\beta)=(2s,\alpha'_n,2s)$. We get, 
\begin{equation} \label{e:ss3}
 \|g\|_{C^{\alpha'_n}_{\ell,q}} \lesssim \|(0,y,0)\|^{4s}.
\end{equation}

This is an improvement on the exponent in the right hand side of \eqref{e:ss1} from $2s$ to $4s$ (at the expense of reducing $\alpha_n$ to $\alpha'_n$). We continue applying Lemma~\ref{l:step1} together with Lemma~\ref{l:cc-x} successively improving the exponent on the right hand side to $6s$, $8s$, $10s$, \dots for as long as this exponent is strictly less than $1+2s$. After $j$ steps, we are left with the inequality
\[
 \|g\|_{C^{\tilde \alpha_j}_{\ell,q}} \lesssim \|(0,y,0)\|^{2s (j+1)} \qquad \text{where } \tilde \alpha_j := \left( \frac{2s}{1+2s} \right)^j \alpha_n.
\]
This iteration continues identically until $\tilde \alpha_j + (j+1)(2s)> 1+2s$. At that point, Lemma~\ref{l:cc-x} takes a different form and the next step gives us,
\[
 \|g\|_{C^{\tilde \alpha_{j+1}}_{\ell,q}} \lesssim \|(0,y,0)\|^{1+2s}.
\]
If the value of $\tilde \alpha_j + (j+1) (2s)$ is only barely above $1+2s$, the value of $\tilde  \alpha_{j+1}$ that we would get applying Lemma~\ref{l:cc-x} might be tiny. In order to avoid that inconvenience, if $\tilde \alpha_j + j(2s) \in (1,1+s]$, then we can perform an intermediate step gaining $s$ derivatives instead of $2s$ derivatives, \textit{i.e.} taking $\alpha_1 =s$ when applying Lemma~\ref{l:cc-x}. That way, we ensure that the value of $\tilde  \alpha_{j+1}$  is bounded below only in terms of the parameters of Theorem~\ref{t:main2}.
One more application of Lemma~\ref{l:step1}  gives us
\begin{equation} \label{e:ss6}
 \|g\|_{C^{2s+\tilde \alpha_{j+1}}_{\ell,q}} \lesssim \|(0,y,0)\|^{1+2s} = |y|.
\end{equation}

Recalling that $g = \Delta_y D^{k_x}_x f$, we finished the proof of \eqref{e:induc-n} with $\alpha_{n+1} := \tilde \alpha_{j+1}\le \alpha_n < 1-2s$. This finishes \textsc{Step} 1 in the proof of Theorem~\ref{t:main2}. That is, we obtained \eqref{e:Cpol-2s+alpha} when $D$ involves derivatives with respect to $x$ only.

\medskip
  \textsc{Step~2.} 
We next prove that for all $k=(k_t,k_x,0)$, and all $q>0$ and $\tau>0$, we can control
  $\|\partial_t^{k_t} D_x^{k_x} f \|_{C_{\ell,q}^{2s+\alpha_k}}$ for some small
  $\alpha_k > 0$. That means that for any $\tau>0$, there is a $C_{k,q}$ so that
  \begin{equation} \label{e:s12}
    \|\partial_t^{k_t} D_x^{k_x} f \|_{C_{\ell,q}^{2s+\alpha_k} ( [\tau,\infty) \times \R^d \times \R^d) } \le C_{k,q}.
\end{equation}
We are going to derive \eqref{e:s12} for all $k_x$ by  induction on $n=k_t$. Remark that for $n=0$, we proved these estimates for all $k_x$  in \textsc{Step}~1. We argue by induction as follows: we  prove that, given any $n \in \mathbb N$, $n \ge 1$, and $m \in \mathbb N$, if \eqref{e:s12} holds whenever $k_t \leq n-1$ and $|k_x|\leq m+1$, and also for $k_t = n$ and $|k_x| < m$ then it also holds true for $k_t = n$ and $|k_x| = m$.

Let $n \ge 1$ and $k_x \in \N^d$ be any multi-index with $|k_x| = m$. Using the inductive hypothesis \eqref{e:s12} with $k_t = n-1$, we apply
\cite[Lemma~2.6]{schauder} and, for any value of $q>0$, get a bound on
\[ \|(\partial_t + v \cdot \nabla_x)\partial_t^{n-1} D^{k_x}_x f
  \|_{C_{\ell,q}^{\alpha} } \lesssim C_{k_0, q} \]
with $k_0= (n-1,k_x,0)$. 

Using the induction assumption for $k_t = n-1$, $|\tilde k_x| = m+1$, we also control the
norm of $(v\cdot \nabla_x) \partial_t^{n-1} D^{k_x}_x f$,
\[ \|(v\cdot \nabla_x) \partial_t^{n-1} D^{k_x}_x f \|_{C_{\ell,q}^{2s+\alpha} } \leq \|\partial_t^{n-1} \nabla_x D^{k_x}_x f \|_{C_{\ell,q+1}^{2s+\alpha}} \lesssim \max_{\tilde k} C_{\tilde k, q+1}\]
with $\tilde k = (n-1,\tilde k_x,0)$ and $|\tilde k_x| = m+1$.

Therefore, we combine the last two estimates to obtain the inequality, for some $\alpha>0$ and some constant $C$ depending on $n$ and $m$,
\begin{equation} \label{e:s13}
 \| \partial_t^{n} D^{k_x}_x f \|_{C_{\ell,q}^{\alpha}} \leq C.
\end{equation}

Our next objective is to turn the estimate \eqref{e:s13} into
\begin{equation} \label{e:s14}
 \| \partial_t^{n} D^{k_x}_x f \|_{C_{\ell,q}^{2s+\alpha'} } \leq C.
\end{equation}

Let $g := \partial_t^{n} D^{k_x}_x f$. We compute an equation for $g$ and get
\[ (\partial_t + v\cdot \nabla_x) g - \Q_1(f,g) = h,\]
where
\[ h = \sum_{\substack{i \leq (n,k_x,0) \\ i \neq (n,k_x,0)} } \Q_1(\hat D_i f, D_i f)  + \sum_{i \leq (n,k_x,0)}  \Q_2(\hat D_i f, D_i f).\]
Here, $i \in \mathbb N^{1+2d}$ is a multi-index, and like in \textsc{Step} 1, $\partial_t^{n} D^{k_x}_x = \hat D_i \circ D_i$.

An  inspection of the functions involved in $h$ shows that, by applying the inductive hypothesis together with Lemmas~\ref{l:Q1_Calpha} and \ref{l:Q2_bound}, we bound $\|h\|_{C^{\alpha'}_{\ell,q}}$ for all $q>0$. Finally, \eqref{e:s14} follows after applying Proposition~\ref{p:global_schauder}.

\medskip

\textsc{Step~3.} In the third and last step, we establish the inequality \eqref{e:Cpol-2s+alpha} for every differential operator
$D=\partial_t^{k_t} D_x^{k_x}D_v^{k_v}$ with $k \in \N^{1+2d}$, and for all $q>0$ and $\tau>0$. We will prove that
 \begin{equation} \label{e:induc-n-3b}
\begin{aligned}
  \exists \alpha_{n,m} >0 \; / \;   \forall k \in \N^{1+2d}, q>0,w \in B_1, \\
  \{ |k_v| \le n, k_t + |k_x| \leq m \} \Rightarrow \quad \|D f\|_{C_{\ell,q}^{2s+\alpha_{n,m}} } \le C_{n,m}.
\end{aligned}
\end{equation}
We proceed the proof of \textsc{Step} 3 by a bidimensional induction similar as in \textsc{Step} 2. 
\medskip

If $s \geq 1/2$, we can proceed like in \textsc{Step} 2. Indeed, \eqref{e:induc-n-3b}  implies that (See Proposition~\ref{p:interpol})
\[  \|\partial_{v_i} Df\|_{C^{\alpha_{n-1,m}}_{\ell,q}} \leq C_{n-1,m}. \]
Thus, we compute an equation for $g = \nabla_v Df$ and argue like in the previous step. 
\medskip

When $s < 1/2$, like in \textsc{Step} 1, it is convenient to set up the induction keeping track of the  H\"older regularity of differential operators, and also of increments. Thus, we prove that for all $n \ge 1$, $m \in \N$,
 \begin{equation} \label{e:induc-n-3}
\begin{aligned}
 \exists \alpha_{n,m} \in (0, 1-2s), \; / \;   \forall k \in \N^{1+2d}, q>0,w \in B_1,   \\
\{  |k_v| \le n -1, \; k_t + |k_x| \leq m\} \Rightarrow \quad \|\Delta_w D f\|_{C_{\ell,q}^{2s+\alpha_{n,m}} } \le C_{n,m} |w|.
\end{aligned}
\end{equation}
The constants $\alpha_{n,m}$ depend on $n$, $m$ and the parameters in Assumption~\ref{a:hydro-assumption}. The constants $C_{n,m}$ depend in addition on $\tau$ and $q$. By taking $w \to 0$, \eqref{e:induc-n-3} implies \eqref{e:induc-n-3b}.
\medskip

The case $n=0$ of \eqref{e:induc-n-3b} was established in \textsc{Step} 2. The inequality \eqref{e:induc-n-3} holds trivially for $n=0$.
\medskip

Now, let $n\geq 1$ and $k$ be any multi-index with $|k_v|=n-1$ and $k_t + |k_x| = m$. From the inductive hypothesis, $Df$ satisfies \eqref{e:induc-n-3b}. Remark that we can assume without loss of generality that $\alpha_{n,m} < 1-2s$ in \eqref{e:induc-n-3b} and \eqref{e:induc-n-3}.  Thus, for any $q>0$,
\begin{equation} \label{e:s31}
 \|Df\|_{C^{2s+\alpha_{n-1,m}}_{\ell,q}} \leq C_{n-1,m}.
\end{equation}

Let $w \in B_1$. Since $2s+\alpha_{n-1,m} < 1$, we apply Lemma~\ref{l:4-v} together with \eqref{e:s31} and obtain, for  $\alpha=\alpha_{n-1,m}>0$,
\begin{align}
\nonumber
  \| \Delta_w Df\|_{C^{\alpha}_{\ell,q}} & \lesssim \left( \|Df\|_{C^{2s+\alpha_{n-1,m}}_{\ell,q}} + \|\nabla_x Df\|_{C^{\alpha_{n-1,m+1}}_{\ell,q}} \right) |w|^{2s} \\
\nonumber    & \lesssim \left( C_{n-1,m} + C_{n-1,m+1} \right) |w|^{2s} \\
  & \lesssim |w|^{2s}.  \label{e:s32}
  \end{align}

In order to obtain \eqref{e:induc-n-3} for $n$ and $m$, we need to enhance the exponent on the right hand side of the inequality above, from $2s$ all the way to one. We do it through an iterative process similar to \textsc{Step} 1.

  \begin{lemma}[Gain of regularity in $v$]\label{l:step3}
    Let $g = \Delta_w Df$, $\beta \in (0,1)$ and assume that    \eqref{e:induc-n-3} holds true for smaller values of $n+m$, or for the same value of $n+m$ with $n$ smaller. If there exists    $\bar \alpha \in (0,\alpha_n]$ such that $2s +\bar \alpha' \notin \{1,2\}$ and,
    \begin{equation}
      \label{e:gain-v-assum}
      \|g\|_{C_{\ell,q+\gamma+3}^{\bar \alpha}( [\tau,\infty) \times \R^d \times \R^d )} \lesssim |w|^\beta,
    \end{equation}
    then
    \[ \|g\|_{C_{\ell,q}^{2s + \bar \alpha'}( [2\tau,\infty) \times \R^d \times \R^d )} \lesssim |w|^\beta, \]
    with $\bar \alpha' = \frac{2s}{1+2s}\bar \alpha$.
  \end{lemma}

\begin{proof}
  The proof is very similar to the proof of Lemma~\ref{l:step1}.  The only difference is that the equation for $g$ will now have terms involving $\nabla_x Df$.

 The function $\bar g = Df$ satisfies
  \[ 
  (\partial_t + v \cdot \nabla_x) \bar g -\Q_1(f, \bar g) = \bar{H}
  \] 
  with
  \[
    \bar{H}  =  \sum_{\substack{i \leq k \\ i \neq k} } \Q_1(\hat D_i f, D_i f)  + \sum_{i \leq k}  \Q_2(\hat D_i f, D_i f) +  \{(\partial_t + v \cdot \nabla_x),D\} f
  \]  
  where $\{(\partial_t + v \cdot \nabla_x),D\} f=(\partial_t + v \cdot \nabla_x)Df-D(\partial_t + v \cdot \nabla_x)f$ (Poisson bracket) and $D_i$ and $\hat{D}_i$  are such that $D_i \circ \hat{D}_i = D$.

Since $\partial_t$ and $D$ commute, $\{(\partial_t + v \cdot \nabla_x),D\} = \{ v \cdot \nabla_x,D\}$. Given $k = (k_t,k_x,k_v)$, by a direct computation one verifies that
\[ \{ v \cdot \nabla_x,D\} = \sum_{\tilde k} D^{\tilde k}\]
where the multi-index $\tilde k$ runs over all multi-indexes with the same order as $k$ so that $D^{\tilde k} = \partial_{x_i} \tilde D$ and $D^k = \partial_{v_i} \tilde D$ for some differential operator $\tilde D$ and $i=1,\dots,d$. According to our induction hypothesis, \eqref{e:induc-n-3} holds for all these indexes $\tilde k$, therefore
\begin{equation} \label{e:s32_commutator_bound}
 \| \Delta_w \{(\partial_t + v \cdot \nabla_x),D\} f\|_{C^{2s+\alpha}_{\ell,q}} \leq  C_{n-1,m+1} |w|.
\end{equation}

The function $g = \Delta_w D f$ satisfies the  equation 
\[
 (\partial_t + v \cdot \nabla_x) g - \Q_1(f,g) = H \quad \text{ in } (0,T) \times \R^d \times \R^d
\]
  where 
  \begin{align*} 
 H &= \Delta_w \bar{H} - \sum_{j=1}^d w_j \tau_w (\partial_{x_j} Df), \\
&= \sum_{\substack{|i|<n \\ i \leq k_x}} \left\{ \Q_1(\Delta_y \hat D_i f, \tau_y D_i f) + \Q_1(\hat D_i f, \Delta_y D_i f) \right\} + \sum_{i \leq k_x} \left\{  \Q_2(\Delta_y \hat D_i f, \tau_y D_i f) +\Q_2(\hat D_i f, \Delta_y D_i f) \right\} \\
&\phantom{=} +  \Delta_w \{(\partial_t + v \cdot \nabla_x),D\}f - \sum_{j=1}^d w_j \tau_w (\partial_{x_j} Df).
\end{align*}
  The last term is the commutator between $\Delta_w$ and the transport part $(\partial_t + v \cdot \nabla_x)$, and it is bounded in $C^{2s+\alpha}_{\ell,q}$, for all $q>0$, by the inductive hypothesis. The first two terms are bounded identically as in the proof of Lemma~\ref{l:step1}. And the third term was bounded in \eqref{e:s32_commutator_bound}. The proof finish by applying Proposition~\ref{p:global_schauder} to $g$, in the same way as in the proof of Lemma~\ref{l:step1}.

\end{proof}

Once Lemma~\ref{l:step3} is established, the rest of the proof of \textsc{Step} 3 proceeds similarly  as in \textsc{Step} 1 using Lemma~\ref{l:cc-v} instead of Lemma~\ref{l:cc-x}.

This finishes the proof of Theorem~\ref{t:main2}.

\appendix
\section{Gressman-Strain coercivity estimate}
\label{s:gressman}

In this appendix, we show how the change of variables described in Section~\ref{s:cov}, together with a local coercivity estimate like the one in Theorem~\ref{t:coercivity}, can be used to recover the global coercivity estimate with respect to the \emph{lifted} anisotropic distance of Gressman and Strain \cite{gressmanstrainBETTERpaper} (see also the prequel paper \cite{gressman2011global}).

The transformation $T_0$ defined in \eqref{e:T0v} depends on a given point $v_0 \in \R^d$. For any such $v_0$, let us consider the pushed forward distance: for
$v_1,v_2 \in E_1 (v_0)= v_0 + T_0 (B_1)$, 
\begin{equation}\label{e:da}
\da (v_1,v_2) = | T_0^{-1} (v_1-v_2)|
\end{equation}
This distance $\da$ depends on the choice of $v_0$. However, as we will see, for any pair $v_1, v_2 \in \R^d$, all the possible values of $\da(v_1,v_2)$ are comparable for all possible choices of $v_0$ so that $v_1,v_2 \in E_1(v_0)$.

We also recall the anisotropic distance defined in \cite{gressman2011global}:
for all $v_1,v_2 \in \R^d$,
\begin{equation}\label{e:dgs}
  \dgs (v_1,v_2) =\sqrt{\frac14 \left(|v_1|^2 - |v_2|^2\right)^2 + |v_1-v_2|^2}.
\end{equation}
%--------------------------------------------------------------------------------------
\begin{lemma}[The anisotropic distance $\da$]\label{l:da}
Given  $v_0 \in \R^d$ with $|v_0| \ge 2$, we have for all $v_1,v_2 \in v_0 + T_0 (B_1)$,
\[ \da (v_1,v_2) \simeq \dgs (v_1,v_2).\]
The hidden constants in $\simeq$ do not depend on any parameter, not even  dimension. 
\end{lemma}
%--------------------------------------------------------------------------------------
\begin{proof}
Since $T_0$ is linear, we have to estimate $|T_0^{-1}(v_1-v_2)|$. Let $v_{1,2} =v_1-v_2$. 
We have 
\[ v_{1,2} = \lambda \frac{v_0}{|v_0|} + w \quad \text{ with } w \cdot
v_0 =0.\]
The real number $\lambda$ satisfies
$\lambda |v_0| = v_{1,2} \cdot v_0$ and
$|v_{1,2}|^2 = \lambda^2 + |w|^2$. Hence we have
\begin{align*}
 \da (v_1,v_2) = &| T_0^{-1} (v_{1,2})|  = \sqrt{\lambda^2 |v_0|^2 + |w|^2} \\
= &\sqrt{\lambda^2 (|v_0|^2 -1) + |v_1-v_2|^2}\\
\simeq & \sqrt{\lambda^2 |v_0|^2 + |v_1-v_2|^2}\\
= &  \sqrt{((v_1-v_2) \cdot v_0)^2 + |v_1-v_2|^2}.
 \end{align*}
We finally use that $T_0 (B_1)$ is a convex subset of $B_1$ in order to get
\[ \left|\frac{v_1+v_2}2 - v_0 \right| \le 1 \]
which allows us to conclude.  
\end{proof}

In \cite{gressman2011global,gressmanstrainBETTERpaper}, Gressman and Strain obtained
sharp coercivity estimates for the linear Boltzmann collision operator
under some conditions on $f$ on mass, concentration and moments. In
the next proposition, we prove an inequality of the same nature.
%------------------------------------------------------------------------------
\begin{prop}[Coercivity estimate]\label{p:GS} Let $f$ be non-negative and such that
Assumption~\ref{a:hydro-assumption} holds. If $\gamma < 0$, we also assume \eqref{e:extra-int}. Let  $g:\R^d \to \R$ be an arbitrary function. Then
\begin{multline}\label{e:GS}
 -\int_{\R^d} \Q (f,g) g \dv \\
\ge c \iint_{\dgs(v,v')<\rho}  (g(v)-g(v'))^2 \frac{(1+|v+ v'|)^{\gamma + 2s +1}}{\dgs(v,v')^{d+2s}} \dv \dv'
- C \int_{\R^d} g(v)^2 (1+| v|)^{\max(\gamma,0)} \dd  v
\end{multline}
where the constants $c$, $\rho$ and $C$ only depend on dimension $d$ and $m_0,M_0,E_0,H_0$ from Assumption~\ref{a:hydro-assumption} and $C_\gamma$ in \eqref{e:extra-int} (only if $\gamma<0$). We recall that $\Q$ denotes the Boltzmann collision operator defined in
\eqref{e:Q} and $\dgs$ denotes the non-isotropic distance defined in \eqref{e:dgs}.
\end{prop}
%----------------------------------------------------------------------------
We recall that the collision operator can be split in a principal part
and a lower order term, see \eqref{e:Q-bis}.  We prepare the proof of
the proposition by first estimating from below the principal
contribution of the bilinear form
$\langle \Q (f,g), g \rangle_{L^2}$.
%-------------------------------------------------------------------------------
\begin{lemma} \label{l:GS} Let $f$ be non-negative and such that
  Assumption~\ref{a:hydro-assumption} holds, and if $\gamma<0$ also \eqref{e:extra-int} holds true. Let $g:\R^d \to \R$ be an
  arbitrary function. Then
\begin{equation}\label{e:GS-symmetric-part}
 \iint_{\dgs(v,v') < R} (g(v)-g(v'))^2 K_f(v,v')  \dd v' \dd v 
\ge c \iint_{\dgs(v,v')<\rho}  (g(v)-g(v'))^2 \frac{(1+|v+ v'|)^{\gamma + 2s +1}}{\dgs(v,v')^{d+2s}} \dv \dv'.
\end{equation}
Here, the constants $c>0$ and $\rho \in (0,1), R \in (2,+\infty)$ only depend on dimension $d$ and
$m_0,M_0,E_0,H_0$ from Assumption \eqref{a:hydro-assumption} and $C_\gamma$ in
\eqref{e:extra-int} (only if $\gamma<0$).  We recall that $K_f$ is the
kernel defined in \eqref{e:kf} and $\dgs$ denotes the non-isotropic
distance defined in \eqref{e:dgs}.
\end{lemma}
%----------------
\begin{proof}
  We are going to use the change of variables from  Section~\ref{s:cov}. We recall that a kernel $\bar K_f$ is defined  in \eqref{e:barKf} and that this kernel satisfies appropriate   ellipticity conditions. 

From Corollary~\ref{c:ndc}, we know that the kernel $\bar K_f$ satisfies \eqref{e:local_coercivity}, with a constant $\lambda$ independent of $v_0$.

Let $R_0 \ge 2$ and $v_0$ such that $|v_0|=R_0$.   We change variables in \eqref{e:nondeg2}. Recall that  $\bar v = v_0 + T_0  v$ and $\bar v' = v_0 + T_0  v'$. We also write  $\bar g(v) = g(\bar v)$. Note that $\dd  v = |v_0| \dd \bar v$. Thus, \  \eqref{e:local_coercivity} for $\bar K_f$ translates into the following inequality for $K_f$,
\begin{multline*}
\iint_{ E_1 (v_0) \times  E_1(v_0)} (g(\bar v')- g(\bar v))^2 |v_0|^{1-\gamma-2s} K_f(\bar v, \bar v') \dd \bar v' \dd \bar v \\
\gtrsim  |v_0|^2 \iint_{ E_{1/2} (v_0)\times  E_{1/2}(v_0)} (g(\bar v')-  g(\bar v))^2 \dgs (\bar v,\bar v')^{-d-2s} \dd \bar v' \dd \bar v
\end{multline*}
where we recall that $E_r (v_0) = v_0+T_0 (B_r)$ for $r>0$. We used the definition of $d_a$ and Lemma~\ref{l:da}. Rearranging the powers of $|v_0|$, we get for any $v_0 \in \R^d \setminus B_2$,
\begin{multline*} 
  \iint_{E_1 (v_0) \times E_1 (v_0) } (g(\bar v')- g(\bar v))^2 K_f(\bar v, \bar v') \dd \bar v' \dd \bar v \\
  \gtrsim (1+|v_0|)^{1+\gamma+2s} \iint_{ E_{1/2} (v_0) \times  E_{1/2} (v_0)} (g(\bar v')-  g(\bar v))^2 \dgs (\bar v,\bar v')^{-d-2s} \dd \bar v' \dd \bar v.
\end{multline*}
We remark that for $\bar v,\bar v' \in E_{1/2} (v_0)$, we have $1+|v_0| \simeq 1+|\bar v + \bar v'|$. Hence, we get 
\begin{multline*}
  \iint_{E_1 (v_0) \times E_1 (v_0) } (g(\bar v')- g(\bar v))^2 K_f(\bar v, \bar v') \dd \bar v' \dd \bar v \\
  \gtrsim \iint_{E_{1/2}(v_0) \times E_{1/2} (v_0)} (g(\bar v')-  g(\bar v))^2 \frac{(1+|\bar v + \bar v'|)^{1+\gamma+2s}}{\dgs (\bar v,\bar v')^{d+2s}} 
\dd \bar v' \dd \bar v.
\end{multline*}
We now multiply the previous inequality by $|v_0|$, integrate with
respect to $v_0 \in \R^d \setminus B_2$. We get
\begin{equation}
  \label{e:recall}
  \iint (g(\bar v')- g(\bar v))^2 K_f(\bar v, \bar v')  W_1 (\bar v,\bar v') \dd \bar v' \dd \bar v 
  \gtrsim \iint (g(\bar v')-  g(\bar v))^2 \frac{(1+|\bar v + \bar v'|)^{1+\gamma+2s}}{\dgs (\bar v,\bar v')^{d+2s}} 
W_{1/2} (\bar v,\bar v') \dd \bar v' \dd \bar v
\end{equation}
with
\[
   W_1 (v,v') := \int_{\R^d \setminus B_2}  |v_0| \un_{v,v' \in E_1(v_0)} \dv_0 \quad \text{ and } \quad 
   W_{1/2} (v,v') := \int_{\R^d \setminus B_2} |v_0| \un_{v,v' \in E_{1/2}(v_0)}  \dv_0
\]
where $\un_A$ denotes the indicator function of a set $A$:
$\un_A (v)=1$ if $v \in A$ and $\un_A (v)=0$ if $v \notin A$.

We now observe that for some constants $R>0$ (large) and $\rho>0$ (small),
\begin{align}
  \label{e:upper} W_1(v,v') &\lesssim \un_{\{\dgs (v,v') < R \}} \\
  \label{e:lower} W_{1/2} (v,v') &\gtrsim \un_{\{\dgs (v,v') < \rho \}} \un_{\{v \notin B_{2} \text{ or } v' \notin B_{2}\}} .
\end{align}

As far as \eqref{e:upper} is concerned, if there exists $v_0 \in \R^d$
such that $v,v' \in E_1(v_0)$, then $\da (v,v') <2$, see
\eqref{e:da}. Thus, from Lemma~\ref{l:da}, $\dgs(v,v') < R$ for some universal constant $R$. Moreover, since we have $\da(v,v_0) < 1$, we also have $\dgs(v,v_0) < R$. In particular $|v| \approx |v_0|$. The set of points $v_0 \in \R^d$ so that $\dgs(v,v_0) < R$ has volume $\approx (1+|v|)^{-1}$. Thus, $W_1 \lesssim |v| (1+|v|)^{-1} \leq 1$, and \eqref{e:upper} follows. As far as \eqref{e:lower} is concerned, if $\dgs (v,v') < \rho$ for $\rho$ small, then the set of $v_0$ so that $v,v' \in E_1(v_0)$ will be indeed of volume $\approx (1+|v|)^{-1}$. If $v \notin B_2$ or $v' \notin B_2$, we ensure that at least half of this set lies outside $B_2$. Note that since $|v_0| \approx |v|>2$ (or $|v'|>2$), we have $|v_0| / (1+|v|) \approx 1$ and \eqref{e:lower} follows.

With \eqref{e:upper} and \eqref{e:lower} at hand, we can deduce from \eqref{e:recall} that
\begin{multline*}
  \iint_{\{ \dgs (\bar v,\bar v')<R\} } (g(\bar v')- g(\bar v))^2 K_f(\bar v, \bar v') \dd \bar v' \dd \bar v  \\
  \gtrsim \iint_{\{ \dgs(\bar v, \bar v') < \rho \}} (g(\bar v')-  g(\bar v))^2 \un_{\{|\bar v|>2 \text{ or } |\bar v'| >2\}}
  \frac{(1+|\bar v + \bar v'|)^{1+\gamma+2s}}{\dgs (\bar v,\bar v')^{d+2s}} \dd \bar v' \dd \bar v.
\end{multline*}

In order to deal with small velocities, the change of variables is not needed: we apply \eqref{e:local_coercivity} (scaled to $B_4$) directly to $K_f$ and get, 
  \begin{align*}
    \iint_{B_{4} \times B_{4}} (g(\bar v')- g(\bar v))^2 K_f(\bar v, \bar v') \dd \bar v' \dd \bar v  
    & \gtrsim \iint_{B_2 \times B_2} (g(\bar v')-  g(\bar v))^2 |\bar v-\bar v'|^{-d-2s} \dd \bar v' \dd \bar v \\
    & \gtrsim \iint_{\{\dgs (\bar v,\bar v') < \rho\}} (g(\bar v')-  g(\bar v))^2 \un_{\{\bar v,\bar v' \in B_2\} } \frac{(1+|\bar v+ \bar v'|)^{\gamma + 2s +1}}{\dgs(\bar v,\bar v')^{d+2s}} \dd \bar v' \dd \bar v.
  \end{align*}

We conclude the proof by combining the estimate for large velocities with the one for small velocities. 
\end{proof}

We can now prove Proposition~\ref{p:GS}. 

\begin{proof}[Proof of Proposition~\ref{p:GS}]
  From Corollary~\ref{c:ndc}, we know that $\bar K_f$ satisfies \eqref{e:local_coercivity} with a $\lambda>0$ independent of $v_0$.

We use again the decomposition \eqref{e:Q-bis} from \cite{villani-book,silvestre2016new}  After straight-forward arithmetic manipulations, we get
\begin{align*} 
 -\int_{\R^d} \Q (f,g) g \dv &= \frac 12 \iint_{\R^d \times \R^d} (g( v')- g( v))^2 K_f( v,  v') \dd  v' \dd  v \\
& \qquad - \frac 12  \int_{\R^d} g(v)^2 \left( \int_{\R^d} (K_f(v,v') - K_f(v',v)) \dd v' \right)  \dd v, \\
& \qquad - \int_{\R^d} (f \ast |\cdot|^\gamma) g(v)^2 \dd v, \\
& = I_1 - I_2 - I_3.
\end{align*}

We use Lemma~\ref{l:GS} to estimate the first term. We use \cite[Lemma 3.6]{imbert2016weak} to estimate the second term. In fact, the classical cancellation lemma from \cite{alexandre2000entropy} (see also  \cite[Lemmas 5.1 and 5.2]{silvestre2016new}) tells us that the second and third terms are
identical. Thus, using \eqref{e:extra-int} if $\gamma <0$, 
\begin{align*}
I_1 &\gtrsim\iint_{\dgs(v,v')<\rho} (g( v') - g( v))^2 \frac{(1+| v+ v'|)^{1+\gamma+2s}} {\dgs ( v, v')^{d+2s}} \dd  v' \dd  v, \\
I_2 &= I_3 = \int_{\R^d} g(v)^2 \left( f(v+w) |w|^\gamma \dd w \right) \dd v \leq \begin{cases}
C (1+|v|)^\gamma \int_{\R^d}g(v)^2 \dv & \text{if $\gamma\geq 0$, with $C=C(M_0,E_0)$},\\
C_\gamma \int_{\R^d}g(v)^2 \dv & \text{if $\gamma < 0$}. \end{cases}
\end{align*}
The proof is now complete. 
\end{proof}

\begin{remark}
  It is possible to justify that the universal constants $R>\rho$ can be  chosen arbitrarily using a covering argument as in  \cite[Section 5.2]{chaker2019coercivity}.  The norm $\mathcal N^s_\gamma$ in   \cite{gressman2011global,gressmanstrainBETTERpaper} is defined with $\rho = 1$.
\end{remark}

\begin{remark} \label{r:GS} The coercivity estimate from
  \cite{gressman2011global} and the coercivity
  estimates from \cite{gressmanstrainBETTERpaper} and in Proposition~\ref{p:GS} involve different operators. Our proposition, as well as the estimate in \cite{gressmanstrainBETTERpaper}, is for the
  linear operator
\begin{equation} \label{e:ourLinear}
 L(g) = -\mathcal \Q(f,g),
\end{equation}
for any given profile $f$ for which the mass, energy and entropy are bounded above, and the mass is bounded below. The estimate in \cite{gressman2011global} is for the linearized Boltzmann operator
\begin{equation} \label{e:theirLinear}
 L(g) = - M^{-1/2} \mathcal \Q(M, M^{1/2} g) - M^{-1/2} \mathcal \Q(M^{1/2} g, M),
\end{equation}
where $M$ is a Maxwellian profile.

The linear operators \eqref{e:ourLinear} and \eqref{e:theirLinear} are different. The operator \eqref{e:ourLinear} is useful to study (so far conditional) regularity estimates for generic solutions of the Boltzmann equation. The operator \eqref{e:theirLinear} is useful to study the stability of the equation for small perturbations around a Maxwellian.

Coercivity estimates from \cite{gressmanstrainBETTERpaper} and from  Proposition~\ref{p:GS} are proved under slightly different sets of  assumptions. It is assumed in \cite{gressmanstrainBETTERpaper} that $f$ satisfies for  all $v \in \R^d$ and $a \in [\gamma,\gamma +2s]$,
  \begin{equation}\label{e:GS-cond}
 \int_{\R^d} f(w) |w-v|^a (1+|w|)^i \dd w \lesssim (1+|v|)^a 
\end{equation}
with $i=1$ if $s < \frac12$ and $i=2$ for $s \ge \frac12$.  For $\gamma<0$, \eqref{e:GS-cond} implies \eqref{e:extra-int} by choosing $a=\gamma$. Notice that \eqref{e:GS-cond} imply a control of moments  of order $2+\gamma +2s$ if $s \ge 1/2$ which can be larger than $2$.

Note also that Assumption L in \cite{gressmanstrainBETTERpaper} is slightly more general than the upper bound on the entropy in Assumption~\ref{a:hydro-assumption}.
\end{remark}

\bibliographystyle{plain}
\bibliography{bootstrap}

\begin{thebibliography}{10}

\bibitem{alexandre2000entropy}
Radjesvarane Alexandre, Laurent Desvillettes, C{\'e}dric Villani, and Bernt
  Wennberg.
\newblock Entropy dissipation and long-range interactions.
\newblock {\em Archive for rational mechanics and analysis}, 152(4):327--355,
  2000.

\bibitem{aes2005}
Radjesvarane Alexandre and Mouhamad El~Safadi.
\newblock Littlewood-{P}aley theory and regularity issues in {B}oltzmann
  homogeneous equations. {I}. {N}on-cutoff case and {M}axwellian molecules.
\newblock {\em Math. Models Methods Appl. Sci.}, 15(6):907--920, 2005.

\bibitem{aes2009}
Radjesvarane Alexandre and Mouhamad Elsafadi.
\newblock Littlewood-{P}aley theory and regularity issues in {B}oltzmann
  homogeneous equations. {II}. {N}on cutoff case and non {M}axwellian
  molecules.
\newblock {\em Discrete Contin. Dyn. Syst.}, 24(1):1--11, 2009.

\bibitem{amuxyCRAS2009}
Radjesvarane Alexandre, Yoshinore Morimoto, Seiji Ukai, Chao-Jiang Xu, and Tong
  Yang.
\newblock Regularity of solutions for the {B}oltzmann equation without angular
  cutoff.
\newblock {\em C. R. Math. Acad. Sci. Paris}, 347(13-14):747--752, 2009.

\bibitem{amuxy2010}
Radjesvarane Alexandre, Yoshinori Morimoto, Seiji Ukai, Chao-Jiang Xu, and Tong
  Yang.
\newblock Regularizing effect and local existence for the non-cutoff
  {B}oltzmann equation.
\newblock {\em Archive for Rational Mechanics and Analysis}, 198(1):39--123,
  2010.

\bibitem{av2002}
Radjesvarane Alexandre and C{\'e}dric Villani.
\newblock On the {B}oltzmann equation for long-range interactions.
\newblock {\em Communications on pure and applied mathematics}, 55(1):30--70,
  2002.

\bibitem{armstrong2019variational}
Scott Armstrong and Jean-Christophe Mourrat.
\newblock Variational methods for the kinetic fokker-planck equation.
\newblock {\em arXiv preprint arXiv:1902.04037}, 2019.

\bibitem{bardos1991}
Claude Bardos, Fran{\c{c}}ois Golse, and David Levermore.
\newblock Fluid dynamic limits of kinetic equations. {I}. {F}ormal derivations.
\newblock {\em J. Statist. Phys.}, 63(1-2):323--344, 1991.

\bibitem{caffarelli2009regularity}
Luis Caffarelli and Luis Silvestre.
\newblock Regularity theory for fully nonlinear integro-differential equations.
\newblock {\em Communications on Pure and Applied Mathematics: A Journal Issued
  by the Courant Institute of Mathematical Sciences}, 62(5):597--638, 2009.

\bibitem{cc}
Luis~A. Caffarelli and Xavier Cabr\'{e}.
\newblock {\em Fully nonlinear elliptic equations}, volume~43 of {\em American
  Mathematical Society Colloquium Publications}.
\newblock American Mathematical Society, Providence, RI, 1995.

\bibitem{cameron2017global}
Stephen {Cameron}, Luis {Silvestre}, and Stanley {Snelson}.
\newblock {Global a priori estimates for the inhomogeneous Landau equation with
  moderately soft potentials}.
\newblock {\em {Ann. Inst. Henri Poincar\'e, Anal. Non Lin\'eaire}},
  35(3):625--642, 2018.

\bibitem{cedric2012theoreme}
Villani C{\'e}dric.
\newblock Th{\'e}or{\`e}me vivant, 2012.

\bibitem{chaker2019coercivity}
Jamil {Chaker} and Luis {Silvestre}.
\newblock {Coercivity estimates for integro-differential operators}.
\newblock {\em {Calc. Var. Partial Differ. Equ.}}, 59(4):20, 2020.
\newblock Id/No 106.

\bibitem{chen2011smoothing}
Yemin Chen and Lingbing He.
\newblock Smoothing estimates for {B}oltzmann equation with full-range
  interactions: Spatially homogeneous case.
\newblock {\em Archive for rational mechanics and analysis}, 201(2):501--548,
  2011.

\bibitem{chen-he-2012}
Yemin Chen and Lingbing He.
\newblock Smoothing estimates for {B}oltzmann equation with full-range
  interactions: Spatially inhomogeneous case.
\newblock {\em Archive for Rational Mechanics and Analysis}, 203(2):343--377,
  2012.

\bibitem{desvillettes-villani}
L.~Desvillettes and C.~Villani.
\newblock On the trend to global equilibrium for spatially inhomogeneous
  kinetic systems: the {B}oltzmann equation.
\newblock {\em Invent. Math.}, 159(2):245--316, 2005.

\bibitem{dw2004}
Laurent Desvillettes and Bernt Wennberg.
\newblock Smoothness of the solution of the spatially homogeneous {B}oltzmann
  equation without cutoff.
\newblock {\em Comm. Partial Differential Equations}, 29(1-2):133--155, 2004.

\bibitem{dyda2011comparability}
Bart{\l}omiej Dyda and Moritz Kassmann.
\newblock Comparability and regularity estimates for symmetric nonlocal
  dirichlet forms.
\newblock {\em arXiv preprint arXiv:1109.6812}, 2011.

\bibitem{dyda2015regularity}
Bart{\l}omiej {Dyda} and Moritz {Kassmann}.
\newblock {Regularity estimates for elliptic nonlocal operators}.
\newblock {\em {Anal. PDE}}, 13(2):317--370, 2020.

\bibitem{golse2016harnack}
Fran\c{c}ois {Golse}, Cyril {Imbert}, and Alexis {Vasseur}.
\newblock {Harnack inequality for kinetic Fokker-Planck equations with rough
  coefficients and application to the Landau equation}.
\newblock {\em {Ann. Sc. Norm. Super. Pisa, Cl. Sci. (5)}}, 19(1):253--295,
  2019.

\bibitem{gressman2011global}
Philip Gressman and Robert Strain.
\newblock {Global classical solutions of the Boltzmann equation without angular
  cut-off}.
\newblock {\em Journal of the American Mathematical Society}, 24(3):771--847,
  2011.

\bibitem{gressmanstrainBETTERpaper}
Philip~T. Gressman and Robert~M. Strain.
\newblock Sharp anisotropic estimates for the {B}oltzmann collision operator
  and its entropy production.
\newblock {\em Adv. Math.}, 227(6):2349--2384, 2011.

\bibitem{henderson2017c}
Christopher {Henderson} and Stanley {Snelson}.
\newblock {\(C^\infty\) smoothing for weak solutions of the inhomogeneous
  Landau equation}.
\newblock {\em {Arch. Ration. Mech. Anal.}}, 236(1):113--143, 2020.

\bibitem{henderson2017local}
Christopher {Henderson}, Stanley {Snelson}, and Andrei {Tarfulea}.
\newblock {Local existence, lower mass bounds, and a new continuation criterion
  for the Landau equation}.
\newblock {\em {J. Differ. Equations}}, 266(2-3):1536--1577, 2019.

\bibitem{HST1}
Christopher Henderson, Stanley Snelson, and Andrei Tarfulea.
\newblock Local well-posedness of the {B}oltzmann equation with polynomially
  decaying initial data.
\newblock {\em Kinet. Relat. Models}, 13(4):837--867, 2020.

\bibitem{HST2}
Christopher Henderson, Stanley Snelson, and Andrei Tarfulea.
\newblock Self-generating lower bounds and continuation for the {B}oltzmann
  equation.
\newblock {\em Calc. Var. Partial Differential Equations}, 59(6):Paper No. 191,
  13, 2020.

\bibitem{huo2008}
Zhaohui Huo, Yoshinori Morimoto, Seiji Ukai, and Tong Yang.
\newblock Regularity of solutions for spatially homogeneous {B}oltzmann
  equation without angular cutoff.
\newblock {\em Kinet. Relat. Models}, 1(3):453--489, 2008.

\bibitem{imbert2018decay}
Cyril {Imbert}, Cl\'ement {Mouhot}, and Luis {Silvestre}.
\newblock {D\'ecroissance aux grandes vitesses pour LES solutions de
  L'\'Equation de Boltzmann sans troncature angulaire}.
\newblock {\em {J. \'Ec. Polytech., Math.}}, 7:143--184, 2020.

\bibitem{imbert2019gaussian}
Cyril {Imbert}, Cl\'ement {Mouhot}, and Luis {Silvestre}.
\newblock {Gaussian lower bounds for the Boltzmann equation without cutoff}.
\newblock {\em {SIAM J. Math. Anal.}}, 52(3):2930--2944, 2020.

\bibitem{schauder}
Cyril Imbert and Luis Silvestre.
\newblock The {S}chauder estimate for kinetic integral equations.
\newblock {\em arXiv preprint arXiv:1812.11870}, 2018.

\bibitem{imbert2016weak}
Cyril {Imbert} and Luis {Silvestre}.
\newblock {The weak Harnack inequality for the Boltzmann equation without
  cut-off}.
\newblock {\em {J. Eur. Math. Soc. (JEMS)}}, 22(2):507--592, 2020.

\bibitem{russell}
Moritz Kassmann and Russell~W. Schwab.
\newblock Regularity results for nonlocal parabolic equations.
\newblock {\em Riv. Math. Univ. Parma (N.S.)}, 5(1):183--212, 2014.

\bibitem{merle2019smooth}
Frank Merle, Pierre Raphael, Igor Rodnianski, and Jeremie Szeftel.
\newblock {On smooth self similar solutions to the compressible Euler
  equations}.
\newblock arXiv:1912.10998, 2019.

\bibitem{merle2019implosion}
Frank Merle, Pierre Raphael, Igor Rodnianski, and Jeremie Szeftel.
\newblock On the implosion of a three dimensional compressible fluid.
\newblock arXiv:1912.11009, 2019.

\bibitem{morimoto2009regularity}
Yoshinori Morimoto, Seiji Ukai, Chao-Jiang Xu, and Tong Yang.
\newblock Regularity of solutions to the spatially homogeneous {B}oltzmann
  equation without angular cutoff.
\newblock {\em Discrete Contin. Dyn. Syst.}, 24(1):187--212, 2009.

\bibitem{morimoto2015local}
Yoshinori Morimoto and Tong Yang.
\newblock Local existence of polynomial decay solutions to the boltzmann
  equation for soft potentials.
\newblock {\em Analysis and Applications}, 13(06):663--683, 2015.

\bibitem{sergio2004recent}
Polidoro Sergio.
\newblock Recent results on {K}olmogorov equations and applications.
\newblock In {\em Workshop on Second Order Subelliptic Equations and
  Applications}, volume~3, pages 129--143. GRAFICOM Edizioni, 2004.

\bibitem{sideris}
Thomas~C. Sideris.
\newblock Formation of singularities in three-dimensional compressible fluids.
\newblock {\em Comm. Math. Phys.}, 101(4):475--485, 1985.

\bibitem{silvestre2016new}
Luis Silvestre.
\newblock {A new regularization mechanism for the Boltzmann equation without
  cut-off}.
\newblock {\em Communications in Mathematical Physics}, 348(1):69--100, 2016.

\bibitem{luis}
Luis Silvestre.
\newblock A new regularization mechanism for the {B}oltzmann equation without
  cut-off.
\newblock {\em Comm. Math. Phys.}, 348(1):69--100, 2016.

\bibitem{snelson2018inhomogeneous}
Stanley Snelson.
\newblock Gaussian bounds for the inhomogeneous {L}andau equation with hard
  potentials.
\newblock {\em SIAM J. Math. Anal.}, 52(2):2081--2097, 2020.

\bibitem{villani-book}
C\'edric Villani.
\newblock A review of mathematical topics in collisional kinetic theory.
\newblock In {\em Handbook of mathematical fluid dynamics, {V}ol. {I}}, pages
  71--305. North-Holland, Amsterdam, 2002.

\end{thebibliography}
\end{document}